\documentclass[11pt]{article}

\input xy

\xyoption{all}

\usepackage{amssymb,amsbsy,amsthm,amsmath,graphicx,epsfig}
\usepackage{mathabx,calrsfs,times}

\usepackage{microtype}

\usepackage{sectsty}

\sectionfont{\normalsize\center}

\subsectionfont{\normalsize}

\newtheorem{thm}{Theorem}[section]
\newtheorem{prop}[thm]{Proposition}
\newtheorem{lem}[thm]{Lemma}
\newtheorem{cor}[thm]{Corollary}
\newtheorem{dfn}[thm]{Definition}

\newtheorem{ques}[thm]{Question}
\newtheorem{conj}[thm]{Conjecture}

\newtheorem{ex}[thm]{Example}
\newtheorem{rmk}[thm]{Remark}
\newtheorem{rmks}[thm]{Remarks}

\renewcommand{\bf}[1]{\mathbf{#1}}
\renewcommand{\rm}[1]{\mathrm{#1}}
\renewcommand{\cal}[1]{\mathcal{#1}}

\newcommand{\bbN}{\mathbb{N}}

\newcommand{\bbR}{\mathbb{R}}
\newcommand{\bbT}{\mathbb{T}}
\newcommand{\bbZ}{\mathbb{Z}}

\newcommand{\bfX}{\mathbf{X}}
\newcommand{\bfY}{\mathbf{Y}}
\newcommand{\bfZ}{\mathbf{Z}}

\newcommand{\sfE}{\mathsf{E}}
\newcommand{\sfP}{\mathsf{P}}
\newcommand{\sfW}{\mathsf{W}}

\newcommand{\rmH}{\mathrm{H}}
\newcommand{\rmI}{\mathrm{I}}

\renewcommand{\d}{\mathrm{d}}
\newcommand{\rme}{\mathrm{e}}
\newcommand{\rmh}{\mathrm{h}}

\newcommand{\A}{\mathcal{A}}

\newcommand{\F}{\mathcal{F}}

\newcommand{\J}{\mathcal{J}}
\newcommand{\calL}{\mathcal{L}}

\renewcommand{\P}{\mathcal{P}}
\newcommand{\Q}{\mathcal{Q}}
\newcommand{\R}{\mathcal{R}}
\newcommand{\calS}{\mathcal{S}}
\newcommand{\T}{\mathcal{T}}

\newcommand{\Z}{\mathcal{Z}}

\newcommand{\frM}{\mathfrak{M}}

\renewcommand{\S}{\Sigma}

\renewcommand{\a}{\alpha}
\renewcommand{\b}{\beta}
\newcommand{\eps}{\varepsilon}
\newcommand{\g}{\gamma}
\renewcommand{\k}{\kappa}
\renewcommand{\l}{\lambda}
\newcommand{\w}{\omega}
\newcommand{\s}{\sigma}
\renewcommand{\phi}{\varphi}

\newcommand{\dom}{\mathrm{dom}}

\newcommand{\ol}[1]{\overline{#1}}

\renewcommand{\qed}{\nolinebreak\hspace{\stretch{1}}$\Box$\vspace{7pt}}
\newcommand{\fin}{\nolinebreak\hspace{\stretch{1}}$\lhd$}

\newcommand{\mcup}{\mbox{$\bigcup$}}

\newcommand{\actson}{\curvearrowright}

\renewcommand{\t}[1]{\widetilde{#1}}
\renewcommand{\to}{\longrightarrow}

\newcommand{\bipack}{\rm{bipack}}
\newcommand{\BIPACK}{\rm{BIPACK}}
\newcommand{\bicov}{\rm{bicov}}
\newcommand{\BICOV}{\rm{BICOV}}
\newcommand{\cov}{\rm{cov}}
\newcommand{\rmD}{\mathrm{D}_{\mathrm{KL}}}
\newcommand{\DCS}{\mathrm{DCS}}
\newcommand{\DCF}{\mathrm{DCF}}
\newcommand{\DCM}{\mathrm{DCM}}
\newcommand{\traj}{\mathrm{traj}}

\begin{document}

\title{\textbf{\Large{Scenery entropy as an invariant of RWRS processes}}}
\author{Tim Austin\thanks{Research supported by a fellowship from the Clay Mathematics Institute}\\ \\ \small{Courant Institute, New York University,}\\ \small{New York, NY 10012, U.S.A.}\\ \small{\texttt{tim@cims.nyu.edu}}}

%

\date{}

\maketitle

\begin{abstract}
Probabilistic models of random walks in random sceneries give rise to examples of probability-preserving dynamical systems. A point in the state spaces consists of a walk-trajectory and a scenery, and its `motion' corresponds to shifting the time-origin.

These models were proposed as natural examples of non-Bernoulli K-automorphisms by Adler, Ornstein and Weiss.  This was proved in a famous analysis by Kalikow using Ornstein's Very Weak Bernoulli characterization of Bernoulli processes.  Since then, various authors have generalized this construction to give other examples, including some smooth examples due to Katok and Rudolph.

However, the methods used to prove non-Bernoullicity do not obviously show that these examples are distinct from one another.  This paper introduces a new isomorphism-invariant of probability-preserving systems, and shows that in a large class of the above examples it essentially captures the Kolmogorov-Sinai entropy of the scenery process alone.  As a result, constructions that use different scenery-entropies give continuum-many non-isomorphic examples.  Conditionally on an invariance principle for certain local times, these include a continuum of distinct smooth non-Bernoulli K-automorphisms on a fixed compact manifold.

\end{abstract}

\setcounter{tocdepth}{1}
\tableofcontents

\section{Introduction}

\subsection{Historical overview}

A simple random walk in a random scenery may be described by the following data:
\begin{itemize}
\item The first ingredient is the space $\{\pm 1\}^\bbZ$ with the product measure $\nu_{1/2}^{\otimes \bbZ}$, where $\nu_{1/2} = \frac{1}{2}(\delta_1 + \delta_{-1})$.  This gives the space of possible step-sequences of the walk, from the infinite past and to the infinite future, with the usual i.i.d. law for those steps.
\item The second ingredient is a probability-preserving system $(C^\bbZ,\mu,S)$, where $C$ is a finite set of `colours', $S$ is the leftward coordinate-shift on $C^\bbZ$, and $\mu$ is an $S$-invariant probability on $C^\bbZ$.  This $\mu$ is the law of a random scenery, which decorates every point in $\bbZ$ with a colour from $C$.
\end{itemize}

Elements $(y_n)_n \in \{\pm 1\}^\bbZ$ may be identified bijectively with paths $\bbZ \to \bbZ$ which pass through the origin and whose increments are all $-1$ or $1$, by identifying $y_n$ with the increment from $n$ to $n+1$.  This converts $(y_n)_n$ into the trajectory taken by the walker, as seen from her current location.  The coordinate shift on $(y_n)_n$ acts on this picture by shifting the origin of time, but retaining the feature that the trajectory passes through the origin: that is, we always view the trajectory from the walker's current location.

This description may naturally be combined with the walker's view of the scenery to form a probability-preserving system $(Z,\rho,R)$ which captures the whole of the above picture.  First let
\[(Z,\rho) := (\{\pm 1\}^\bbZ\times C^\bbZ,\nu_{1/2}^{\otimes \bbZ}\otimes \mu).\]
For the dynamics, think of $((y_n)_n,(x_m)_m) \in Z$ as a pair
\[\big(\hbox{trajectory through origin},\ \ \hbox{scenery viewed by walker at origin}\big),\]
and let $R$ be the transformation which shifts time one step forward, but preserves the feature that the walker's location is the origin.  In notation, this is
\[R((y_n)_n,(x_m)_m) = ((y_{n+1})_n,(x_{m+y_0})_m).\]

One checks easily that $\rho$ is $R$-invariant.  This system is called the \textbf{random walk in random scenery $\mu$}, and will be denoted $\rm{RWRS}_\mu$.

The systems $\rm{RWRS}_\mu$ are important in ergodic theory because they are simple and natural examples of an abstract phenomenon: for many possible choices of $\mu$, they are K-automorphisms but not Bernoulli systems.  This was conjectured by Adler, Ornstein and Weiss, who observed that the K-property is fairly easy to prove (it also holds for more complicated random walks, as shown by Meilijson~\cite{Meil74}).  However, non-Bernoullicity was not proved at that time, and it was recorded as an open problem in~\cite{Weiss72}.  This problem was solved by Kalikow (who refers to this as the `$[T,T^{-1}]$ system', as have many more recent authors).

\vspace{7pt}

\noindent\textbf{Theorem~(\cite{Kal82})}\quad \emph{The process $\rm{RWRS}_{\nu_{1/2}^{\otimes \bbZ}}$ is not Bernoulli. \qed}

\vspace{7pt}

The heart of Kalikow's work is to show that $\rm{RWRS}_{\nu_{1/2}^{\otimes \bbZ}}$ does not have the Very Weak Bernoulli property, one of the equivalent characterizations of Bernoullicity involved in Ornstein's famous solution of the Bernoulli Isomorphism Problem.

Other non-Bernoulli K-automorphisms were constructed before Kalikow's work: the first in~\cite{Orn73}, and then a continuum family of them in~\cite{OrnShi73}.  However, those examples were all obtained by cutting and stacking for this deliberate purpose.

For any abstract ergodic-theoretic phenomenon, it is of additional interest to find examples that arise naturally from other parts of mathematics (see Section 14 of Thouvenot's essay~\cite{Tho02} for further discussion).  For non-Bernoulli K-automorphisms, more progress in this direction was made by Feldman in~\cite{Feld76}.  He exhibited some examples in the form of skew products, somewhat resembling $\rm{RWRS}$s, as an application of his new notion of loose Bernoullicity.  However, these examples still required the a priori cut-and-stack construction of a non-loosely-Bernoulli automorphism.  A smooth version of this construction was then carried out by Katok in~\cite{Kat80}, who points out that Ratner's formidable work~\cite{Rat79} provides natural, geometric non-loosely-Bernoulli transformations for ingredients.  Nevertheless, following~\cite{Kal82}, $\rm{RWRS}$s remain the principal `natural' examples of non-Bernoulli K-automorphisms (and, indeed, Kalikow actually shows that $\rm{RWRS}_{\nu_{1/2}^{\otimes \bbZ}}$ is not even loosely Bernoulli, although we will ignore this strengthening here).  Thouvenot gives an overview of these developments in~\cite[Section 12]{Tho02}, as well as more complete references.

Since~\cite{Kal82}, Kalikow's argument has been generalized in various directions.  Smooth examples of skew products analogous to $\rm{RWRS}$s are shown to satisfy the same conclusion in~\cite{Rud88}, following a suggestion in~\cite{Kat80}.  In these examples, the trajectories of simple random walk are replaced by the sequences of ergodic sums of a smooth function over an Anosov diffeomorphism.  More recently,~\cite{denHSte97} analysed quite general examples of random walks in $\bbZ^d$, $d \geq 1$, with sceneries given by shift-invariant measures on $C^{\bbZ^d}$, obtaining new non-Bernoulli-K examples when $d = 2$.

(Several more recent works have also explored necessary and sufficient conditions for a $\rm{RWRS}$ with its obvious generating partition to be Weakly Bernoulli.  This is a more restrictive question than Very Weak Bernoullicity, and can be approached using simpler methods than Kalikow's.  However, it is not an invariant of measure-theoretic isomorphism. We will not discuss it further in this paper.)

Having shown that some $\rm{RWRS}$s are not Bernoulli, it is natural to ask when they are isomorphic to \emph{each other}.  Kalikow's method does not seem to resolve this question directly, even for i.i.d. sceneries.  First, note that if $(C^\bbZ,\mu,S) \cong (D^\bbZ,\theta,S)$, then $\rm{RWRS}_\mu \cong \rm{RWRS}_\theta$, since the former isomorphism my simply be applied to the second coordinate of $\{\pm 1\}^\bbZ\times C^\bbZ$.  However, the reverse implication can fail.  For instance, if $(C^\bbZ,\mu,S)$ is a coding of an ergodic circle-rotation, then $\rm{RWRS}_\mu$ is an isometric extension of the Bernoulli shift $(\{\pm 1\}^\bbZ,\nu^{\otimes \bbZ}_{1/2},S)$, and in this case it is still Bernoulli of the same entropy~\cite{AdlShi72,AdlShi74}, so
\[\rm{RWRS}_{\rm{ergod.\ rotn.}} \cong \rm{RWRS}_{\rm{trivial\ system}}.\]

On the other hand, Kalikow's result itself shows that the isomorphism class of $\rm{RWRS}_\nu$ does remember something (necessarily isomorphism-invariant) about the scenery process.

The purpose of this paper is to show that this includes the entropy of the scenery process.  This result seems to have been expected for some time; I learnt of this expectation from J.-P. Thouvenot, but it is also hinted at in Vershik's paper~\cite{Ver00} in connection with his notion of `secondary entropy'.

\subsection{Statement of the main results}

The main result below applies to a generalization of RWRS processes constructed from certain`cocycle random walks'.  To formulate it, suppose that $\bfY = (Y,\nu,S)$ is a probability-preserving system, that $\s:Y\to \bbR$ is measurable, and that $\bfX = (X,\mu,T)$ is a jointly measurable and probability-preserving action of $\bbR$.   Then the \textbf{generalized RWRS system with base $(\bfY,\s)$ and fibre $\bfX$} is the resulting skew-product transformation on $(Y\times X,\nu\otimes \mu)$:
\[(S\ltimes_\s T)(y,x) := (Sy,T^{\s(y)}x).\]
This system will be denoted $\bfY \ltimes_\s \bfX$.

For example, suppose that
\[\bfY := (\{\pm 1\}^\bbZ,\nu_{1/2}^{\otimes \bbZ},S), \quad  \s((y_n)_{n\in \bbZ}) := y_0,\]
and that $(X,\mu,T)$ is a continuous-time flow such that $(X,\mu,T^1) \cong (C^\bbZ,\mu',S)$. Then
\[\bfY \ltimes_\s \bfX \cong \rm{RWRS}_{\mu'}.\]

We shall prove that for certain fixed choices of $(\bfY,\s)$, the entropy of $\bfX$ is an isomorphism-invariant of the whole generalized RWRS system $\bfY \ltimes_\s \bfX$.  The argument will assume some quite delicate conditions on the system $\bfY$ and cocycle $\s$.  In the first place:
\begin{quote}
$Y \subseteq A^\bbZ$ is a subshift of finite type; $S$ is the coordinate-shift; $\nu$ is a Gibbs measure for a H\"older continuous potential on $Y$; and $\s$ is a H\"older continuous non-coboundary with $\int \s\,\d\nu = 0$.
\end{quote}

We refer to these assumptions collectively as $(\bfY,\s)$ being a `well-distributed pair'.  The proofs below will make use of this assumption in many different ways.  It could probably be replaced with a longer list of more bespoke assumptions, but it seems simpler to restrict to the above class. Many of the consequences of this assumption that we need assert various kinds of resemblance to Brownian motion at all sufficiently large scales, with some explicit rate on the convergence.  This is in a similar spirit to the `asymptotically Brownian' condition required by Rudolph in~\cite{Rud88}, but technically different.


In addition to the above, we will need to assume that our well-distributed pair satisfies an `Enhanced Invariance Principle', which describes the asymptotic law of the cocycle $\s$ and also its occupation measures over long time-scales.  I believe that this principle holds for \emph{all} well-distributed pairs, and can therefore be dropped from explicit mention in Theorem A below.  However, it is not yet available in the literature in that generality.  It is available for some more specific examples, and at time of writing I understand that Michael Bromberg is working on the general case.  This principle will be formulated carefully in Subsection~\ref{subs:EIP}.

\vspace{7pt}


\noindent\textbf{Theorem A}\quad \emph{Suppose that $\bfY = (Y,\nu,S)$ and $\s:Y\to \bbR$ form a well-distributed pair which satisfies the Enhanced Invariance Principle.  If $\bfX_i$, $i=1,2$ are two flows such that there exists a factor map $\bfY \ltimes_\s \bfX_1 \to \bfY \ltimes_\s \bfX_2$, then $\rmh(\bfX_1) \geq \rmh(\bfX_2)$.}

\vspace{7pt}

Importantly, this allows factor maps that do not act as the identity on the base system $\bfY$.  We must therefore find a way to extract the entropy of the scenery from $\bfY \ltimes_\s \bfX$ as an abstract p.-p. system, without assuming knowledge of the distinguished factor map $\bfY \ltimes_\s \bfX\to \bfY$.

It is important that one fix the choice of $(\bfY,\s)$.  Indeed, the invariant that we shall actually produce takes the form $f(\s)\rmh(\bfX)$, where $f$ is some function of $\s$ which is homogeneous of order $1$.  It is easy to see that if one replaces $\s$ with $2\s$ and $\bfX$ with its slowdown by a factor of $2$, then the resulting generalized RWRS systems are isomorphic, so this fixing of $\s$ is essential.

Aaronson's recent work~\cite{Aar12} implies a special case of Theorem A in which the factor map is assumed to respect the coordinate factor map to $\bfY$.  Applied to our setting, Corollary 5 of that paper shows that if $\bfY$ and $\s$ are the process and cocycle of classical simple random walk, then a relative factor map
\begin{center}
$\phantom{i}$\xymatrix{ \bfY \ltimes_\s \bfX_1 \ar_{\rm{coord.\ proj.}}[dr] \ar[rr] && \bfY \ltimes_\s \bfX_2 \ar^{\rm{coord.\ proj.}}[dl]\\
& \bfY }
\end{center}
can exist only if $\rmh(\bfX_1) \geq \rmh(\bfX_2)$.  (Aaronson also handles the case of other stable random walks, which we leave aside here.)  For a canonical choice of generating partition $\R$ for these systems $\bfX\ltimes_\s \bfY$, this result follows from a calculation of the distributions of the relative complexities of $(\R,N)$-names over the base system $\bfY$, regarded as random variable on the probability space $(Y,\nu)$.  Our work below will turn out to need many of the same calculations as Aaronson's.  However, relative complexities give an invariant only of relative isomorphism: they do not serve to control arbitrary factor maps $\bfY \ltimes_\s \bfX_1 \to \bfY \ltimes_\s \bfX_2$.

Theorem A also has precedents in the study of non-invertible RWRS processes, for which a point in the state space records only the future trajectory of the walk.  The analog of Theorem A with one-sided simple random walk in the base was proved by Heicklen, Hoffman and Rudolph in~\cite{HeiHofRud00}, and a generalization to some other skew products, including some smooth examples, was given by Ball in~\cite{Bal03}.  In some ways the steps in our work below reflect those papers, except that they make essential use of some extra isomorphism-invariant structure of a non-invertible transformation: the decreasing filtration of pre-images of the $\s$-algebra.  This idea goes back to work of Vershik around 1970: see~\cite{Ver94}.



Conditionally on the Enhanced Invariance Principle, Theorem A also covers certain smooth analogs of RWRSs on compact manifolds, using appropriate codings from Gibbs measures on subshifts.  Arguably, these form the most `natural' among all kinds of example in ergodic theory (see again~\cite[Section 14]{Tho02}, which includes a discussion of some of these smooth RWRS-like examples, as studied in~\cite{Rud88}).  The first smooth non-Bernoulli K-automorphisms were constructed in~\cite{Kat80}, but again with a more complicated description.

For instance, let $A:\bbT^2\to \bbT^2$ be a hyperbolic toral automorphism and $m$ be the Haar probability measure, let $g = (g^t)_{t \in \bbR}$ be an Anosov flow on a compact manifold $M$ that preserves the Riemannian volume-form $\mu$, and on the space $(\bbT^2\times M,m\otimes \mu)$ consider the skew-product transformations
\[T_r(x_1,x_2,p) := (A(x_1,x_2),g^{r\sin x_1}p).\]
Kalikow's argument itself was extended to cover examples such as these in~\cite{Rud88}.


These can clearly be written as skew products of the form $\bfY \ltimes_\s \bfX$ with $\bfY = (\bbT^2,m,A)$ and $\bfX = (M,\mu,g)$.  This $\bfY$ has a coding given by an a.e. one-one H\"older function $F:(Y,\nu,S)\to (\bbT^2,m,A)$ for some SFT $Y$ and H\"older-potential Gibbs measure $\nu$~(see~\cite{Bow08}).  Therefore Theorem A applies to these examples provided one knows the Enhanced Invariance Principle.  Since $A\ltimes_{r\s} g = A\ltimes_\s g^{(r)}$, where $g^{(r)}$ is the speedup of $g$ by the constant factor $r$, it follows that the quantity
\[\rmh(\mu,g^{(r)}) = r\rmh(\mu,g)\]
is an isomorphism-invariant of $\bfY \ltimes_\s \bfX$.  Since $\rmh(\mu,g)$ is finite and positive, these values are distinct for distinct $r$, and so, conditionally on the Enhanced Invariance Principle, this family of examples proves the following.

\vspace{7pt}

\noindent\textbf{Conditional Corollary B}\quad \emph{For any $h \in (0,\infty)$, there is a compact manifold with a smooth volume form that admits continuum-many smooth, volume-preserving K-automorphisms of entropy $h$ which are pairwise non-isomorphic. \qed} 

\vspace{7pt}

The corresponding result for non-invertible maps was also proved by Ball in~\cite{Bal03}.  This possible consequence of the current work was brought to my attention by J.-P. Thouvenot.

\subsection{A new isomorphism invariant}\label{subs:intro-invt}

The key to Theorem A will be a new isomorphism-invariant of probability-preserving systems.

The definition of this new invariant is rather involved, and will not be given in full until Section~\ref{sec:RRDr}.  However, some motivation for it can be given in advance.  This will involve standard notions from information theory, which the unfamiliar reader can find recalled in Subsection~\ref{subs:info}.

Consider again the basic examples $\rm{RWRS}_\mu$. First, let us recall why the Kolmogorov-Sinai entropy of $\rm{RWRS}_\mu$ does not give any information about $\rmh(\mu,S)$.  Let $\bfY := (\{\pm 1\}^\bbZ,\nu_{1/2}^{\otimes \bbZ},S)$, let $\bfX := (C^\bbZ,\mu,S)$ be a scenery process (here in discrete time), let $\rho = \nu_{1/2}^{\otimes \bbZ}\otimes \mu$, and let $\s:\{\pm 1\}^\bbZ\to \{\pm 1\}$ be the time-zero coordinate.  These data together define $\rm{RWRS}_\mu = \bfY \ltimes_\s \bfX$.  Let $(Z,\rho) := (\{\pm 1\}^\bbZ\times C^\bbZ,\nu_{1/2}^{\otimes \bbZ}\otimes \mu)$, and let
\[\a:Z \to \{\pm 1\}\times C\]
be the time-zero map corresponding to the obvious generating partition $\R$ for $\rm{RWRS}_\mu$.  Let $\Q$ be the time-zero partition of $C^\bbZ$.

For $N \in \bbN$, let $\rho_N := \a^{[0;N)}_\ast\mu \in \Pr((\{\pm 1\}\times C)^N)$ be the distribution of the $(\a,N)$-name
\begin{multline*}
\big(\a(z),\a((S\ltimes_\s S)(z)),\ldots,\a((S\ltimes_\s S)^{N-1}(z))\big)\\
= \big((y_0,\ldots,y_{N-1}),(x_0,x_{y_0},x_{\s^y_1},\ldots,x_{\s^y_{N-1}})\big)
\end{multline*}
when $z = (y,x)$ is drawn from $\rho$.  The Kolmogorov-Sinai entropy of $\rm{RWRS}_\mu$ is given by the leading-order behaviour of the sequence of Shannon entropies $\rmH(\rho_N) = \rmH_\rho(\R^{[0;N)})$.

This can be computed in terms of the information function of $\R^{[0;N)}$:
\[\rmI_{\mu,\R^{[0;N)}}:Z\to [0,\infty):z\mapsto -\log \rho(\R^{[0;N)}(z)).\]
In the $N$-name written above, the string $(y_0,\ldots,y_{N-1})$ is equally likely to be any element of $\{\pm 1\}^N$, so this contributes $(\log 2) N$ to the value $\rmI_{\rho,\R^{[0;N)}}(y,x)$.  However, having fixed $(y_0,\ldots,y_{N-1})$, the possible output strings $(x_0,x_{y_0},\ldots,x_{\s^y_{N-1}})$ are in bijective correspondence with the scenery-portions $(x_m)_{m \in \s^y_{[0;N)}}$, where $\s^y_{[0;N)} = \{\s^y_n\,|\ n \in [0;N)\}$.  This gives the total value for the information function as
\[\rmI_{\rho,\R^{[0;N)}}(y,x) = (\log 2)N + \rmI_{\mu,\Q^{\s^y_{[0;N)}}}(x).\]
By the Shannon-McMillan Theorem (recalled as Theorem~\ref{thm:SM} below), for typical $(y,x)$ and large $N$ this is
\begin{eqnarray}\label{eq:typical-info}
(\log 2)N + \rmh(\mu,S) |\s^y_{[0;N)}| + \rm{o}(|\s^y_{[0;N)}|).
\end{eqnarray}
Simple random walk on $\bbZ$ behaves diffusively, meaning that for typical $y$ the cardinality $|\s^y_{[0;N)}|$ is of order $\sqrt{N}$.  Therefore for typical $y$ the above value is given by
\[(\log 2)N + \rmh(\mu,S) c_N(y)\sqrt{N} + \rm{o}(\sqrt{N})\]
for some value $c_N(y)$ which is typically of order $1$.

Thus, if we ignore certain rare events in $(y,x)$, then the entropy of the scenery contributes only a correction of order $\sqrt{N}$ to
\[\rmH_\rho(\R^{[0;N)}) = \int \rmI_{\rho,\R^{[0;N)}}(z)\,\rho(\d z).\]
This sublinear correction disappears in the limit that computes $\rmh(\rm{RWRS}_\mu)$.

In general, sublinear terms in the growth-rate of $\rmH_\rho(\R^{[0;N)})$ are not isomorphism-invariant, so we cannot use the above calculation to prove the invariance of $\rmh(\mu,S)$.  Towards fixing this problem, let us next consider a different way to look at these corrections, in terms of another information-theoretic quantity: the mutual information between the $N$-step past $\R^{[-N;0)}$ and the $N$-step future $\R^{[0;N)}$.  By definition, this is
\[\rmH_\rho(\R^{[-N;0)}) + \rmH_\rho(\R^{[0;N)}) - \rmH_\rho(\R^{[-N;N)}).\]
Now each term here may be written as an integral of information functions:
\[\int \big(\rmI_{\rho,\R^{[-N;0)}}(z) + \rmI_{\rho,\R^{[0;N)}}(z) - \rmI_{\rho,\R^{[-N;N)}}(z)\big)\,\rho(\d z).\]
Let us again ask about the typical behaviour of the integrand here for $z \sim \rho$, ignoring certain extreme events (specifically, that the simple random walk covers much more ground that expected between times $-N$ and $N$).

Substituting from~(\ref{eq:typical-info}), we find that for typical $z = (y,x)$ and sufficiently large $N$ we have
\begin{eqnarray*}
&&\rmI_{\rho,\R^{[-N;0)}}(z) + \rmI_{\rho,\R^{[0;N)}}(z) - \rmI_{\rho,\R^{[-N;N)}}(z)\\
&&= (\log 2)(N + N - 2N) + \rmh(\mu,S)\big(|\s^y_{[-N;0)}| + |\s^y_{[0;N)}| - |\s^y_{[-N;N)}|\big) + \rm{o}(\sqrt{N})\\
&&= 0 + \rmh(\mu,S)|\s^y_{[-N;0)}\cap \s^y_{[0;N)}| + \rm{o}(\sqrt{N}).
\end{eqnarray*}
Heuristically, this calculation runs as follows: the steps taken by the walk in the past and future are independent, so contribute nothing to the mutual information; and the remaining mutual information is all contributed by \emph{that portion of the scenery visited by both the $N$-step past and the $N$-step future}.

Now, an easy appeal to Donsker's Invariance Principle gives that as $N\to\infty$, the random variable $y \mapsto |\s^y_{[-N;0)}\cap \s^y_{0;N)}|/\sqrt{N}$ converges in law to the random variable $\calL^1(B_{[0,1]}\cap B'_{[0,1]})$, where $B$ and $B'$ are independent Brownian motions.  This suggests that, provided one allows for this limiting behaviour of the random variable $y \mapsto |\s^y_{[-N;0)}\cap \s^y_{0;N)}|$, the constant $\rmh(\mu,S)$ should be visible in the asymptotic behaviour of $\rmI_\rho(\R^{[-N;0)};\R^{[0;N)})$ (perhaps after allowing the excision of a small-measure subset of $\{\pm 1\}^\bbZ\times C^\bbZ$ to remove `pathological' random-walk trajectories).

As with the sublinear entropy-corrections themselves, one expects that the sequence of mutual informations $\rmI_\rho(\R^{[-N;0)};\R^{[0;N)})$ does not give an isomorphism-invariant of general processes $(\bfZ,\R)$ (although I have not proved this carefully).  The key remaining idea is to modify the definition of $\rmI_\rho$ to obtain a more robust quantity.

The way to do this is suggested by a general viewpoint that already has already been very fruitful in the study of Kolmogorov-Sinai entropy.  Given an ergodic system $\bfZ = (Z,\rho,R)$, a finite measurable partition $\R$ of $Z$ and a finite-valued map $\a:Z\to A$ which generates $\R$, the Shannon-McMillan Theorem expresses $\rmh(\bfZ,\R)$ as the exponential growth rate of the effective number of points in $A^N$ needed to support $\a^{[0;N)}_\ast\rho$.  However, as observed by Feldman~(\cite{Feld80}), this may also be approximated by choosing some sufficiently small $\delta > 0$, and then asking after the exponential growth rate of the number of $(\delta N)$-balls needed to cover most of the measure $\a^{[0;N)}_\ast\rho$ in the Hamming metric spaces
\[(A^N,d_{\rm{Ham}}).\]

Having proved this covering-number representation, the isomorphism-invariance of $\rmh(\bfZ,\R)$ follows fairly easily, since an isomorphism of processes may be approximated, for sufficiently large $N$, by a sequence of \emph{Lipschitz} maps between these metric spaces, for which the change in those covering numbers is easily controlled.

Inspired by this viewpoint, our replacement for the sequence $\rmI_\rho(\R^{[-N;0)},\R^{[0;N)})$ will be a sequence of values measuring how much `information' is held by both of the partitions $\R^{[-N;0)}$ and $\R^{[0;N)}$ if one insists that this `information' can be recovered robustly if one allows small errors according to the Hamming metrics on $A^{[-N;0)}$ and $A^{[0;N)}$.

An important step in this paper is the rigorous development of this new invariant, via notions defined on abstract spaces that carry pairs of metrics.  This will be the work of Section~\ref{sec:RRDr}.

\subsection{Outline of the remaining sections}


Sections~\ref{sec:prelim-anal} and~\ref{sec:prelim-ET} present a variety of standard or routine results that will be needed later, concerning analysis and dynamics respectively.  Subsection~\ref{subs:EIP} formulates the Enhanced Invariance Principle and described some cases in which it is known.

Section~\ref{sec:prelim-discuss} is a warm-up for the rest of the paper.  It describes some basic features of the marginal metric spaces that arise from the skew-products in Theorem A.

Section~\ref{sec:RRDr} introduces the specific new isomorphism-invariant at the heart of the proof of Theorem A, estimates it in a few simple cases, and states the more precise Theorem~\ref{thm:rate} about its behaviour for the skew-products that appear in Theorem A.

Sections~\ref{sec:before-upper-bd} and~\ref{sec:upper-bd} prove the upper bound asserted in Theorem~\ref{thm:rate}.

Section~\ref{sec:dCs} returns to the study of well-distributed cocycles, focusing on some more subtle properties that are needed for the lower bound.  Chief among these is the ability, for a `typical' trajectory of the cocycle $\s$ over the interval $\{0,1\ldots,N-1\}$, to find very many somewhat large subsets of this interval on which $\s$ is injective, and which have a discrete `Cantor-like' structure.

Section~\ref{sec:lower-bd} then uses these finer properties to prove the lower bound asserted in Theorem~\ref{thm:rate}, and hence complete the proof of Theorem A.  This is more difficult than the upper-bound proof, and draws important ideas from~\cite{Kal82}.

Finally, Section~\ref{sec:further-ques} formulates some open questions and directions for further investigation.

\subsection*{Acknowledgements}

Jean-Paul Thouvenot shared with me several important insights in connection with the problems addressed here, as well as his enthusiasm for them.  Peter Nandori introduced me to Local Limit Theorems for smooth cocycles and helped me greatly in understanding them, and then Michael Bromberg shared with me his current progress on local times for cocycles over Gibbs-Markov processes. \fin

\section{Preliminaries: analysis and probability}\label{sec:prelim-anal}

\subsection{Basic conventions}

An \textbf{interval} will be either an interval in $\bbR$ or a discrete interval in $\bbZ$; the ambient set will always be clear from the context.  If $a,b \in \bbZ$ with $a \leq b$ then $[a;b] := [a;b+1) := (a-1;b] := \{a,a+1,\ldots,b\}$.  Sometimes we use the abbreviation $[n] := [0;n)$.

Given an interval $K \subseteq \bbR$, we will let $\rm{Int}(K)$ denote the collection of all nonempty compact subintervals of $K$.  We give it the topology inherited from the obvious identification with $\{(u,v) \in K\,|\ u\leq v\} \subseteq \bbR^2$.

Lebesgue measure on $\bbR$ will be denoted by $\cal{L}^1$.  If $I$ is a bounded interval in either $\bbR$ or $\bbZ$, then $\rm{U}_I$ will denote the uniform probability distribution on $I$.

We will use $\star$ to denote convolution of functions or measures on $\bbR$, in any case in which it is well-defined.

In this paper, a \textbf{mollifier} will be a compactly-supported smooth function $\phi:\bbR\to [0,\infty)$ which is symmetric about the origin and satisfies $\int \phi\,\d\calL^1 = 1$.

The following popular notation from harmonic analysis will be useful later.  Given two collections $(A_i)_{i\in I}$, $(B_i)_{i\in I}$ of non-negative real numbers and another structure or quantity $X$, we write $A_i \lesssim_X B_i$ to assert that there is a constant $C \in (0,\infty)$ depending only on $X$ such that $A_i \leq CB_i$ for all $i$.  We write $A_i \sim_X B_i$ in case both $A_i \lesssim_X B_i$ and $B_i \lesssim_X A_i$.

\subsection{Probability}

Various later arguments will involve comparisons with Brownian motion.  We will always let $\sfW \in \Pr C[0,\infty)$ be the classical Wiener measure, and let $\sfW_{[0,1]} \in \Pr C[0,1]$ be the law of $B|_{[0,1]}$ for $B \sim \sfW$.  This latter is supported on the closed subset
\[C_0(0,1] := \{f \in C[0,1]\,|\ f(0) = 0\}.\]

If $(X,\S,\mu)$ is a probability space and $A \in \S$ has $\mu(A) > 0$, then $\mu_{|A}$ will denote the conditional measure $\mu(A\cap \,\cdot\,)/\mu(A)$.

In our dynamical applications, all probability spaces will be standard Borel, and we will generally omit their $\s$-algebras from the notation.

We will later make several uses of the following quantitative approximation to absolute continuity.

\begin{dfn}[Approximate absolute continuity]\label{dfn:approx-ab-ct}
Let $(X,\S)$ be a measurable space, $\mu$ and $\nu$ be finite measures on $X$, and $\eps \in [0,\infty)$ and $M \in (0,\infty)$.  Then we write that $\mu \ll_{M,\eps} \nu$ if
\[\mu(A) \leq M\nu(A) + \eps \quad \forall A \in \S,\]
and we write that $\mu \sim_{M,\eps} \nu$ if $\mu \ll_{M,\eps} \nu$ and $\nu \ll_{M,\eps} \mu$.
\end{dfn}

In case $\mu$ and $\nu$ are both probability measures, an easy exercise gives
\[\mu \ll_{1,\eps} \nu \quad \Longleftrightarrow \quad \nu \ll_{1,\eps} \mu \quad \Longleftrightarrow \quad \mu \sim_{1,\eps} \nu\]
(where the first equivalence holds because the Jordan decomposition gives $(\mu - \nu)^+(X) = (\mu - \nu)^-(X)$ for any two probability measures).  On the other hand, $\mu \ll_{M,0} \nu$ if and only if $\mu$ is absolutely continuous with respect to $\nu$ and $\|\d\mu/\d\nu\|_{L^\infty(\nu)} \leq M$.

The following basic properties are also routine to verify.

\begin{lem}\label{lem:abs-ct-basics}
Approximate absolute continuity enjoys the following properties:
\begin{itemize}
\item If $\mu_1 \ll_{M_1,\eps_1} \mu_2$ and $\mu_2 \ll_{M_2,\eps_2} \mu_3$, then
\[\mu_1 \ll_{M_1M_2,M_1\eps_2+ \eps_1} \mu_3.\]
\item If $\mu,\nu,\theta \in \Pr \bbR$ and $\mu\ll_{M,\eps}\nu$, then also $\theta\star \mu\ll_{M,\eps}\theta\star \nu$. \qed
\end{itemize}
\end{lem}

\subsection{Information Theory}\label{subs:info}

We shall make use of several notions from Information Theory.  The main definitions are recalled here, but we shall largely take standard facts for granted: Cover and Thomas~\cite{CovTho06} is a canonical reference.

Given a countable set $A$ and $\mu \in \Pr A$, the \textbf{Shannon entropy} of $\mu$ is
\[\rmH(\mu) := -\sum_{a \in A}\mu\{a\}\log \mu\{a\} \in [0,+\infty].\]
Relatedly, if $(X,\mu)$ is any probability space and $\phi:X\to A$ is measurable, then $\rmH_\mu(\phi) := \rmH(\phi_\ast\mu)$; and if $\P$ is a countable measurable partition of $X$, then $\rmH_\mu(\P) := \rmH_\mu(\phi)$ for any choice of countable-valued map $\phi$ whose level-sets are the cells of $\P$.

If $(X,\S,\mu)$ is any probability space and $\nu$ is another probability on $X$, then the \textbf{Kullback-Leibler divergence of $\nu$ with respect to $\mu$} is
\[\rmD(\nu\,|\,\mu) := \left\{\begin{array}{ll}\int_X \frac{\d\nu}{\d\mu}\log\frac{\d\nu}{\d\mu}\,\d\mu &\quad \hbox{if}\ \nu \ll \mu\\ +\infty & \quad \hbox{else}\end{array}\right. \in [0,+\infty].\]

Next, suppose that $\P$ and $\Q$ are two countable measurable partitions of $(X,\S,\mu)$.  Then the \textbf{conditional entropy} of $\P$ given $\Q$ is the quantity
\[\rmH_\mu(\P\,|\,\Q) := \sum_{C \in \Q}\mu(C) \rmH_{\mu_{|C}}(\P)\]
(where we interpret those $C \in \Q$ for which $\mu(C) = 0$ as contributing zero). The \textbf{mutual information} of $\P$ and $\Q$ under $\mu$ is defined by
\[\rmI_\mu(\P;\Q) := \rmH_\mu(\P) - \rmH_\mu(\P\,|\,\Q).\]
A standard calculation shows that this is symmetric in $\P$ and $\Q$, and also that
\begin{eqnarray}\label{eq:rel-ent-of-bary}
\rmI_\mu(\P;\Q) &=& \rmH_\mu(\P) + \rmH_\mu(\Q) - \rmH_\mu(\P\vee \Q)\nonumber\\
&=& \int \rmD\big(\phi_\ast(\mu_{|\Q(x)})\,\big|\,\phi_\ast\mu\big)\,\mu(\d x),
\end{eqnarray}
where $\phi:X\to A$ is any finite-valued function generating the partition $\P$ (see, for instance, Equations (2.45) and (2.36) in~\cite[Section 2.4]{CovTho06}).  More generally, given a third partition $\R$, the \textbf{conditional mutual information} of $\P$ and $\Q$ given $\R$ is
\begin{eqnarray*}
\rmI_\mu(\P;\Q\,|\,\R) &:=& \rmH_\mu(\P\,|\,\R) - \rmH_\mu(\P\,|\,\Q\vee \R)\\
&=& \int \rmI_{\mu_{|\R(x)}}(\P;\Q)\,\mu(\d x),
\end{eqnarray*}
where the second equality is another standard calculation.

These definitions easily give the following.

\begin{lem}\label{lem:mut-inf-add}
Let $(X_i,\S_i,\mu_i)$ for $i=1,2$ be probability spaces, and for each $i$ let $\P_i$, $\Q_i$ and $\R_i$ be countable measurable partitions of $X_i$.  Then
\[\rmI_{\mu_1\otimes \mu_2}(\P_1\otimes \P_2;\Q_1\otimes \Q_2\,|\,\R_1\otimes \R_2) = \rmI_{\mu_1}(\P_1;\Q_1\,|\,\R_1) + \rmI_{\mu_2}(\P_2;\Q_2\,|\,\R_2).\]
\qed
\end{lem}

We will also need the following simple but less standard calculations.

\begin{lem}[Conditioning mutual information on a subset]\label{lem:cond-cond-mut-inf}
If $(X,\S,\mu)$ is a probability space, $\P$, $\Q$, and $\R$ are countable partitions in $\S$, and $A\in \S$ has positive measure, then
\[\mu(A)\rmI_{\mu_{|A}}(\P;\Q\,|\,\R) \leq \log 2 + \rmI_\mu(\P;\Q\,|\,\R).\]
\end{lem}

\begin{proof}
Let $\A := \{A,X\setminus A\}$. From the definition of conditional mutual information and the fact that it is always non-negative~(\cite[Corollary 2.6.3]{CovTho06}), one obtains
\begin{eqnarray*}
\mu(A)\rmI_{\mu_{|A}}(\P;\Q\,|\,\R) &\leq& \mu(A)\rmI_{\mu_{|A}}(\P;\Q\,|\,\R) + \mu(X\setminus A)\rmI_{\mu_{|X\setminus A}}(\P;\Q\,|\,\R)\\
&=& \rmI_\mu(\P;\Q\,|\,\R\vee \A).
\end{eqnarray*}
The Chain Rule for mutual information~(\cite[Theorem 2.5.2]{CovTho06}) gives
\begin{multline*}
\rmI_\mu(\P \vee \A;\Q\,|\,\R) = \rmI_\mu(\A;\Q\,|\,\R) + \rmI_\mu(\P;\Q\,|\,\R\vee \A)\\
\Longrightarrow \quad \rmI_\mu(\P;\Q\,|\,\R\vee \A) \leq \rmI_\mu(\P\vee \A;\Q\,|\,\R),
\end{multline*}
and now another use of the definitions, subadditivity of entropy and the Data-Processing Inequality gives
\begin{eqnarray*}
\rmI_\mu(\P\vee \A;\Q\,|\,\R) &=& \rmH_\mu(\P\vee \A\,|\,\R) - \rmH_\mu(\P\vee \A\,|\,\Q\vee \R)\\
&\leq& \rmH_\mu(\A) + \rmH_\mu(\P\,|\,\R) - \rmH_\mu(\P\,|\,\Q\vee \R)\\
&=& \rmH_\mu(\A) + \rmI_\mu(\P;\Q\,|\,\R).
\end{eqnarray*}
Finally, $\rmH_\mu(\A) \leq \log |\A| = \log 2$.
\end{proof}

\begin{lem}[Uniform integrability from relative entropy bound]\label{lem:ent-to-unif-int}
If $(X,\S,\mu)$ is a probability space and $\nu = f\cdot \mu\in \Pr X$ with $D := \rmD(\nu\,|\,\mu) < \infty$, then for any $C > 0$ one has
\[\nu \ll_{\rme^C,(D + \rme^{-1})/C} \mu.\]
\end{lem}

\begin{proof}
Since $D = \int f \log f\,\d\mu$, and the function $t\mapsto t\log t$ has a global minimum at $t = \rme^{-1}$ with value $-\rme^{-1}$, one has
\begin{multline*}
\nu\{f > \rme^C\} = \int_{\{f > \rme^C\}}f\,\d\mu = \int_{\{\log f > C\}} f\,\d\mu \leq \int_{\{\log f > C\}} f\frac{\log f}{C}\,\d\mu\\ \leq \frac{1}{C}\int|f\log f|\,\d\mu \leq \frac{D + \rme^{-1}}{C}.
\end{multline*}
Therefore for any measurable $A \subseteq X$ one has
\[\nu(A) = \int_A f\,\d\mu \leq \rme^C\mu(A) + \nu(A\cap \{f > \rme^C\}) \leq \rme^C\mu(A) + \frac{D + \rme^{-1}}{C}.\]
\end{proof}

\subsection{Metric and pseudometric spaces}

If $(X,d)$ is a metric or pseudometric space, $x \in X$ and $r \geq 0$, then
\[B^d_r(x) := \{y \in X\,|\ d(x,y) < r\}\]
is the radius-$r$ open ball around $x$.  It will sometimes be abbreviated to $B_r(x)$ if $d$ is understood.  If $F \subseteq X$, then $B_r(F) := \bigcup_{x \in F}B_r(x)$. A subset $F \subseteq X$ is \textbf{$r$-separated} if
\[d(x,y) \geq r \quad \forall x,y \in F\ \hbox{distinct}.\]
The \textbf{$r$-covering number} of $(X,d)$ is
\[\cov((X,d),r) = \min\{|F|\,|\ F \subseteq X,\ B_r(F) = X\}.\]

A \textbf{metric measure} (`\textbf{m.m.}') \textbf{space} is a triple $(X,d,\mu)$ consisting of a metric space $(X,d)$ and a Radon measure $\mu$ on $X$.  In this paper it will always be tacitly assumed that $\mu$ is finite.  If $\mu(X) = 1$ then $(X,d,\mu)$ is a \textbf{metric probability} (`\textbf{m.p.}') \textbf{space}.  All m.m. spaces appearing below will either be compact or arise as Borel subsets of compact spaces.

It will sometimes be necessary to generalize this class to include pseudometrics.  However, the open balls for a pseudometric may not generate the whole of the relevant $\s$-algebra.  Thus, in this paper, a \textbf{pseudometric measure} (resp. \textbf{pseudometric probability}) (`\textbf{psm.m.}', resp. `\textbf{psm.p}') \textbf{space} will be a triple $(X,d,\mu)$ in which $X$ is a standard Borel space, $\mu$ is a measure (resp. probability) on $X$, and $d:X\times X\to [0,\infty)$ is a pseudometric which is Borel measurable on $X\times X$ and is totally bounded.  Clearly all compact m.m. spaces fall into this class.  This definition is similar to, though slightly more restrictive than, Vershik's class of `admissible' pseudometrics in~\cite{Ver10}.

If $X$ is a standard Borel space, then one may obtain a totally bounded Borel pseudometric $d$ on $X$ by letting $(Z,d^Z)$ be a compact metric space and $\phi:X\to Z$ a Borel map, and then taking $d := d^Z \circ \phi^{\times 2}$.  An easy exercise shows that every totally bounded Borel pseudometric $d$ on $X$ arises this way, by letting $(Z,d^Z)$ be the completion of the quotient of $X$ by the zero-distance equivalence relation defined by $d$. 

If $(X,d,\mu)$ is a psm.m. space with $\s$-algebra $\S$ and $U \in \S$, then we usually abbreviate
\[(U,d|_{U\times U},\mu|_{\S\cap U}) =: (U,d,\mu),\]
so this latter has total mass $\mu(U)$. On the other hand, if $(X,d,\mu)$ is a psm.p. space and $U \in \S$ has $\mu(U) > 0$, then
\[(U,d|_{U\times U},\mu(U)^{-1}\cdot \mu|_{\S\cap U}) =: (U,d,\mu_{|U}),\]
another psm.p. space.

Given a psm.m. space $(X,d,\mu)$ and $a,r > 0$, the \textbf{$a$-partial $r$-covering number} is
\begin{eqnarray}\label{eq:cov}
\rm{cov}_a((X,d,\mu),r) := \min\{|F|\,|\ F\subseteq X,\ \mu(B_r(F)) > a\}.
\end{eqnarray}

Much of the work later will concern a natural `roughening' of the class of Lipschitz maps.  Given pseudometric spaces $(X,d^X)$ and $(Y,d^Y)$, and also $c,L \geq 0$, a map $f:X\to Y$ is \textbf{$c$-almost $L$-Lipschitz} if
\[d^Y(f(x),f(x')) \leq Ld^X(x,x') + c \quad \forall x,x' \in X.\]
This class of maps already has a natural place in the study of concentration of measure.  For instance, it appears repeatedly in Chapter 3$\frac{1}{2}$ of Gromov~\cite{Gro01} (starting in the proof of 3$\frac{1}{2}$.15(b)), under the terminology `$K$-Lipschitz up to $c$'.

\section{Preliminaries: ergodic theory}\label{sec:prelim-ET}

We shall need to call on a variety of classical results from ergodic theory, and especially from entropy and Ornstein theory for probability-preserving transformations.  Two standard references that emphasize the material we need are Shields~\cite{Shi96} and Kalikow and McCutcheon~\cite{KalMcC10}.

\subsection{Probability-preserving systems and their entropy}\label{subs:basic}

In the following, a \textbf{probability-preserving} (`\textbf{p.-p.}') \textbf{system} is a triple $(X,\mu,T)$ in which $(X,\mu)$ is a standard Borel probability space and $T:X \to X$ is measurable, has a measurable inverse, and preserves $\mu$.  Similarly, a \textbf{p.-p. flow} is a triple $(X,\mu,T)$ in which $(X,\mu)$ is standard Borel and $T:\bbR\actson X$ is jointly measurable and $\mu$-preserving.  Many properties of such a flow are closely related to properties of its \textbf{time-$1$ system} $(X,\mu,T^1)$.

The classical entropy theory of p.-p. systems is most easily introduced in terms of finite partitions of $X$ (or, equivalently, finite-valued measurable functions on $X$).  We will assume this theory as it is presented, for example, in~\cite{Shi96} or~\cite{KalMcC10}.

An essential tool will be the Shannon-McMillan Theorem.  Some further notation will be useful.  Suppose that $(X,\mu,T)$ is a p.-p. system and that $\P$ is a finite Borel partition of $X$.  A pair such as $(\bfX,\P)$ will be called a \textbf{process}.  For any subset $F \subseteq \bbZ$, let
\[\P^F := \bigvee_{n \in F} T^{-n}(\P),\]
where this is interpreted as a new partition in case $F$ is finite, or, more generally, as a $\s$-subalgebra of the $\s$-algebra of $X$ if $F$ is infinite.  Now let
\[X^{\rm{SM}}_{I,\eps} := \big\{x \in X\,\big|\ \rme^{-(\rmh(\bfX,\P) + \eps)|I|} < \mu(\P^I(x)) < \rme^{-(\rmh(\bfX,\P) - \eps)|I|}\big\}\]
(so this depends on $\P$, although the notation suppresses that dependence). Clearly $X^{\rm{SM}}_{I + n,\eps} = T^n(X^{\rm{SM}}_{I,\eps})$ for every $n \in \bbZ$.

The following can be found in~\cite[Section 4.2]{KalMcC10} or~\cite[Sections I.5 and I.6]{Shi96}.

\begin{thm}[Shannon-McMillan Theorem]\label{thm:SM}
If $(X,\mu,T)$ is ergodic then
\[\mu(X^{\rm{SM}}_{I,\eps}) \to 1\]
as $|I| \to \infty$ for any fixed $\eps > 0$. \qed
\end{thm}

The following is also essentially a standard result.

\begin{lem}\label{lem:alt-ent-formula}
For any $N\geq 1$, the Kolmogorov-Sinai entropy satisfies
\[\rmH_\mu(\P^{[0;N)}\,|\,\P^{[-M;0)}) \ \downarrow \ \rmh(\bfX,\P)N \quad \hbox{as}\ M\to\infty.\]
\end{lem}

\begin{proof}
When $N=1$, this can be obtained from~\cite[Corollary 423 and Theorem 434]{KalMcC10} or from~\cite[Equation I.6(3)]{Shi96}.  For general $M$, the chain rule for relative entropy (see, for instance,~\cite[Equation I.6(1)]{Shi96} or~\cite[Section 2.5]{CovTho06}) gives
\[\rmH_\mu(\P^{[0;N)}\,|\,\P^{[-M;0)}) = \sum_{n=0}^{N-1} \rmH_\mu(T^{-n}(\P)\,|\,\P^{[-M;n)}).\]
Since $N$ is fixed, we may now apply the special case to each right-hand summand separately as $M\to\infty$.
\end{proof}

Now consider two discrete intervals $I,J \subseteq \bbZ$ such that $I\cup J$ is also a discrete interval: thus, either one of them is empty, or they are adjacent, or they intersect.

\begin{lem}\label{lem:cond-mut-inf-sublin}
Given $\bfX$ and $\P$, there is a function $g:\bbN\to [0,\infty)$ with $g(m) = \rm{o}(m)$ as $m\to\infty$ such that
\[\rmI_\mu(\P^J;\P^I\,|\,\P^{I\cap J}) \leq g(|J\setminus I|).\]
\end{lem}

\begin{proof}
The definition of $\rmI_\mu$ gives
\begin{eqnarray}\label{eq:cond-mut-inf-form-again}
\rmI_\mu(\P^J;\P^I\,|\,\P^{I\cap J}) &=& \rmH_\mu(\P^J\,|\,\P^{I\cap J}) - \rmH_\mu(\P^J\,|\,\P^I) \nonumber \\
&=& \rmH_\mu(\P^{J\setminus I}\,|\,\P^{I\cap J}) - \rmH_\mu(\P^{J\setminus I}\,|\,\P^I),
\end{eqnarray}
since $\P^I \vee \P^{I\cap J} = \P^I$. Various cases are now trivial: if either $I$ or $J$ is empty, or if either $I\subseteq J$ or $J\subseteq I$, then this right-hand side collapses to zero.

In the remaining case, we observe that $J\setminus I$ is also a nonempty interval.  In this case, standard monotonicity properties of conditional entropy together with Lemma~\ref{lem:alt-ent-formula} give
\[\rmh(\bfX,\P)|J\setminus I| \leq \rmH_\mu(\P^{J\setminus I}\,|\,\P^I) \leq \rmH_\mu(\P^{J\setminus I}\,|\,\P^{I\cap J}) \leq \rmH_\mu(\P^{J\setminus I}).\]
However, the right-hand quantity here is of the form
\[\rmh(\bfX,\P)|J\setminus I| + g(|J\setminus I|)\]
for some sublinear function $g$, so the right-hand side of~(\ref{eq:cond-mut-inf-form-again}) is bounded by this $g$, completing the proof.
\end{proof}

Given a p.-p. transformation $\bfX = (X,\mu,T)$ and a finite Borel partition $\P$, one may always choose a finite set $A$ and function $\phi_0:X\to A$ which generates $\P$.  Having done so, let $\phi_n:= \phi_0\circ T^n$ for each $n \in \bbZ$, and more generally $\phi_F := (\phi_n)_{n\in F}:X\to A^F$ for $F \subseteq \bbZ$.  Abbreviate $\phi_\bbZ =: \phi$, so this is now a factor map
\[(X,\mu,T)\to (A^\bbZ,\phi_\ast\mu,S).\]
The entropy of the process $(\bfX,\P)$ may be understood as the entropy rate of $\phi_\ast\mu$, regarded as the law of a stationary sequence of $A$-valued random variables.

Having fixed $\P$, $A$ and $\phi$, the map $\phi_{(-\infty;0)}:X\to A^{(-\infty;0)}$ is referred to as the \textbf{past} of the process $(\bfX,\P)$.  The measure $\mu$ may be disintegrated over $\phi_{(-\infty;0)}$, giving a probability kernel
\[A^{(-\infty;0)} \to \Pr X:z \mapsto \mu_z;\]
this is referred to as \textbf{conditioning on the past}.  Various entropy-theoretic properties may be expressed in terms of these conditional measures: in the first place,
\[\rmh(\bfX,\P) = \int \rmH_{\mu_{\phi^{(-\infty;0)}(x)}}(\P)\,\mu(\d x) = \int \rmH(\phi_\ast \mu_{\phi^{(-\infty;0)}(x)})\,\mu(\d x),\]
the expected Shannon entropy of $\P$ given the past (see~\cite[Subsection I.6.b]{Shi96}).

\subsection{Compact models}

Instead of finite partitions, much of our later work will rely on endowing $X$ with a compact metric for which $T$ is continuous.  This is always possible by the following classical result (see, for instance,~\cite[Theorem 5.7]{Varadara85}):

\begin{thm}\label{thm:vara}
If $(X,\mu,T)$ is any jointly measurable p.-p. action of an l.c.s.c. group on a standard Borel probability space, then it is isomorphic as such to a jointly continuous action on a compact metric space with an invariant probability measure. \qed
\end{thm}

In case $(X,d^X)$ is a compact metric space, $T$ is a jointly continuous action of $\bbZ$ or $\bbR$ on $X$, and $\mu \in \Pr^T X$, we shall refer to $(X,d^X,\mu,T)$ as a \textbf{compact model p.-p. system} or \textbf{flow}.  We shall work with compact models of our systems in much of the sequel.  Of course, after choosing compact models, we must still allow arbitrary Borel (not necessarily continuous) factor maps between them.  They key to using the metric space structure, in spite of this flexibility, will be Lusin's Theorem.

One can use such a choice of metric $d^X$ to express the Kolmogorov-Sinai entropy.  This relationship can be traced back to Feldman's work in~\cite{Feld80}, and it is worked out in detail (for actions of general unimodular amenable groups) by Ornstein and Weiss in~\cite[Part II]{OrnWei87}.  We quickly recall some of the results that we need here, largely referring to that latter work.

First, for any compact model p.-p. system $(X,d^X,\mu,T)$ and any finite $F \subseteq \bbZ$, let
\[d^\bfX_F(x,x') := \sum_{n \in F}d^X(T^nx,T^nx').\]
This is a sequence of metrics on $X$. In terms of this construction, for any $r > 0$, one defines the \textbf{spatial $r$-entropy} $\rmh(\mu,T,d^X,r)$ by
\begin{eqnarray}\label{eq:KS-dfn}
\rmh(\mu,T,d^X,r) := \sup_{\eps > 0}\liminf_{N\to\infty}\frac{1}{N}\rm{cov}_{1 - \eps}((X,d^\bfX_{[0;N)},\mu),r N).
\end{eqnarray}

Similarly, if $(X,d^X,\mu,T)$ is a compact p.-p. flow and $F \subseteq \bbR$ is measurable with finite measure, then
\[d^\bfX_F(x,x') := \int_F d^X(T^tx,T^tx')\,\d t,\]
and the \textbf{spatial $r$-entropy} $\rmh(\mu,T,d^X,r)$ is again given by~(\ref{eq:KS-dfn}), where now $N$ is allowed to run through real values.

The connection between these spatial entropies and the Kolmogorov-Sinai entropy is the following, established in~\cite{Feld80,OrnWei87}:

\begin{prop}\label{prop:spatial-ent-and-KS-ent}
In the setting of either a compact model system or compact model flow, one has
\[\sup_{r > 0}\rmh(\mu,T,d^X,r) = \lim_{r \to 0}\rmh(\mu,T,d^X,r) = \rmh(\mu,T).\]
\qed
\end{prop}

Corresponding to this, one would expect a relative of the Shannon-McMillan Theorem~\ref{thm:SM} for the exponential order of the $\mu$-measure of a typical small-radius ball in the space $(X,d^\bfX_{[0;N)},\mu)$, once $N$ is large.  Such a result is proved in~\cite[Section II.4, Theorem 5]{OrnWei87}.  The related result that we will use below is actually a step on the way to their proof of that theorem.

\begin{prop}[{\cite[Section II.4, Proposition 3]{OrnWei87}}]\label{prop:from-OW}
For any $\b \in (0,1]$, $r > 0$ and $h^\ast < \rmh(\mu,T,d^X,r)$, one has
\[\cov_\b\big(\big(X,d^\bfX_{[0;N)},\mu),r N) > \exp (h^\ast N)\]
for all sufficiently large $N$. \qed
\end{prop}

The approach to entropy theory using compact metrics, rather than partitions, will be highly convenient in the rest of this paper.  In Section~\ref{sec:RRDr}, a new invariant of systems will be defined explicitly in terms of the sequences of metrics $d^\bfX_{[0;N)}$, and we will see that this `geometric' definition leads naturally to a proof of isomorphism-invariance similar to a proof of the Kolmorogov-Sinai Theorem in terms of these metrics.

For the entropy theory of $\bbR$-actions, it has long been known that the metric-based approach is considerably cleaner and more efficient: this realization goes back to Feldman~\cite{Feld80}, and stimulated the use of compact metrics in ergodic theory more generally.  This program has recently been actively promoted by Vershik and his co-workers~(\cite{Ver10,VerZatPet13}).  As will become clear in Section~\ref{sec:RRDr}, the present paper owes a great deal to this point of view.

Given a topological flow $T:\bbR\actson X$ with metric $d^X$, another dynamically-defined sequence of metrics on $X$ may be obtained by supremizing over time-intervals, rather than integrating: for any nonempty compact $F \subseteq \bbR$, let
\[d^{\bfX,\infty}_F(x,x') := \sup_{t \in F}d^X(T^tx,T^tx').\]
When it is necessary to distinguish this from the earlier metric, we will refer to the metrics $d^\bfX_F$ as \textbf{Hamming-like metrics} and to the metrics $d^{\bfX,\infty}_F$ as \textbf{Bowen-Dinaburg metrics}. In topological dynamics, the asymptotic packing or covering numbers of the metrics $d^{\bfX,\infty}_F$ are the basis of the Bowen-Dinaburg approach to topological entropy, but are not so directly related to Kolmogorov-Sinai entropy.  However, it will be convenient to know later that given a topological flow $(X,T)$ and an ergodic invariant probability $\mu$, these alternative metrics may also be used to define $\rmh(\mu,T)$.  This has previously been proved in~\cite[Theorem 1.1]{Kat80b}.  However, we will need a slightly stronger, local version of that control, so we include a precise statement and proof here.  Clearly $d^{\bfX,\infty}_F \geq d^\bfX_F$, but we will need a result in the reverse direction.

\begin{lem}\label{lem:d1-cov-and-dinf-cov}
If $(X,d^X,\mu,T)$ is an ergodic compact model flow, then for every $\eps,\delta > 0$ there is a $\delta_1 > 0$ such that, for every $x \in X$ and $K \in \rm{Int}(\bbR)$ with $\calL^1(K) \geq 1$, one has
\[\cov\big((B^{d^\bfX_K}_{\delta_1\calL^1(K)}(x),d^{\bfX,\infty}_K),\delta \big) < \exp(\eps \calL^1(K)).\]
\end{lem}

\begin{proof}
Clearly it suffices to prove this with $I = [0,a]$ for some $a \geq 1$.  Let $N := \lfloor a \rfloor$, and observe that $N \geq a/2$.

By the joint continuity of $T$, there is some $\delta' > 0$ such that
\[\forall x,x' \in X, \quad d^X(x,x') < \delta' \quad \Longrightarrow \quad \max_{t \in [-2,2]}d^X(T^tx,T^tx') < \delta/2,\]
and now there is also some $\delta'' > 0$ such that
\[\forall x,x' \in X, \quad d^X(x,x') < \delta'' \quad \Longrightarrow \quad \max_{t \in [-2,2]}d^X(T^tx,T^tx') < \delta'.\]

This latter condition implies that if $d^X(T^nx,T^nx') \geq \delta'$ for some $n \in \bbZ$, then $d^X(T^tx,T^tx') \geq \delta''$ for all $t \in [n,n+1]$, and therefore
\begin{eqnarray}\label{eq:real-d1-and-int-d1}
\int_0^a d^X(T^tx,T^tx')\,\d t \geq \delta''|\{n \in [0;N)\,|\ d^X(T^nx,T^nx') \geq \delta'\}|.
\end{eqnarray}

Let $\P = (P_1,\ldots,P_m)$ be a Borel partition of $X$ into sets of diameter less than $\delta'$.  Having chosen this, let $\eta \in (0,\eps)$ be so small that in the space $[0;m]^N$ the cardinality of a Hamming ball of radius $\eta N$ is less than $\rme^{\eps N}$ for all $N \geq 1$.  Finally, choose $\delta_1 := \eta\delta''/2$.

After these preliminaries, suppose that $x,x' \in X$ satisfy $d^\bfX_{[0,a]}(x,x') < \delta_1a \leq \eta\delta''N$.  Then~(\ref{eq:real-d1-and-int-d1}) implies that
\[|\{n \in [0;N)\,|\ d^X(T^nx,T^nx') \geq \delta'\}| < \eta N.\]
Fix $x$, and for each $n \in \bbZ$ let
\[P_{n,0} := T^{-n}(B^{d^X}_{\delta'}(T^nx)) \quad \hbox{and} \quad P_{n,i} := T^{-n}(P_i) \quad \hbox{for}\ i=1,2,\ldots,m.\]
Then the above estimates imply that
\[B^{d^\bfX_{[0,a]}}_{\delta_1 a}(x) \subseteq \bigcup_{\hbox{\scriptsize{$\begin{array}{c}(w_0,\ldots,w_{N-1}) \in [0;m]^N\\ |\{n \in [0;N)\,|\ w_n \neq 0\}| < \eta N\end{array}$}}}P_{0,w_0}\cap P_{1,w_1}\cap \cdots \cap P_{N-1,w_{N-1}}.\]
By the choice of $\delta'$, each individual intersection on the right here has $d^{\bfX,\infty}_{[0,a]}$-diameter less than $\delta$, and by the choice of $\eta$ the number of such intersections appearing in this union is less than $\rme^{\eps N}$.
\end{proof}

\subsection{Gibbs measures on mixing SFTs}\label{subs:Gibbs}

The source of base systems for the examples in Theorem A is the class of Gibbs measures on mixing SFTs, and other invariant states on topological dynamical systems that can be suitably coded from these.  These form the basic setting of the `thermodynamic formalism'.  The standard monographs~\cite{Bow08,ParPol90} provide a good reference for most of our needs, and~\cite{Rue04,Sin72} largely cover the same material.

Given a finite alphabet $A$, we shall usually consider $A^\bbZ$ endowed with the metric
\[d(a,a') := \sum_{n \in \bbZ}2^{-|n|}1_{\{a_n \neq a_n'\}}.\]
We also endow $A^{(-\infty;0)}$ with the analogous metric.

A function $A^\bbZ \to \bbR$ is \textbf{H\"older} if it is so with respect to $d$ for some positive H\"older exponent, and similarly for a function $A^{(-\infty;0)}\to \bbR$.  A function on $A^\bbZ$ is \textbf{one-sided} if it factorizes through the coordinate projection $A^\bbZ\to A^{(-\infty;0]}$.  Motivated by the thermodynamic formalism, we will sometimes refer to a H\"older function restricted to any closed subset of $A^\bbZ$ as a \textbf{potential} (ignoring the many more general potentials that can be considered in the thermodynamic formalism).

As usual, a \textbf{subshift of finite type} (`\textbf{SFT}') in $A^\bbZ$ is a closed $S$-invariant subset $Y \subseteq A^\bbZ$ defined by a finite set of forbidden subwords.  We always endow such an SFT with the restriction $d^Y$ of the metric $d$ above.

Given a topologically mixing SFT $Y \subseteq A^\bbZ$ and a potential $\phi:Y \to \bbR$, there is always an associated \textbf{Gibbs measure} $\nu \in \Pr^S Y$, uniquely characterized by the property that there are $c_1,c_2 \in (0,\infty)$ and $P \in \bbR$ such that
\begin{eqnarray}\label{eq:dfn-Gibbs}
c_1\exp\Big(P|I| + \sum_{n\in I}\phi(S^ny)\Big) \leq \nu(\P^I(y)) \leq c_2\exp\Big(P|I| + \sum_{n\in I}\phi(S^ny)\Big)
\end{eqnarray}
for all $y\in Y$ and bounded discrete intervals $I \subseteq \bbZ$: see~\cite[Theorem 1.4]{Bow08} or~\cite[Chapter 3]{ParPol90}.  Henceforth we shall refer to a triple $(Y,\nu,S)$ in which $(Y,S)$ is a mixing SFT and $\nu$ is the Gibbs measure associated to some potential as a \textbf{mixing Gibbs system}.

Now let $\a:Y\to A$ be the time-zero coordinate map and let $\P$ be the partition it generates.  Let $Y^- := \a^{(-\infty;0]}(Y) \subseteq A^{(\infty;0]}$.  The Gibbs measure $\nu$ associated to a potential $\phi$ is constructed via its image $\nu^- := \a^{(-\infty;0]}_\ast\nu \in \Pr Y^-$.  This image determines $\nu$ uniquely, by $S$-invariance.  As in the proof of Ruelle's Perron-Frobenius Theorem (see~\cite[Theorem 2.2]{ParPol90}), one may always find another H\"older function $\psi:Y^- \to \bbR$ such that
\begin{itemize}
\item $\phi - \psi\circ \a^{(-\infty;0]}$ is cohomologous to a constant over $S$ among H\"older functions, and
\item the Perron-Frobenius operator $C(Y^-)\to C(Y^-)$ defined by
\[L_\psi f(y) := \sum_{a \in A\,|\,ya \in Y^-}\rme^{\psi(ya)}f(ya)\]
satisfies $L_\psi 1_{Y^-} = 1_{Y^-}$, and otherwise has spectrum contained in a disk of radius strictly less than $1$ (in this case the Perron-Frobenius operator is said to be `normalized'~\cite[Chapter 2]{ParPol90}).
\end{itemize}
Having found this $\psi$, the measure $\nu^-$ is the unique probability measure for which $L_\psi^\ast\nu^- = \nu^-$.

After reconstructing $\nu$ from $\nu^-$, this Perron-Frobenius operator has the interpretation that for any bounded measurable function $f:Y^- \to \bbR$ and $r \in \bbN$ one has
\[\sfE_\nu(f\circ \a^{(-\infty;0]}\circ S^r\,|\,\P^{(-\infty;0]}) = (L_\psi^r f)\circ \a^{(-\infty;0]}.\]
In particular, if $y\mapsto \nu_y$ is the disintegration of $\nu$ over the strict past $\a^{(-\infty;0)}:Y\to A^{(-\infty;0)}$, then the equation $L_\psi^\ast \nu^- = \nu^-$ implies
\[\a_\ast\nu_y = \sum_{a \in A\,|\,ya \in Y^-}\rme^{\psi(ya)}\delta_a.\]

By the H\"older condition and the fact that $Y$ is an SFT, there are $b < \infty$, $\b \in (0,1)$ and $N_0 \in \bbN$ such that
\begin{multline*}
N \geq N_0 \quad \hbox{and} \quad y,y' \in Y^-\ \hbox{with}\ \P^{[-N_0;0]}(y) = \P^{[-N_0;0]}(y') \quad\\ \Longrightarrow \quad \{a\,|\ ya \in Y^-\} = \{a\,|\ y'a \in Y^-\}\ \hbox{and}\ \max_{a\,|\,ya \in Y^-} |\phi(ya) - \phi(y'a)| < b\b^N.
\end{multline*}
This has proved the following.

\begin{lem}[H\"older continuity of conditional measures]\label{lem:cond-meas-Hold}
In the setting above there are $N_0 \in \bbN$, $b < \infty$ and $\b \in (0,1)$ such that for any $N\geq N_0$ one has
\begin{multline*}
y,y' \in Y^- \ \hbox{with}\ \P^{[-N;0)}(y) = \P^{[-N;0)}(y') \\ \Longrightarrow \quad \a_\ast\nu_y \sim \a_\ast\nu_{y'} \quad \hbox{and} \quad
\rme^{-b\b^N} < \frac{\d(\a_\ast\nu_y)}{\d(\a_\ast\nu_{y'})} < \rme^{b\b^N}.
\end{multline*}
\qed
\end{lem}

\begin{cor}\label{cor:Gibbs-mut-inf}
If $(Y,\nu,S)$ and $\P$ are as above and also $p \in \bbN\cup \{0\}$, then
\[\sup_{N\geq 1}\rmI_\nu(\P^{[-p;N+p)};\P^{[-N-p;p)}) < \infty.\]
\end{cor}

\begin{proof}
Suppose first that $p=0$.  The Chain Rule for mutual information~(\cite[Theorem 2.5.2]{CovTho06}) gives
\begin{eqnarray}\label{eq:chain-rule}
\rmI_\nu(\P^{[0;N)};\P^{[-N;0)}) = \sum_{n=0}^{N-1}\rmI_\nu(S^{-n}(\P);\P^{[-N;0)}\,|\,\P^{[0;n)}).
\end{eqnarray}
Letting $N_0$ be as in Lemma~\ref{lem:cond-meas-Hold}, we obtain
\[\rme^{-b\b^n} < \frac{\d\big((\a\circ S^n)_\ast (\nu_{|\P^{[-N;n)}(y)})\big)}{\d\big((\a\circ S^n)_\ast (\nu_{|\P^{[0;n)}(y)})\big)} < \rme^{b\b^n} \quad \forall n \geq N_0,\]
and hence
\[\rmD\big((\a\circ S^n)_\ast (\nu_{|\P^{[-N;n)}(y)})\,\big|\,(\a\circ S^n)_\ast (\nu_{|\P^{[0;n)}(y)})\big) < b\b^n.\]

Integrating over $\nu_{|\P^{[0;n)}(y)}$ and recalling Equation~(\ref{eq:rel-ent-of-bary}), this gives
\[\rmI_\nu(S^{-n}(\P);\P^{[-N;0)}\,|\,\P^{[0;n)}) < b\b^n \quad \forall n\geq N_0.\]
Therefore the right-hand sum in~(\ref{eq:chain-rule}) is bounded by
\[\sum_{n= 0}^{N_0-1}\rmI_\nu(S^{-n}(\P);\P^{[-N;0)}\,|\,\P^{[0;n)}) + \sum_{n= N_0}^{N-1}b\b^n,\]
which remains bounded as $N\to\infty$ because $\sum_n b\b^n$ is a convergent series.

Finally, if $p \geq 1$, then the definition and standard properties of mutual information give
\begin{eqnarray*}
&&\rmI_\nu(\P^{[-p;N+p)};\P^{[-N-p;p)})\\ &&= \rmH_\nu(\P^{[-N-p;p)}) + \rmH_\nu(\P^{[-p;N+p)}) - \rmH_\nu(\P^{[-N-p,N+p)})\\
&&\leq \rmH_\nu(\P^{[-N-p;-N)}) + \rmH_\nu(\P^{[0;p)}) + \rmH_\nu(\P^{[-p;0)}) + \rmH_\nu(\P^{[N;N+p)})\\
&&\quad + \rmH_\nu(\P^{[-N;0)}) + \rmH_\nu(\P^{[-N;0)}) - \rmH_\nu(\P^{[-N,N)})\\
&&\leq 4p\rmH_\nu(\P) + \rmI_\nu(\P^{[0;N)};\P^{[-N;0)}).
\end{eqnarray*}
\end{proof}

\subsection{H\"{o}lder cocycles over mixing SFTs}\label{subs:Hold-coc-prop}

The structure of a generalized RWRS system seems to depend rather delicately on the cocycle $\s$ which defines it.  This subsection is given to various properties of such cocycles that will be needed later.  The general flavour is of comparing them over long time-scales with Brownian motion.  Such probabilistic limit theorems are a very classical subject in dynamics.  They are all widely-known for simple random walk itself: a suitable reference is~\cite{Rev90}. In our slightly more general setting, much of what we need will be taken from Guivarc'h and Hardy's classic work~\cite{GuiHar88}, which in turn built on older methods of Nagaev for certain Markov chains~\cite{Nagaev57}, among others.  A suitable invariance principle is proved by Bunimovich and Sinai in~\cite{BunSin81}, or can be deduced from the strong invariance principles given in~\cite{PhiSto75}.

Let $(Y,\nu,S)$ be as in the previous subsection. A \textbf{cocycle} over a $(Y,\nu,S)$ will be a real-valued measurable function $\s:Y\to \bbR$.  The term `cocycle' will refer either to this function itself, or to the resulting function $\bbZ\times Y\to \bbR$ defined by the partial sums:
\[(n,y) \mapsto \left\{\begin{array}{ll}\sum_{m=0}^{n-1} \s(S^my)& \quad \hbox{if}\ n \geq 1\\ 0 & \quad \hbox{if}\ n= 0\\ -\sum_{m=n}^{-1} \s(S^my) & \quad \hbox{if}\ n \leq -1.  \end{array} \right.\]
It will often be convenient to denote this value by $\s^y_n$.  We may also think of it as a random function
\[\bbZ\to \bbR:n\mapsto \s^y_n\]
defined on the probability space $(Y,\nu)$.  For a fixed choice of $y$, we will refer to the function $\s^y$ as a \textbf{cocycle-trajectory} to emphasize this point of view.

Mean-zero, H\"older cocycles enjoy a (weak) version of Donsker's Invariance Principle.  It is proved for any dynamical system admitting a suitable Markov partition (including our mixing Gibbs systems $(Y,\nu,S)$) in~\cite{BunSin81}: see their Theorems 2'' and 3.  To formulate it, it will be convenient to introduce the maps $\traj_N:\bbR^\bbZ \to C[0,1]$ defined by
\[\traj_N(\s)(t) := N^{-1/2}\big((Nt - \lfloor Nt\rfloor)\s_{\lceil Nt\rceil} + (\lfloor Nt + 1\rfloor - Nt)\s_{\lfloor Nt \rfloor}\big)\]
(that is, $\traj_N$ rescales $\s$ horizontally by $N^{-1}$ and vertically by $N^{-1/2}$, and then interpolates linearly to produce a function on $[0,1]$).  Similarly, define $\traj_{-N}:\bbR^\bbZ\to C[0,1]$ by
\[\traj_{-N}(\s)(s) := N^{-1/2}\big((-Ns - \lfloor -Ns\rfloor)\s_{\lceil -Ns\rceil} + (\lfloor -Ns + 1\rfloor + Ns)\s_{\lfloor -Ns \rfloor}\big).\]

\begin{thm}[Invariance principle]\label{thm:IP}
If $(\bfY,\s)$ is a mixing Gibbs system and $\s:Y\to \bbR$ is a H\"older cocycle with $\int \s\,\d\nu = 0$, then there is some $c \geq 0$ for which the \textbf{Invariance Principle} holds:
\[\traj_N(\s^y) \stackrel{\rm{law}}{\to} c B \quad \hbox{as} \quad N \to \infty,\]
where the left-hand side is regarded as a random variable on the probability space $(Y,\nu)$, and the right hand side has law $\sfW_{[0,1]}$. Moreover, $c = 0$ if and only if $\s$ is a coboundary over $S$ among H\"older functions. \qed
\end{thm}

In view of its r\^ole in the above theorem, we shall call the constant $c^2$ the \textbf{effective variance} of $(\bfY,\s)$.  Henceforth we will work with non-coboundaries, and will generally normalize so that the effective variance is $1$.

\begin{dfn}
A \textbf{well-distributed pair} is a pair $(\bfY,\s)$ in which $\bfY = (Y,\nu,S)$ is a mixing Gibbs system and $\s:Y\to \bbR$ is a H\"older non-coboundary with $\int \s\,\d\nu = 0$ and with effective variance $1$.
\end{dfn}

The next result gives the optimum rate of convergence to a Gaussian law for the distribution of $\s^y_N$ for a fixed $N$.

\begin{thm}[Berry-Esseen property:~{\cite[Th\'eor\`eme B.IV.2]{GuiHar88}}]\label{thm:BE}
If $(\bfY,\s)$ is well-distributed, then
\begin{eqnarray}\label{eq:fast-CLT}
\sup_{t \in \bbR}\big|\nu\{\s^y_N \leq t\sqrt{N}\} - \rm{N}(-\infty,t)\big| \lesssim_{\bfY,\s} \frac{1}{\sqrt{N}} \quad \forall N\geq 1,
\end{eqnarray}
where $\rm{N}(-\infty,t)$ is the cumulative distribution function of a standard Gaussian. \qed
\end{thm}


\subsection{The Enhanced Invariance Principle}\label{subs:EIP}

In addition to the preceding results on cocycle-distribution, we will also need an enhancement of the conclusion of Theorem~\ref{thm:IP} which simultaneously describes the frequency of visits of $\s^y$ to different regions in $\bbR$.  To formulate this, given $y \in Y$ and a nonempty finite subset $F \subseteq \bbZ$, let
\[\g^y_F := \frac{1}{|F|}\sum_{n\in F}\delta_{\s^y_n}.\]
This is the \textbf{occupation measure of $\s$ over the set of times $F$}.

Let $L^B(u)$, $u \in \bbR$, be Brownian local time at time $1$, regarded as a $C_\rm{c}(\bbR)$-valued random variable on the space $(C_0(0,1],\sfW_{[0,1]})$ (see, for instance,~\cite[Chapter 22]{Kal02}).  Observe that if $\phi:\bbR \to [0,\infty)$ is a mollifier and $\theta$ is a Radon measure on $\bbR$, then the convolution $\phi \star \theta$ may always be understood as the smooth function
\[u \mapsto \int \phi(v-u)\,\theta(\d v).\]

The following is the additional property of a well-distributed pair that we will need.

\begin{dfn}[Ehanced Invariance Principle]\label{dfn:EIP}
The well-distributed pair $(\bfY,\s)$ satisfies the \textbf{Enhanced Invariance Principle} if
there is a mollifier $\phi$ such that
\[\big(\traj_N(\s^y),\,((\phi\star \g^y_{[0;N)})(\sqrt{N} u))_{u \in \bbR}\big) \stackrel{\rm{law}}{\to} (B,L^B)\]
for the product of the uniform and locally uniform topologies on $C_0(0,1]\times C_\rm{c}(\bbR)$. As before, the left- and right-hand sides here are understood as random variables on $(Y,\nu)$ and $(C_0(0,1],\sfW_{[0,1]})$, respectively.
\end{dfn}

I strongly suspect that every well-distributed pair satisfies the Enhanced Invariance Principle, so that the above could instead be introduced as a theorem.  If $\s$ is aperiodic (see~\cite{GuiHar88}), then the above convergence should actually hold for every mollifier $\phi$.  If $\s$ is cohomologous to an $\ell\bbZ$-valued cocycle for some $\ell > 0$, say $\s = \tau + f\circ S - f$, then the occupation measures of $\s^y$ are `adjustments' of those of $\tau^y$, which are supported on $\ell\bbZ$.  However, the above should still hold provided $\phi$ is strictly positive on an interval $[-a,a]$ with $a > \max\{\ell,\|f\|_\infty\}$.  The proofs of these results should be based on the same spectral analysis of the complex Ruelle operator as in~\cite{GuiHar88} or~\cite[Chapter 4]{ParPol90}.  However, as far as I know this result has appeared in the literature only in special cases:
\begin{itemize}
\item In case $(\bfY,\s)$ is the pair of a simple random walk, then it follows from a much stronger classical coupling result between occupation measures of simple random walk and Brownian motion~(\cite[Theorem 10.1]{Rev90}).

\item The generalization to partial sums of Markov chains was recently established by Bromberg and Kosloff~\cite{BromKos14}, building on older results of Borodin~\cite{Bor81}.
\end{itemize}

Thus, our Theorem A is unconditional in either of the above cases.  The first of these covers the classical RWRSs.

I understand that the full generalization (even to the still-broader setting of finite-variance H\"older cocycles on Gibbs-Markov shifts --- see~\cite[Chapter 4]{Aar97},~\cite{AarDen01}) will be the subject of future work by Bromberg.

Similar results for cocycles over general Young towers appear as~\cite[Theorem 9]{DolSzaVar08} and~\cite[Proposition 3]{NanSza12}, but focusing only on finite-dimensional marginals.

The Enhanced Invariance Principle will be used to prove Theorem~\ref{thm:rate}, which evaluates our forthcoming new invariant in the case of generalized RWRS systems.  In fact, it will be needed only for proving the lower-bound half of that Theorem, in Sections~\ref{sec:dCs} and~\ref{sec:lower-bd}.

\begin{rmk}
In recent years there has been considerable interest in generalizing probabilistic limit theorems for ergodic sums to dynamical systems that admit a more general Markov-Gibbs structure or a suitable Young tower~(\cite{You98}): see, for instance,~\cite{AarDen01,Gou05,SzaVar04,DolSzaVar08,Xia09} and the many further references there.  A fairly gentle introduction to the use of Young towers is in~\cite[Chapter 4]{Bala00}, and related material can also be found in the monograph~\cite{HenHer01}.

In suspect that Theorem A can be extended to the study of generalized RWRS systems with base and cocycle given by one of these more general settings.  However, in addition to the Enhanced Invariance Principle, one would need some restriction on the relevant generating partition to obtain an analog of Corollary~\ref{cor:Gibbs-mut-inf}. \fin
\end{rmk}

\begin{rmk}
Rudolph's work in~\cite{Rud88} studies systems satisfying a rather different kind of convergence to Brownian motion: his \textbf{asymptotically Brownian} cocycles $\s$ admit some $\eta > 0$ and a $(\nu,\sfW)$-coupling $\sfP$ such that for $\sfP$-a.e. $(y,B)$ one has
\[|\s^y_n - B_n| = \rm{o}(n^{1/2 - \eta}) \quad \hbox{as}\ n\to\infty.\]
This definition follows Philipp and Stout~\cite{PhiSto75}, who establish that a wide variety of examples are $\eta$-asymptotically Brownian for some $\eta$.  In principle, the existence of such a coupling is significantly stronger than the conclusion of Theorem~\ref{thm:IP}, but it also does not seem to imply the Enhanced Invariance Principle without some additional arguments as in~\cite{BromKos14}, so our assumptions on $\s$ are actually somewhat askew to Rudolph's.  It could be that our Theorem A gives new examples of non-Bernoulli K-automorphisms, not covered by~\cite{Rud88}, but I do not know of any specific systems that fall into this gap. \fin
\end{rmk}

\section{Informal discussion of the RWRS marginal metrics}\label{sec:prelim-discuss}

This section is discursive.  It is not needed for the logic in the rest of the paper, but offers some motivation for the constructions that follow.

The new invariant below is defined in terms of the `marginal' m.p. spaces that arise from a given compact model of a generalized RWRS system.  This section will begin with a sketch of the `marginal' m.p. spaces that arise from the canonical generating partition of a classical RWRS example.

\subsection{Conditioning on the scenery, or the past}\label{subs:prelim-discuss-1}

Let $\a:\{\pm 1\}^\bbZ\times C^\bbZ\to \{\pm 1\}\times C$ be the obvious generating partition for $\rm{RWRS}_\mu$, and let $\rho_N = \a^{[0;N)}_\ast\rho$ be the distribution of the $(\a,N)$-name, as in Subsection~\ref{subs:intro-invt}.  Given a scenery distribution $\mu \in \Pr^S C^\bbZ$, let $\mu_I$ be its marginal on $C^I$ for any $I \subseteq \bbZ$.

Let $d_\a$ be the pseudometric on $\{\pm 1\}^\bbZ\times C^\bbZ$ given by the pullback under $\a$ of the complete metric on $\{\pm 1\}\times C$.  Then the marginal psm.p. spaces given by $(d_\a)^{\rm{RWRS}_\mu}_{[0;N)}$ are likewise pulled back from the finite m.p. spaces
\[\big((\{\pm 1\}\times C)^N,d_{\rm{Ham}},\rho_N\big).\]

We now sketch a provisional description of these m.p. spaces. This is in terms of the $1$-Lipschitz quotient map
\begin{center}
$\phantom{i}$\xymatrix{
\big((\{\pm 1\}\times C)^N,d_{\rm{Ham}},\rho_N\big)\ar[d]\\
\big(\{\pm 1\}^N,d_{\rm{Ham}},\nu_{1/2}^{\otimes N}\big).}
\end{center}
The idea is to describe $\rho_N$ as a lift of $\nu_{1/2}^{\otimes N}$ through this map.

Consider a fixed scenery $c = (c_m)_m\in C^\bbZ$, and define the function $F_c:\{\pm 1\}^N \to C^N$ by
\[F_c((y_n)_{n=0}^{N-1}) = (c_{\s^y_0},c_{\s^y_1},\ldots,c_{\s^y_{N-1}}).\]
Clearly this output depends only on the finite portion $c|_{\s^y_{[0;N)}}$ of $c$.  Let
\[\rho_{N,c} := \int_{\{\pm 1\}^N} \delta_{(y,F_c(y))}\,\nu_{1/2}^{\otimes N}(\d y),\]
the result of lifting $\nu_{1/2}^{\otimes N}$ to the graph of $F_c$.

We can now write the lifted measure $\rho_N$ as the average of the conditional measures of $\rho_N$ given the scenery, and these latter are precisely the graph-supported measures $\rho_{N,c}$:
\begin{eqnarray}\label{eq:cond-meass}
\rho_N = \int_{C^\bbZ}\rho_{N,c}\,\mu(\d c).
\end{eqnarray}
This decomposition of $\rho_N$ is obtained canonically from the process $(\rm{RWRS}_\mu,\a)$: it is the pushforward under $\a^{[0;N)}$ of the disintegration of $\rho$ over the strict past $\a^{(-\infty;0)}$.  This is because
\begin{itemize}
\item on the one hand, the past of the simple random walk is independent of the future,
\item but on the other, simple random walk is recurrent, so the past of the whole process a.s. determines the scenery exactly.
\end{itemize}

\subsection{Separating the conditional measures}

We can now describe the overall strategy of the proof of non-Bernoullicity in~\cite{Kal82}.  The heart of Kalikow's work is to prove that there are arbitrarily large $N$ for which the following holds.

\begin{thm}[{\cite{Kal82}}]\label{thm:Kal2}
For a fixed sequence of walk-steps $y \in \{\pm 1\}^N$ and a fixed scenery $c' \in C^\bbZ$, it holds for most $c \in C^\bbZ$ that
\[d_{\rm{Ham}}\big((y,F_c(y)),\ \rm{spt}\,\rho_{N,c'}\big) > 10^{-20} N,\]
where `most' means `with high probability as $N\to\infty$'. \qed
\end{thm}

(Indeed, Kalikow actually proves this with Feldman's weaker $\ol{\rm{f}}$-metric in place of $d_{\rm{Ham}}$.)

By Fubini's Theorem, the above implies that for $(\mu\otimes \mu)$-most pairs $(c,c')$ there is a subset $W_{c,c'} \subseteq \{\pm 1\}^N$ such that
\[\nu_{1/2}^{\otimes N}(W_{c,c'}) = 1 - \rm{o}(1) \quad \hbox{and} \quad d_{\rm{Ham}}\big((\rm{id},F_c)(W_{c,c'}),\rm{spt}\,\rho_{N,c'}\big) > 10^{-20}N.\]
This implies that a typical pair of conditional measures $\rho_{N,c}$, $\rho_{N,c'}$ are $\Omega(N)$-separated in the Wasserstein metric associated to $d_{\rm{Ham}}$, and hence that $\rm{RWRS}_\mu$ does not satisfy the Very Weak Bernoulli condition. 

An alternative description of this reasoning, more intrinsic to the metric geometry of $((\{\pm 1\}\times C)^N,d_{\rm{Ham}},\rho_N)$, uses a different characterization of Bernoullicity in terms of measure concentration.

\begin{dfn}[(Almost) Exponential measure concentration]
Let $(X_n,d_n,\mu_n)$ be a sequence of compact psm.p. spaces.  The sequence exhibits \textbf{exponential measure concentration} if for every $\delta > 0$ there is a $c > 0$ such that for any Borel set $U \subseteq X_n$ one has
\[\mu_n(U) \geq \rme^{-c n} \quad \Longrightarrow \quad \mu_n(B^{d_n}_\delta(U)) \geq 1 - \rme^{-c n}\]
for all sufficiently large $n$.  The constant $c$ is the \textbf{exponential rate} of this concentration at distance $\delta$.

The sequence exhibits \textbf{almost exponential measure concentration} if there is a sequence of Borel subsets $X_n' \subseteq X_n$ such that $\mu_n(X'_n) \to 1$ and $(X_n',d_n,(\mu_n)_{|X_n'})$ exhibits exponential measure concentration.
\end{dfn}

\begin{thm}[Exponential measure concentration in Bernoulli shifts]\label{thm:Bern-shift-almost-exp-conc}
Let $\bfX = (X,\mu,T)$ be a p.-p. system of entropy $h < \infty$, and let $\P$ be a finite generating partition of $\bfX$.  Then $\bfX$ is Bernoulli if and only if the sequence of psm.p. spaces $(X,N^{-1}d^\P_{[0;N)},\mu)$ exhibits almost exponential measure concentration. \qed
\end{thm}

This is essentially the same as~\cite[Theorem III.4.3]{Shi96}, or can be quickly deduced from the implications proved in~\cite[Chapter 5]{KalMcC10}.  It was introduced explicitly into ergodic theory by Marton and Shields in~\cite{MarShi94}, where it was called the `blowing-up property'.  It is, however, also very close to Thouvenot's notion of `extremality', presented in~\cite[Definition 6.3]{Tho02} but devised much earlier.  These properties are now properly viewed as instances of the general phenomenon of concentration of measure: see, for instance,~\cite{Led--book} or~\cite[Chapter 3$\frac{1}{2}$]{Gro01} for an introduction.

Returning to RWRS$_\mu$, now fix some very small $\eps > 0$. Since simple random walk is diffusive, we may pick some large distance-cutoff $R \in \bbN$ so that the set
\[Y_N := \{y \in \{\pm 1\}^N\,|\ \s^y_{[0;N)}\subseteq [-R\sqrt{N};R\sqrt{N}]\}\]
has $\nu_{1/2}^{\otimes N}(Y_N) > 1 - \eps$ for all sufficiently large $N$.  Let $Z_N := Y_N\times C^N$, so
\[\rho_N(Z_N) = \rho_{N,c}(Z_N) = \nu_{1/2}^{\otimes N}(Y_N) > 1 - \eps,\]
because $\rho_N$ and each $\rho_{N,c}$ is a lift of $\nu_{1/2}^{\otimes N}$.

In addition, Theorem~\ref{thm:SM} gives subsets $X^{\rm{SM}}_{I,\eps} \subseteq C^I$ for each bounded discrete interval $I \subseteq \bbZ$ such that
\[|X^{\rm{SM}}_{I,\eps}| \leq \exp((\rmh(\mu,S) + \eps)|I|) \quad \hbox{and} \quad \mu_I(X^{\rm{SM}}_{I,\eps}) > 1 - \rm{o}(1) \ \hbox{as}\ |I| \to\infty.\]
Let $X_N := X^{\rm{SM}}_{[-R\sqrt{N};R\sqrt{N}],\eps}$.

If $y \in Y_N$, then $F_c(y)$ depends only on the portion $c|_{[-R\sqrt{N};R\sqrt{N}]}$, and therefore $(\rho_{N,c})_{|Z_N}$ depends only on $c|_{[-R\sqrt{N};R\sqrt{N}]}$.  With some slight abuse of notation, it follows that
\begin{multline}\label{eq:cond-meass2}
\rho_N \approx_\eps (\rho_N)_{|Z_N} = \int_{C^{[-R\sqrt{N};R\sqrt{N}]}} (\rho_{N,c})_{|Z_N}\,\mu_{[-R\sqrt{N};R\sqrt{N}]}(\d c)\\
\approx_\eps \int_{X_N} (\rho_{N,c})_{|Z_N}\,\mu_{[-R\sqrt{N};R\sqrt{N}]}(\d c)
\end{multline}
for sufficiently large $N$.  Thus, most of the mass in the decomposition~(\ref{eq:cond-meass}) is a convex combination of $|X_N| \leq \exp(2R(\rmh(\mu,S) + \eps)\sqrt{N})$ different measures supported on the graphs of the functions $F_c|_{Y_N}$.

Now, Kalikow's conclusion in Theorem~\ref{thm:Kal2} may easily be adapted to see that most pairs of the measures in the coarsened decomposition~(\ref{eq:cond-meass2}) are also well-separated in the Wasserstein metric.  Since there are only $\exp(\rm{O}(\sqrt{N}))$ of these measures, an easy argument now shows that this precludes $\rho_N$ from exhibiting almost exponential measure concentration.

\subsection{Significance for approximate recovery of the scenery}\label{subs:approx-recov}

Our work below will re-use the main ideas from Kalikow's proof of Theorem~\ref{thm:Kal2}, but to a different end.  As discussed in the Introduction, the scenery entropy $\rmh(\mu,S)$ should appear in estimates on the mutual information $\rmI_\rho(\a^{[-N;0)};\a^{[0;N)})$.  However, we need to make this quantity more robust, by asking after the information about a pair $(y,c)$ that can be recovered if one knows the output strings $\a^{[-N;0)}(y,c)$ and $\a^{[0;N)}(y,c)$ only approximately.

We still expect this information to reside in that part of the scenery visited by both of the trajectories $\s^y_{[-N;0)}$ and $\s^y_{[0;N)}$, so the heart of the matter is now the ability to recover $c|_{\s^y_{[0;N)}}$ approximately if one only knows
\[\a^{[0;N)}(y,c) = (y,F_c(y)).\]
approximately.

This is difficult, because the map $(c,y) \mapsto F_c(y)$ can contract the relevant Hamming distances very greatly.

\begin{ex}\label{ex:hard-to-recover}
If $y = (y_n)_{n \in [0;N)}$ and $y' = (y'_n)_{n \in [0;N)}$ are chosen so that $y_0 = 1$, $y'_0 = -1$, but $y_n = y'_n$ for all $n\in [1;N)$, then 
\[d_{\rm{Ham}}(y,y') = 1,\]
but
\[\s^y_n = \s^{y'}_n + 2 \quad \forall n \in [1;N).\]
Therefore, if $c \in C^\bbZ$ and $c' := S^2c$, then $F_c(y)$ and $F_{c'}(y')$ agree in every coordinate in $[1;N)$.  Thus
\[d_{\rm{Ham}}\big((y,F_c(y)),(y',F_{c'}(y')\big) = 1,\]
even though $c$ and $c'$ could be very far apart according to the relevant Hamming metric.  More subtle examples of this phenomenon are described in~\cite{Lin99}. \fin
\end{ex}

Therefore, if one knows $y|_{[0;N)}$ only up to a small Hamming-metric error, it could happen that $c|_{\s^y_{[0;N)}}$ cannot be recovered up to a small Hamming error from the output-string $F_c(y)$.  In order to work around this problem, we will need to set up a different, weaker sense in which approximate knowledge of $(y,F_c(y))$ constrains the possible choices of $c$, which is still strong enough that we obtain the same leading-order asymptotics as for true mutual information.

In view of the above example, a natural conjecture in this direction would be that, after excluding a small-probability set of `bad' trajectories $y$, it holds that
\begin{multline*}
d_{\rm{Ham}}\big((y,F_c(y)),(y',F_{c'}(y'))\big) \approx 0\\
\Longrightarrow \quad \frac{|\s^y_{[0;N)} \triangle \s^{y'}_{[0;N)}|}{\sqrt{N}} \approx 0 \quad \hbox{and} \quad \overline{\rm{f}}_{\s^y_{[0;N)}}(c,c') \approx 0,
\end{multline*}
where $\overline{\rm{f}}_I$ is Feldman's metric over a bounded discrete interval $I$ from~\cite{Feld76}.   Unfortunately, I do not know how to prove this.  Instead, we will work with an even weaker (and significantly more complicated) notion of similarity between sceneries.  Setting up this notion and then proving the analog of the above implication will be the most substantial part of our work, and will occupy most of Sections~\ref{sec:dCs} and~\ref{sec:lower-bd}.

\begin{rmk}
The above discussion is suggestive of a link with the `scenery reconstruction problem', which asks whether the entire scenery $c$ can eventually be reconstructed from only the output string $(c_0,c_{y_0},c_{\s^y_1},\ldots)$, with probability $1$ in the choice of $(y_0,y_1,\ldots)$.  Much is known about that problem, but the methods do not seem well-adapted to the problem of `approximate reconstruction' described above.  Essentially, this is because in those works the scenery is reconstructed only very `slowly': that is, the patch $c|_{[-m;m]}$ can be recovered with high probability only once one has seen $(c_0,c_{y_0},c_{\s^y_1},\ldots,c_{\s^y_M})$ for some $M \gg m^2$.  The best control on the necessary $M$ is some high-degree polynomial in $m$, obtained by Matzinger and Rolles in~\cite{MatRol03}. They conjecture that it suffices to use $M \ll m^{2 + \eps}$ for any $\eps > 0$, but this would still be too large for our purposes.  Nevertheless, it would be interesting to know of any conceptual intersection between their methods and ours.

More background on scenery reconstruction can be found in Section 3 of the survey~\cite{denHSte06}, and in the dedicated surveys~\cite{MatLember--overview} and~\cite{Kes98}. \fin
\end{rmk}

\begin{rmk}
Another proposal for an invariant of systems that should capture something like the above sequence of mutual informations is Vershik's `secondary entropy', formulated in~\cite[Section 7]{Ver00}.  Essentially, it amounts to quantifying the failure of the Very Weak Beroulli property of an abstract process $(\bfZ,\R)$ in terms of packings numbers within the space of future-name distributions.  However, I am not aware that this quantity has been shown to be invariant under isomorphisms of processes, and I also do not see how to estimate it accurately enough for RWRS processes.  Nevertheless, Vershik's idea was a key motivation for the invariant that we define below. \fin
\end{rmk}

\section{The new invariant}\label{sec:RRDr}

This section is largely concerned with general metric or pseudometric spaces, or general compact model p.-p. systems $\bfX = (X,d^X,\mu,T)$.  For these systems, the key to our new invariant will be to consider not just the asymptotic behaviour of the sequence of metrics $d_{[0;N)}^\bfX$ on $(X,\mu)$, but that of the sequence of \emph{pairs} of metrics
\[d_{[-N;0)}^\bfX \quad \hbox{and} \quad d_{[0;N)}^\bfX.\]

\subsection{Pair-metric spaces and bi-neighbourhoods}

\begin{dfn}
A \textbf{pair-metric space} is a triple $(X,d_1,d_2)$ in which $d_1$ and $d_2$ are two compact metrics generating the same topology on $X$.  A \textbf{pair-m.m.} (resp. \textbf{pair-m.p.}) \textbf{space} is a quadruple $(X,d_1,d_2,\mu)$ consisting of a pair-metric space and a finite Radon (resp. Radon probability) measure on $X$.
\end{dfn}

Note that we always assume compactness without mentioning it in the nomenclature.  It will be important that $d_1$ and $d_2$ do not generate different topologies.

It will be convenient to allow also pairs of pseudometrics.

\begin{dfn}
A \textbf{pair-pseudometric space} is a triple $(X,d_1,d_2)$ in which $X$ is a standard Borel space and $d_1$ and $d_2$ are two totally bounded Borel pseudometrics $X\times X\to [0,\infty)$.  A \textbf{pair-psm.m.} (resp. \textbf{pair-psm.p.}) \textbf{space} is a quadruple $(X,d_1,d_2,\mu)$ consisting of a pair-psuedometric space and a finite Radon (resp. Radon probability)  measure on $X$.
\end{dfn}

Note again that we always assume total boundedness without mentioning it in the nomenclature.

\begin{dfn}
If $(X,d^X_1,d^X_2)$ and $(Y,d^Y_1,d^Y_2)$ are pair-pseudometric spaces and $c,L > 0$, then a map $f:X\to Y$ is \textbf{$L$-pair-Lipschitz} (resp. \textbf{$c$-almost $L$-pair-Lipschitz}) if it is $L$-Lipschitz (resp. $c$-almost $L$-Lipschitz) as a map $(X,d^X_i) \to (Y,d^Y_i)$ for $i = 1,2$.
\end{dfn}

\begin{ex}\label{ex:simple-wedge-diag}
Let $X = [0,1]^3$, and let
\[d_1((x_1,x_2,x_3),(x'_1,x'_2,x'_3)) := |x_1 - x_1'| + |x_2 - x_2'|\]
and
\[d_2((x_1,x_2,x_3),(x'_1,x'_2,x'_3)) := |x_2 - x_2'| + |x_3 - x_3'|.\]
Then $(X,d_1,d_2)$ is a pair-pseudometric space in which neither $d_1$ nor $d_2$ is a metric. \fin
\end{ex}

Given a p.-p. system $(X,\mu,T)$ and a totally bounded Borel pseudometric $d^X$ on $X$, we will consider the sequence of pair-psm.p. spaces
\[(X,d_{[-N;0)}^\bfX,d^\bfX_{[0;N)},\mu), \quad N\geq 1.\]
These are referred to as the \textbf{marginal pair-psm.p. spaces} of $(X,d^X,\mu,T)$.

Our new invariant will involve some quantification of how much information is `robust' under both of the pseudometrics $d^\bfX_{[-N;0)}$ and $d^\bfX_{[0;N)}$ on $X$.  This will be made precise via the following notion.

\begin{dfn}
Let $(X,d_1,d_2)$ be a pair-pseudometric space and $\delta \geq 0$.  The \textbf{$\delta$-bi-neighbourhood} in $(X,d_1,d_2)$ around a point $x \in X$ is the set
\[B^{d_2}_\delta(B^{d_1}_\delta(x)).\]
A pair of points $(x,y) \in X^2$ is \textbf{$\delta$-bi-separated} in $(X,d_1,d_2)$ if
\[B^{d_2}_\delta(B^{d_1}_\delta(x))\cap B^{d_2}_\delta(B^{d_2}_\delta(y)) = \emptyset.\]
\end{dfn}

These definitions are not symmetrical in $d_1$ and $d_2$; though possibly disappointing, this will not matter in the sequel.

Give a subset $F \subseteq X$, its $\delta$-bi-neighbourhood is
\[B_\delta^{d_2}(B_\delta^{d_1}(F)) = \bigcup_{x \in F}B_\delta^{d_2}(B_\delta^{d_1}(x)).\]

The property of bi-separation will not be used much below, but it gives some useful first intuition for bi-neighbourhoods.  Explicitly, $x,y \in X$ are $\delta$-bi-separated if for any $x',y',z \in X$, the following four inequalities cannot all hold:
\[d_1(x,x') \leq \delta, \quad d_2(x',z) \leq \delta, \quad d_2(z,y') \leq \delta \quad \hbox{and} \quad d_1(y',y) \leq \delta.\]
Thus, this asserts that one cannot move from $x$ to $y$ by taking a jump which is very small for the metric $d_1$, then two jumps which are very small for $d_2$, then another jump which is very small for $d_1$.

Clearly if $(x,y)$ is $\delta$-bi-separated, then one must have $d_1(x,y) \geq 2\delta$ and also $d_2(x,y) \geq 2\delta$.  However, the reverse of this implication need not hold, even approximately.

\begin{ex}
Recall the pair-pseudometric space in Example~\ref{ex:simple-wedge-diag}, and let $x = (x_1,x_2,x_3)$ and $x' = (x'_1,x'_2,x'_3)$ be points of $[0,1]^3$.  Then
\[B^{d_2}_\delta(B^{d_1}_\delta(x)) := \{(x'_1,x'_2,x'_3)\,|\ |x_2 - x_2'| \leq 2\delta\},\]
and so $x,y$ are $\delta$-bi-separated if and only if
\[|x_2 - y_2| \geq 4\delta.\]
In particular, the points $(0,1,0)$ and $(1,1,1)$ are far apart according to both $d_1$ and $d_2$, but are not $\delta$-bi-separated for any $\delta > 0$. \fin
\end{ex}

The following is now the obvious ananlog of~(\ref{eq:cov}) for bi-neighbourhoods.

\begin{dfn}
Let $(X,d_1,d_2,\mu)$ be a pair-psm.m. space.  For $a,\delta > 0$, the \textbf{$a$-partial $\delta$-bi-covering number} is
\[\bicov_a((X,d_1,d_2,\mu),\delta) := \min\{|F|\,|\ F \subseteq X,\ \mu(B_\delta^{d_2}(B_\delta^{d_1}(F))) > a\}.\]

We also define simply
\[\bicov((X,d_1,d_2),\delta):= \min\{|F|\,|\ F\subseteq X,\ B_\delta^{d_2}(B_\delta^{d_1}(F)) = X\},\]
by analogy with classical covering numbers.
\end{dfn}

\begin{rmk}
Similarly, there is an obvious definition of $\bipack((X,d_1,d_2),\delta)$ in terms of bi-separation.  However, unlike for classical covering and packing numbers, I believe there are no simple relations between $\bicov$ and $\bipack$.  In essence, this is because the estimates relating covering and packing numbers rely on the inclusion
\[B_\delta(B_\delta(x)) \subseteq B_{2\delta}(x) \quad \forall x,\delta.\]
However, no corresponding inclusion need hold in the pair-pseudometric setting: given any $\delta \ll \delta'$, one can easily concoct examples in which $B_\delta^{d_2}(B_\delta^{d_1}(B_\delta^{d_2}(B_\delta^{d_1}(x))))$ is much larger than $B_{\delta'}^{d_2}(B_{\delta'}^{d_1}(x))$.

In fact, one could develop most of the rest of the present paper using bi-packing instead of bi-covering numbers, and I believe they would still serve to distinguish RWRS systems.  Bi-covering numbers seem to require slightly simpler estimates, so we focus on them.  However, it would be interesting to know of examples of systems for which these two different quantities give genuinely different invariants, perhaps with one behaving trivially and the other non-trivially. \fin
\end{rmk}

Now suppose that $(X,d^X_1,d^X_2)$ and $(Y,d^Y_1,d^Y_2)$ are pair-pseudometric spaces, that $c,L > 0$, and that $\Phi:X\to Y$ is a $c$-almost $L$-pair-Lipschitz map.  In this case, one has the obvious inclusion
\[\Phi\big(B^{d^X_2}_\delta(B^{d^X_1}_\delta(x))\big) \subseteq B^{d^Y_2}_{L\delta+c}(B^{d^Y_1}_{L\delta+c}(\Phi(x))) \quad \forall x \in X,\]
and hence also
\begin{eqnarray}\label{eq:bihood-under-pair-Lip}
\Phi\big(B^{d^X_2}_\delta(B^{d^X_1}_\delta(F))\big) \subseteq B^{d^Y_2}_{L\delta+c}(B^{d^Y_1}_{L\delta+c}(\Phi(F))) \quad \forall F \subseteq X.
\end{eqnarray}
This leads immediately to the following.

\begin{lem}\label{lem:bicov-under-pair-Lip}
Let $c,L,a,\delta > 0$. Suppose that $(X,d^X_1,d^X_2)$ and $(Y,d^Y_1,d^Y_2)$ are pair-pseudometric spaces, that $\mu$ is a finite Borel measure on $X$, and that $\Phi:X\to Y$ is a $c$-almost $L$-pair-Lipschitz map.  Then
\[\bicov_a((X,d^X_1,d^X_2,\mu),\delta) \geq \bicov_a((Y,d_1^Y,d_2^Y,\Phi_\ast\mu),L\delta + c).\]
\end{lem}

\begin{proof}
If $F \subseteq X$, then~(\ref{eq:bihood-under-pair-Lip}) implies
\[\Phi_\ast\mu\big(B_{L\delta+c}^{d_2^Y}(B_{L\delta+c}^{d_1^Y}(\Phi(F)))\big) \geq \Phi_\ast\mu\big(\Phi\big(B_\delta^{d_2^X}(B_\delta^{d_1^X}(F))\big)\big) \geq \mu(B_\delta^{d_2^X}(B_\delta^{d_1^X}(F))).\]
\end{proof}

%
%

\subsection{Passing to subsets}\label{subs:competition}

By analogy with~(\ref{eq:KS-dfn}), a natural place to look for a new invariant of a p.-p. system $(X,\mu,T)$ would be in the asymptotic behaviour of
\begin{eqnarray}\label{eq:simple-from-bicov}
\bicov_{1 - \eps}((X,d^\bfX_{[-N;0)},d^\bfX_{[0;N)},\mu),\delta N)
\end{eqnarray}
as $N\to\infty$ for a suitable choice of (pseudo)metric $d$ on $X$, possibly then also sending $\eps \downarrow 0$ and $\delta \downarrow 0$ in the right order.

The arguments below can easily be adapted to show that one does obtain isomorphism-invariants this way.  However, as far as I know, they do not achieve the purpose of distinguishing RWRS systems.  Instead, our new invariant will be obtained from the bi-covering numbers of various \emph{subspaces} of $(X,d^\bfX_{[-N;0)},d^\bfX_{[0;N)},\mu)$.

The need to pass to subsets in a controlled way will be discussed in more detail shortly.  There are surely many ways to do this which will lead to a more refined invariant.  The procedure of this subsection is the simplest I have found to work, but is by no means canonical.

The key next point to emphasize is that bi-neighbourhoods can behave much more subtly than ordinary neighbourhoods under passing to subspaces.  If $x \in Y \subseteq X$, then the $\delta$-bi-neighbourhood of $x$ in the pair-metric subspace $(Y,d_1,d_2)$ is
\begin{eqnarray}\label{eq:subset-bad}
Y\cap B_\delta^{d_2}(Y\cap B_\delta^{d_1}(x)),
\end{eqnarray}
and this may be much smaller than just $Y\cap B_\delta^{d_2}(B_\delta^{d_1}(x))$.  Crucially, this means that bi-covering numbers can \emph{in}crease under passing to subsets.

\begin{ex}\label{ex:simple-again}
Let $(X,d_1,d_2)$ be as in Example~\ref{ex:simple-wedge-diag}, and let $U := \{(x,0,x)\,|\ x \in [0,1]\} \subset [0,1]^3$.  One has
\[d_i((x,0,x),(y,0,y)) = |x-y| \quad \hbox{for both}\ i=1,2,\]
and so within the pair-pseudometric space $(U,d_1,d_2)$, the $\delta$-bi-neighbourhood of $(x,0,x)$ is precisely
\[\{(y,0,y)\,|\ |x - y| \leq 2\delta\}.\]

By contrast, letting $V := [0,1]\times \{0\}\times [0,1]$, for any $(x,0,x'),(y,0,y') \in V$ one has
\[d_1((x,0,x'),(x,0,y')) = d_2((x,0,y'),(y,0,y')) = 0,\]
and so for every point of $V$, its $\delta$-bi-neighbourhood in $(V,d_1,d_2)$ is the whole of $V$, for any $\delta > 0$.

Therefore, even though $U\subseteq V$, we obtain
\[\bicov((U,d_1,d_2),\delta) \sim (2\delta)^{-1} \quad \hbox{whereas} \quad \bicov((V,d_1,d_2),\delta) = 1 \quad \forall \delta > 0.\]
\fin
\end{ex}

Now consider some further parameters $\a \in [1,\infty)$ and $\k > \k' > 0$.

\begin{dfn}\label{dfn:big-bicov}
For a pair-psm.p. space $(X,d_1,d_2,\mu)$, $\a \in [1,\infty)$, $\delta > 0$, and $\k > \k' > 0$ we define the \textbf{bi-covering number profile} by
\[\BICOV_{\a,\k,\k',\delta}(X,d_1,d_2,\mu) := \min_{\|\d\mu'/\d\mu\|_\infty \leq \a}\ \max_{\hbox{\scriptsize{$\begin{array}{c}U\subseteq X\\ \mu'(U) \geq \k\end{array}$}}}\bicov_{\k'}((U,d_1,d_2,\mu'),\delta).\]
\end{dfn}

This definition is quite involved, and clearly warrants some discussion.

An intuitive way to think about Definition~\ref{dfn:big-bicov} is in terms of a competition between two players, Max-er and Min-er.  Given a compact pair-psm.p. space $(X,d_1,d_2,\mu)$, Max-er and Min-er compete to produce a subset $U \subseteq X$.  Max-er's goal to to maximize the resulting value of $\bicov_{\k'}((U,d_1,d_2,\mu'),\delta)$ for some new auxiliary measure $\mu'$, and Min-er's goal is to minimize it.  They play as follows\footnote{Note that because the number of turns is limited to two, this is not a `game' in the fully-fledged mathematical sense.}:
\begin{enumerate}
\item First, Min-er may choose any new measure $\mu' \in \Pr X$, provided
\begin{eqnarray}\label{eq:bdd-RN-deriv}
\Big\|\frac{\d\mu'}{\d\mu}\Big\|_\infty \leq \a.
\end{eqnarray}
The natural choice to imagine here is $\mu' := \mu_{|A}$ for some $A \subseteq X$ with $\mu(A) \geq \a^{-1}$.  We allow the relaxation to arbitrary measures satisfying~(\ref{eq:bdd-RN-deriv}) because it makes some later arguments smoother (and we work with $\|\cdot\|_\infty$, rather than any other norm, also as a matter of convenience).
\item Second, Max-er chooses a subset $U\subseteq X$ for which $\mu'(U) \geq \k$.  For instance, if $\mu' = \mu_{|A}$, then this is equivalent to $\mu(U\cap A) \geq \k\mu(A)$.  So this choice by Max-er is constrained by Min-er's earlier choice of $\mu'$: for instance, for any subset $A$ of measure at least $\a^{-1}$, Min-er is able to force Max-er to include a not-too-small piece of that subset in her choice of $U$.
\end{enumerate}

(Implicitly, there is a third minimization turn implied by the definition of $\bicov$, in which Min-er chooses a subset of $U$ of measure at least $\k'$ that can be covered most efficiently by bi-neighbourhoods.  The flexibility of this last choice is also important in case Max-er's choice of $U$ contains some unwieldy subset of measure less than $\k - \k'$, since Min-er is then not required to cover that portion of $U$.)

Let us motivate this idea by sketching how it repairs certain defects of its simpler relative in~(\ref{eq:simple-from-bicov}).

As suggested above,~(\ref{eq:simple-from-bicov}) can be used to give an isomorphism invariant of p.-p. systems. The problem seems to be that it is very difficult to compute, for two distinct reasons.
\begin{itemize}
\item Firstly, $X$ could contain small subsets that have a heavy `pathological' effect on the bi-covering numbers, in that they either decrease or increase them drastically.  A drastic decrease is easy to visualize: imagine removing a tendril of fairly small measure which is long and thin for both $d_1$ and $d_2$.  This possibility would already be dealt with by our requiring only a partial covering of $X$, up to a certain measure.  However, as seen in Example~\ref{ex:simple-again}, removing a subset can also \emph{increase} bi-covering numbers, and I do not know how to rule out the possibility that removing a very small subset is responsible for a very large increase.  We need a definition that is stable under this possibility as well.

Definition~\ref{dfn:big-bicov} overcomes this latter problem in the second turn of the competition above: it is in Max-er's interest to choose a subset that removes any `decreasing pathology'.

\item Secondly, even if one is allowed to trim away pathologies of both the kinds above, the pair-psm.p. spaces
\[(X,d^\bfX_{[-N;0)},d^\bfX_{[0;N)},\mu)\]
can still be quite `inhomogeneous': they can contain various large-measure subsets that exhibit a broad spectrum of different asymptotics for their bi-covering numbers.  It could be difficult to work out how these different subsets contribute to an overall bi-covering number.  This will be discussed further for the particular skew-products of Theorem A in Section~\ref{sec:prelim-discuss}.

To overcome this problem, Definition~\ref{dfn:big-bicov} allows Min-er a first turn in which he is allowed to restrict attention to any not-too-small subset -- this should result in him cutting away the `bigger part' of $X$ from the point of view of $\bicov$.
\end{itemize}

Crucially, the formulation of Definition~\ref{dfn:big-bicov} in terms of repeated optimization --- that is, as a competition --- gives a way to excise these problems that is intrinsic to the pair-psm.p.-space structure.  This intrinsicality of $\rm{BICOV}$ will be key to its giving an isomorphism invariant of systems.

Understanding Definition~\ref{dfn:big-bicov} in terms of a competition will also help to guide us through the proofs of estimates on $\rm{BICOV}$ values later in the paper.  To prove an upper bound, one imagines playing as Min-er with Max-er playing optimally, and to prove a lower bound, one imagines the reverse.

\subsection{The new invariant}

To define our new invariant in terms of $\BICOV$, it is natural to focus on the metrics appearing in compact models.  Theorem~\ref{thm:vara} gives such a model for any system, and we will soon show that two isomorphic compact models give the same invariant up to some natural equivalence.  (However, it is sometimes convenient to use other pseudometrics on $X$ for some comparison with the metric in a compact model, hence the decision to include general pseudometrics above.)

The marginal pair-m.p. spaces of different compact systems are related using the following extension of Lusin's Theorem.

\begin{lem}\label{lem:marg-maps-approx}
Let $\Phi:(X,d^X,\mu,T)\to (Y,d^Y,\nu,S)$ be a Borel factor map of compact model p.-p. systems, and let $\eps > 0$.  Then there is an $L < \infty$ such that for all sufficiently large $N \in \bbN$ there is a compact subset $X_0 \subseteq X$ with $\mu(X_0) > 1 - \eps$ and such that $\Phi|X_0$ is continuous and $(\eps N)$-almost $L$-pair-Lipschitz from $(X_0,d^\bfX_{[-N;0)},d^\bfX_{[0;N)})$ to $(Y,d^\bfY_{[-N;0)},d^\bfY_{[0;N)})$.
\end{lem}

\begin{proof}
Let $\eps_1 := \eps/(2\rm{diam}(Y,d^Y) + 1)$.  By Lusin's Theorem, there is a compact subset $X_1 \subseteq X$ such that $\mu(X_1) > 1 - \eps_1^2$ and $\Phi|X_1$ is continuous.  That continuity implies that $\Phi|X_1$ is also $\eps_1$-almost $L$-Lipschitz for some $L < \infty$.

Now let $N \in \bbN$, and for each $x \in X$ let 
\[I_{1,x} := \{n \in [-N;0)\,|\ T^n x\not\in X_1\} \quad \hbox{and} \quad I_{2,x} := \{n \in [0;N)\,|\ T^nx \not\in X_1\}.\]
Let
\[X_2 := \big\{x \in X\,\big|\ |I_{1,x}\cup I_{2,x}| \leq 2\eps_1 N \big\}.\]
Since
\[\int_X |I_{1,x}\cup I_{2,x}|\,\mu(\d x) = 2N\mu(X\setminus X_1) < 2\eps_1^2 N,\]
Markov's Inequality implies $\mu(X\setminus X_2) < \eps_1 \leq \eps$.

Now suppose that $x,x' \in X_2$.  Then
\begin{eqnarray*}
&&d^\bfY_{[-N;0)}(\Phi(x),\Phi(x'))\\ &&= \sum_{n=-N}^{-1} d^Y(\Phi(T^nx),\Phi(T^nx'))\\
&&\leq |I_{1,x} \cup I_{1,x'}|\rm{diam}(Y,d^Y) + \sum_{n \in [-N;0)\setminus I_{1,x}\cup I_{1,x'}} d^Y(\Phi(T^nx),\Phi(T^nx'))\\
&&\leq 2\eps_1 N \rm{diam}(Y,d^Y) + \eps_1 N + L\sum_{n=-N}^{-1}d^X(T^nx ,T^nx')\\
&&= \eps N + Ld^\bfX_{[-N;0)}(x,x'),
\end{eqnarray*}
showing that $\Phi|X_2$ is $(\eps N)$-almost $L$-Lipschitz.  The analogous estimate holds also for $d^\bfX_{[0;N)}$ and $d^\bfY_{[0;N)}$.

Finally, another appeal to Lusin's Theorem gives a further compact subset $X_0 \subseteq X_2$ such that $\Phi|X_0$ is continuous, $\mu(X_0) > 1 - \eps$, and the above almost Lipschitz bounds must still hold.
\end{proof}

\begin{prop}\label{prop:big-bicov-factor}
Suppose that $\Phi:(X,d^X,\mu,T)\to (Y,d^Y,\nu,S)$ is a Borel factor map of compact model p.-p. systems, and that $\a > 1$, $\delta > 0$ and $\k > \k' > 0$.  Then for every $\a_1 \in [1,\a)$ there are $\k_1 \in (\k',\k)$ and $\delta_1 \in (0,\delta)$ such that
\[\BICOV_{\a_1,\k_1,\k',\delta_1 N}(X,d^\bfX_{[-N;0)},d^\bfX_{[0;N)},\mu)\\ \geq \BICOV_{\a,\k,\k',\delta N}(Y,d^\bfY_{[-N;0)},d^\bfY_{[0;N)},\nu)\]
for all sufficiently large $N$.
\end{prop}

\begin{proof}
The parameters $\a_1 < \a$, $\k' < \k$ and $\delta$ are fixed.  Choose $\eps$ so small that one has
\[\a\eps < 1, \quad \a_1 < (1 - \a\eps)\a, \quad \eps < \delta/2, \quad \hbox{and} \quad \k_1 := (1 - \a\eps)\k \in (\k',\k).\]
Let $L < \infty$ be given by Lemma~\ref{lem:marg-maps-approx} for this $\eps$, and now choose $\delta_1$ so small that $L\delta_1 + \eps < \delta$.

Having chosen these parameters, and given any $N$ which is sufficiently large for the conclusion of Lemma~\ref{lem:marg-maps-approx}, set
\[m := \BICOV_{\a_1,\k_1,\k',\delta_1 N}(X,d^\bfX_{[-N;0)},d^\bfX_{[0;N)},\mu).\]
We will imagine playing as Min-er in the competition described in Subsection~\ref{subs:competition} with the input space $(Y,d^\bfY_{[-N;0)},d^\bfY_{[0;N)},\nu)$.

By the definition of $m$, we may choose some $\mu' \in \Pr X$ with $\|\d\mu'/\d\mu\|_\infty \leq \a_1$, and with the property that for every $U \subseteq X$ one has
\begin{eqnarray}\label{eq:bicov-bound}
\mu'(U) \geq \k_1 \quad \Longrightarrow \quad \bicov_{\k'}((U,d^\bfX_{[-N;0)},d^\bfX_{[0;N)},\mu'),\delta_1 N) \leq m.
\end{eqnarray}

Now, recalling our choice of $\eps$, let $X_0 \subseteq X$ be given by Lemma~\ref{lem:marg-maps-approx}, and then let $\mu'' := \mu'_{|X_0}$.  Since $\mu'(X\setminus X_0) \leq \a\mu(X\setminus X_0) < \a\eps$, we have
\[\Big\|\frac{\d\mu''}{\d\mu}\Big\|_\infty \leq \frac{\a_1}{1 - \a\eps} < \a.\]
Also, if $U \subseteq X_0$ with $\mu''(U) \geq \k$, then
\[\mu'(U) \geq \mu'(X_0)\k > (1 - \a\eps)\k = \k_1,\]
whereas for any $W \subseteq X_0$ one has $\mu''(W) = \mu'(W)/\mu'(X_0) \geq \mu'(W)$.  Therefore~(\ref{eq:bicov-bound}) implies that also
\begin{eqnarray*}
\mu''(U) \geq \k \quad \Longrightarrow \quad \bicov_{\k'}((U,d^\bfX_{[-N;0)},d^\bfX_{[0;N)},\mu''),\delta_1 N) \leq m.
\end{eqnarray*}

Finally, let $\nu'' := \Phi_\ast\mu''$, so this also satisfies $\|\d\nu''/\d\nu\|_\infty < \a$.  This will be our choice of measure on $Y$.  Given any $V \subseteq Y$ with $\nu''(V) \geq \k$, let $U := \Phi^{-1}(V)\cap X_0$.  Then also $\mu''(U) \geq \k$, and $\Phi$ defines a $c$-almost $L$-pair-Lipschitz and measure-preserving map
\[(U,d^\bfX_{[-N;0)},d^\bfX_{[0;N)},\mu'') \to (V,d^\bfY_{[-N,0)},d^\bfY_{[0;N)},\nu''),\]
so Lemma~\ref{lem:bicov-under-pair-Lip} gives
\begin{multline*}
\bicov_{\k'}\big((U,d^\bfX_{[-N;0)},d^\bfX_{[0;N)},\mu''),\delta_1 N\big)\\ \geq \bicov_{\k'}\big((V,d^\bfY_{[-N,0)},d^\bfY_{[0;N)},\nu''),(L\delta_1+c)N\big).
\end{multline*}
Since $L\delta_1 + c \leq \delta$, this completes the proof. \end{proof}

\begin{rmk}
It is not clear how well bi-covering numbers behave under Cartesian products.  However, one cannot hope for any nontrivial estimates for joinings.  This can be seen from the result of ~\cite{SmoTho79} that \emph{any} positive-entropy system is a joining of three Bernoulli factors.  We will see later that Bernoulli systems gives trivial $\rm{BICOV}$ values, whereas some positive-entropy systems, such as nontrivial RWRSs, do not --- so the triviality of the former cannot give a bound on the latter. \fin
\end{rmk}

\begin{dfn}
The family of sequences
\[\BICOV_{\a,\k,\k',\delta N}(X,d^\bfX_{[-N;0)},d^\bfX_{[0;N)},\mu), \quad N\geq 1,\]
parameterized by $\a \geq 1$, $\k > k' > 0$ and $\delta > 0$, is the \textbf{bi-covering rate} of $(X,d^X,\mu,T)$.  Proposition~\ref{prop:big-bicov-factor} implies that it depends only on the isomorphism class of $(X,\mu,T)$, up to the notion of equivalence implied by that proposition.
\end{dfn}

In the sequel, it will also be useful to compare the bi-covering rates of different pseudometrics defined on the same system.  The following is immediate from the definition of $\bicov_\bullet$.

\begin{lem}\label{lem:bd-above-with-ptn}
Let $(X,d^X,\mu,T)$ be a compact model p.-p. system, let $M,\eps > 0$, and let $\rho$ be a totally bounded Borel pseudometric on $X$ such that $d \leq M\rho + \eps$.  Then also
\begin{multline*}
\BICOV_{\a,\k,\k',(M\delta+\eps) N}(X,d^\bfX_{[-N;0)},d^\bfX_{[0;N)},\mu)\\ \leq \BICOV_{\a,\k,\k',\delta N}(X,\rho^\bfX_{[-N;0)},\rho^\bfX_{[0;N)},\mu)
\end{multline*}
for all $\a > 1$, $\k > \k' > 0$, $\delta > 0$ and $N \in \bbN$. \qed
\end{lem}

Most often this will be used with $\rho(x,x') := 1_{\P(x)\neq \P(x')}$ for some finite measurable partition $\P$.

\subsection{Two elementary examples}

Before broaching the bi-covering rates of generalization RWRS systems, it will be instructive to analyze them in two rather simpler cases.  This subsection is essentially a digression, and can be skipped without missing any of the proof of Theorem A.

\subsubsection{Isometric systems}

\begin{prop}\label{prop:bicov-for-cpt}
If $T$ is an isometry of the compact metric space $(X,d)$ and $\mu \in \Pr^T X$ is ergodic, then
\begin{multline*}
\min_{\big\|\frac{\d\mu'}{\d\mu}\big\|_\infty \leq \a}\max_{\mu'(U) \geq \k}\rm{cov}_{\k'}((U,d,\mu'),2\delta)\\
\leq \BICOV_{\a,\k,\k',\delta N}(X,d^\bfX_{[-N;0)},d^\bfX_{[0;N)},\mu)\\ \leq \min_{\big\|\frac{\d\mu'}{\d\mu}\big\|_\infty \leq \a}\max_{\mu'(U) \geq \k}\rm{cov}_{\k'}((U,d,\mu'),\delta)
\end{multline*}
for all $\a \in (1,\infty)$, $\delta > 0$, $\k > k' > 0$, and $N \in \bbN$.
\end{prop}

Of course, the optimizations involved in these upper and lower bounds may still be non-trivial, and depend rather delicately on $(X,d,\mu)$, but they do not involve $N$.

\begin{proof}
Since $T$ is an isometry, one has $d\circ (T^{\times 2})^n = d$ for all $n$, and hence
\[d^\bfX_{[-N;0)} = d^\bfX_{[0;N)} = N\cdot d.\]

For any metric space $(X,d)$, any $\delta > 0$ and any $x \in X$, one has
\[B_\delta(x) \subseteq B_\delta(B_\delta(x)) \subseteq B_{2\delta}(x).\]
Therefore, for any $U \subseteq X$, one has
\begin{multline*}
\cov_{\k'}((U,d),2\delta) \leq \bicov_{\k'}((U,d,d),\delta)\\ = \bicov_{\k'}((U,d^\bfX_{[-N;0)},d^\bfX_{[0;N)}),\delta N) \leq \rm{cov}_{\k'}((U,d),\delta).
\end{multline*}

Now performing the optimization over $\mu'$ and $U$ completes the proof.
\end{proof}

\subsubsection{Bernoulli systems}

\begin{prop}\label{prop:Bern-bicov-trivial}
If $(X,d^X,\mu,T)$ is a compact model of a Bernoulli system, then
\[\BICOV_{\a,\k,\k',\delta N}(X,d^\bfX_{[-N;0)},d^\bfX_{[0;N)},\mu) = 1\]
for all sufficiently large $N$, for all $\a > 1$, $\k > \k' > 0$ and $\delta > 0$.
\end{prop}

Thus, Propositions~\ref{prop:bicov-for-cpt} and~\ref{prop:Bern-bicov-trivial} show that both compact systems and Bernoulli systems have bi-covering rates that do not grow with $N$, even though they are in many ways `extreme opposites' with regard to mixing behaviour.

The first step is an auxiliary result comparing marginal distributions over different time-intervals.  Let $\bfX$ be as above, let $\P$ be any finite Borel partition of $X$, let $m := |\P|$ and let $\xi:X\to [m]$ be a finite-valued function generating $\P$.  Given $\eta > 0$ and $N \in \bbN$, define
\[X^{\rm{fat}}_{N,\eta} := \big\{x \in X\,\big|\ \xi^{[0;N)}_\ast(\mu_{|\P^{[-N;0)}(x)}) \ll_{\exp(\eta N),\eta} \xi^{[0;N)}_\ast\mu\big\}.\]

\begin{lem}\label{lem:fat}
For every $\eta > 0$ one has
\[\mu(X^{\rm{fat}}_{N,\eta}) \to 1 \quad \hbox{as}\ N\to\infty.\]
\end{lem}

\begin{proof}
Equation~(\ref{eq:rel-ent-of-bary}) and a special case of Lemma~\ref{lem:cond-mut-inf-sublin} give
\begin{multline*}
\int_X \rmD(\xi^{[0;N)}_\ast(\mu_{|\P^{[-N;0)}(x)})\,|\,\xi^{[0;N)}_\ast\mu)\,\mu(\d x) = \rmI_\mu(\P^{[0;N)};\P^{[-N;0)})\\ = \rm{o}(N) \quad \hbox{as}\ N\to\infty.
\end{multline*}
Therefore, by Markov's inequality, the sets
\[\{x \in X\,|\ \rmD(\xi^{[0;N)}_\ast(\mu_{|\P^{[-N;0)}(x)})\,|\,\xi^{[0;N)}_\ast\mu) \leq \eta^2N - \rme^{-1}\}\]
have measure tending to $1$ as $N \to \infty$.  By Lemma~\ref{lem:ent-to-unif-int} with $C := \eta N$, these are contained in the sets $X^{\rm{fat}}_{N,\eta}$.
\end{proof}

The approximate absolute continuity in the definition of $X^{\rm{fat}}_{N,\eta}$ will be used in conjunction with Theorem~\ref{thm:Bern-shift-almost-exp-conc}.  Let $X^{\rm{conc}}_N$ be a sequence of high-probability Borel subsets of $X$ such that $(X^{\rm{conc}}_N,N^{-1}d^\P_{[0;N)},\mu_{|X^{\rm{conc}}_N})$ exhibits exponential measure concentration, as given by that theorem.  Let $c(\delta) > 0$ be the exponential rate of concentration for this sequence for each radius $\delta > 0$. Also, if $I = [a,a+N) \subseteq \bbZ$, then let $X^{\rm{conc}}_I := T^a(X^{\rm{conc}}_N)$.  This is clearly $\P^I$-measurable.

Now given $\g,\delta > 0$, define
\begin{multline*}
X^{\rm{loc.exp}}_{N,\g,\delta} := \big\{x \in X^{\rm{conc}}_{[0;N)}\,\big|\ \hbox{if}\ U \subseteq X\ \hbox{and}\ \mu_{|\P^{[-N;0)}(x)}(U) \geq \g\ \hbox{then}\\ \mu(B^{d^\P_{[0;N)}}_{\delta N}(U)\,|\,X^{\rm{conc}}_{[0;N)}) > 1 - \rme^{-c(\delta)N}\ \big\}.
\end{multline*}
Intuitively, $X^{\rm{loc.exp}}_{N,\g,\delta}$ consists of those $x$ such that if an event $U$ is reasonably likely given the `past' $\P^{[-N;0)}(x)$, then a small Hamming-neighbourhood around $U$ for the `future' $\P^{[0;N)}$ is very nearly the whole of $X^{\rm{conc}}_{[0;N)}$.

\begin{lem}
For every $\g,\delta > 0$, one has
\[\mu(X^{\rm{loc.exp}}_{N,\g,\delta}) \to 1 \quad \hbox{as}\ N\to\infty.\]
\end{lem}

\begin{proof}
Choose some $\eta < \min\{\g/2,c(\delta)\}$, which implies that
\[\frac{\g\rme^{-\eta N}(1 - 2\eta/\g)}{2} \geq \rme^{-c(\delta)N}\]
for all sufficiently large $N$.

Let
\begin{multline*}
Y_N := \big\{x \in X^{\rm{conc}}_{[0;N)}\,\big|\ \mu_{|\P^{[-N;0)}(x)}(X^{\rm{conc}}_{[0;N)}) > 1 - \g/2\\ \hbox{and}\ \xi^{[0;N)}_\ast(\mu_{|\P^{[-N;0)}(x)}) \ll_{\exp(\eta N),\eta} \xi^{[0;N)}_\ast\mu\big\}.
\end{multline*}
Since
\[\mu(X^{\rm{conc}}_{[0;N)}) = \int \mu_{|\P^{[-N;0)}(x)}(X^{\rm{conc}}_{[0;N)})\,\mu(\d x) \to 1 \quad \hbox{as}\ N\to\infty,\]
Markov's Inequality, Lemma~\ref{lem:fat} and Theorem~\ref{thm:Bern-shift-almost-exp-conc} imply that $\mu(Y_N) \to 1$ as $N\to\infty$.

We will show that $Y_N \subseteq X^{\rm{loc.exp}}_{N,\g,\delta}$, so suppose that $x \in Y_N$ and $U \subseteq \P^{[-N;0)}(x)$ with $\mu_{|\P^{[-N;0)}(x)}(U) \geq \g$.  Let $U' := U\cap X^{\rm{conc}}_{[0;N)}$, so the definition of $Y_N$ implies that $\mu_{|\P^{[-N;0)}(x)}(U') \geq \g/2$.  Let $W := \P^{[0;N)}(U')$, so
\begin{eqnarray}\label{eq:U-and-U-prime}
B^{d^\P_{[0;N)}}_{\delta N}(W) = B^{d^\P_{[0;N)}}_{\delta N}(U') \subseteq B^{d^\P_{[0;N)}}_{\delta N}(U).
\end{eqnarray}

Now one has
\[\xi^{[0;N)}_\ast(\mu_{|U'}) \ll_{2/\g,0} \xi^{[0;N)}_\ast(\mu_{|\P^{[-N;0)}(x)}) \ll_{\exp(\eta N),\eta} \xi^{[0;N)}_\ast\mu,\]
and hence, by the rules in Lemma~\ref{lem:abs-ct-basics},
\[\xi^{[0;N)}_\ast(\mu_{|U'}) \ll_{2\exp(\eta N)/\g,2\eta/\g} \xi^{[0;N)}_\ast\mu.\]
This implies that
\[\mu(U') \geq \frac{\g\rme^{-\eta N}}{2} (\mu_{|U'}(U') - 2\eta/\g) = \frac{\g\rme^{-\eta N}(1 - 2\eta/\g)}{2},\]
and this is at least $\rme^{-c(\delta)N}$ for $N$ large enough, by our choice of $\eta$.  Now~(\ref{eq:U-and-U-prime}) and the definition of $c(\delta)$ complete the proof.
\end{proof}

\begin{proof}[Proof of Proposition~\ref{prop:Bern-bicov-trivial}]
First fix $\delta > 0$, let $\P$ be a finite Borel partition of $X$ into sets of diameter less than $\delta/2$, and let $d^\P$ be the associated pseudometric: $d^\P(x,x') := 1_{\P(x) \neq \P(x')}$.  It follows that
\begin{eqnarray}\label{eq:B-contains-PI}
B^{d_I^\bfX}_{\delta |I|}(x) \supseteq B^{d_I^\P}_{\delta |I|/2}(x)
\end{eqnarray}
for any bounded discrete interval $I \subseteq \bbZ$ and any $x \in X$.

Now fix $\a \in (1,\infty)$ and $\k > \k' > 0$, and let $\g:= \k/\a$ and $\eps := (\k - \k')/\a$.  We will show that for any sufficiently large $N$ and any $U \subseteq X$ with $\mu(U) \geq \g$, there are many points $x \in U$ has the property that
\begin{eqnarray}\label{eq:big-bineigh-for-Bern}
\mu\big(U\cap B^{d^\P_{[0;N)}}_{\delta N/2}\big(U\cap \P^{[-N;0)}(x)\big)\big) > \mu(U) - \eps.
\end{eqnarray}
In view of~(\ref{eq:B-contains-PI}), this implies the same lower bound for the $(\delta N)$-bi-neighbourhood of $x$ in $(U,d^\bfX_{[-N;0)},d^\bfX_{[0;N)})$, completing the proof.

To prove~(\ref{eq:big-bineigh-for-Bern}), observe that
\[\g \leq \mu(U) = \int_X \mu_{|\P^{[-N;0)}(x)}(U)\,\mu(\d x) = \int_{X^{\rm{loc.exp}}_{N,\g/2,\delta/2}} \mu_{|\P^{[-N;0)}(x)}(U)\,\mu(\d x) + \rm{o}(1)\]
as $N\to \infty$, so provided $N$ is sufficiently large, there must be some $x \in X^{\rm{loc.exp}}_{N,\g/2,\delta/2}$ such that
\[\mu_{|\P^{[-N;0)}(x)}(U) \geq \g/2.\]
Since this implies that $U \cap \P^{[-N;0)}(x) \neq \emptyset$, and the left-hand side of this last inequality depends only on the cell $\P^{[-N;0)}(x)$, we may move $x$ within that cell if necessary to assume in addition that $x \in U$.  However, by the definition of $X^{\rm{loc.exp}}_{N,\g/2,\delta/2}$, these assumptions now imply
\begin{multline*}
\mu\big(B^{d^\P_{[0;N)}}_{\delta N/2}\big(U\cap \P^{[-N;0)}(x)\big)\big)\\ \geq \mu\big(B^{d^\P_{[0;N)}}_{\delta N/2}\big(U\cap \P^{[-N;0)}(x)\big)\,\big|\,X^{\rm{conc}}_{[0;N),\delta/2}\big) - \mu(X\setminus X^{\rm{conc}}_{[0;N),\delta/2})\\ > 1- \rme^{-c(\delta/2)N} - \mu(X\setminus X^{\rm{conc}}_{[0;N),\delta/2}),
\end{multline*}
and this is greater than $1 - \eps$ for all sufficiently large $N$, implying~(\ref{eq:big-bineigh-for-Bern}).
\end{proof}

\subsection{Behaviour of the invariant for generalized RWRS systems}

We will now formulate our main result for the bi-covering rate of generalized RWRS systems.  It involves a certain universal function $\psi_{\rm{BM}}:[1,\infty)\to (0,\infty]$, defined in terms of geometric features of Brownian sample paths.

\begin{dfn}\label{dfn:psi}
The function $\psi_{\rm{BM}}:[1,\infty)\to (0,\infty]$ is defined as follows:
\[ \psi_{\rm{BM}}(\a) := \inf\Big\{\psi \in (0,\infty]\,\Big|\ \sfW^{\otimes 2}_{[0,1]}\big\{\calL^1(B_{[0,1]}\cap B'_{[0,1]}) \leq \psi\} \geq 1/\a\Big\}.\]
\end{dfn}

It is easy to check that the random variable
\begin{multline*}
C_0(0,1] \times C_0(0,1] \to [0,\infty):(B,B')\mapsto \cal{L}^1(B_{[0,1]}\cap B'_{[0,1]})\\ = \min\{\sup_{0 \leq t \leq 1}B_t,\sup_{0 \leq t \leq 1}B'_t\} - \max\{\inf_{0\leq t \leq 1}B_t,\inf_{0 \leq t \leq 1}B'_t\}
\end{multline*}
has an atomless distribution under $\sfW^{\otimes 2}$.  Using this and standard properties of Brownian motion, one easily verifies the following.

\begin{lem}\label{lem:props-of-psiBM}
The function $\psi_{\rm{BM}}$ has the following properties:
\begin{itemize}
\item $\psi_{\rm{BM}}(1) = \infty$;
\item $\psi_{\rm{BM}}(\a) \in (0,\infty)$ for all $\a > 1$;
\item $\psi_{\rm{BM}}$ is strictly decreasing;
\item $\psi_{\rm{BM}}$ is continuous. \qed
\end{itemize}
\end{lem}

In terms of $\psi_{\rm{BM}}$, our main result is as follows.

\begin{thm}\label{thm:rate}
If $(\bfY,\s)$ is a well-distributed pair satisfying the Enhanced Invariance Principle, $\bfX$ is an ergodic compact model flow, and $\a \in (1,\infty)$, then
\begin{multline*}
\sup_{\k > \k' > 0}\sup_{\delta > 0}\limsup_{N\to\infty}\frac{\log \BICOV_{\a,\k,\k',\delta}(Y \times X,d_{[-N;0)}^{\bfY \ltimes_\s \bfX},d^{\bfY \ltimes_\s \bfX}_{[0;N)},\nu \otimes \mu)}{\sqrt{N}}\\ = \psi_{\rm{BM}}(\a)\rmh(\bfX).
\end{multline*}
\end{thm}

(Proposition~\ref{prop:big-bicov-factor} already shows that this $\sup\limsup$ does not depend on the choice of compact metric models for $\bfY$ and $\bfX$ as abstract p.-p. systems.)  The reason for this somewhat delicate dependence on the properties of Brownian motion will become clear during the proof.

\begin{rmk}
I expect that the limit-supremum here is actually a limit, and that this requires only a slight enhancement of the proof of the lower bound given below.  However, that enhancement seems to require rather heavier bookkeeping, so we do not pursue it in this paper. \fin
\end{rmk}

\begin{proof}[Proof of Theorem A from Theorem~\ref{thm:rate}]
Suppose that $\bfY \ltimes_\s \bfX_1$ and $\bfY \ltimes_\s \bfX_2$ are two examples as in Theorem A, and that the former admits a factor map to the latter.  Then we may take logarithms in the inequality of Proposition~\ref{prop:big-bicov-factor}, divide by $\sqrt{N}$, and then deduce from Theorem~\ref{thm:rate} that $\rmh(\bfX_1) \geq \rmh(\bfX_2)$.
\end{proof}

The rest of this paper is given to the proof of Theorem~\ref{thm:rate}.  This will involve separate proofs of upper and lower bounds, the second being the more difficult direction.

\section{The combinatorial basis of the upper bound}\label{sec:before-upper-bd}

This section introduces a general tool which will underly the proof of the upper bound in Theorem~\ref{thm:rate}.  Although very elementary, it may be of interest in its own right.  Subsection~\ref{subs:gen-upper-bound} also gives an easier outing for this tool, proving that the bi-covering rates of arbitrary p.-p. systems are always sublinear.

\subsection{A bound using mutual information}

For this subsection, fix a probability space $(X,\mu)$.  We will next develop ways to find an efficient covering of a `large' (in terms of $\mu$) portion of $X$ using certain distinguished subsets, based on some other information about those subsets.

Our most basic result in this direction assumes that these special subsets are involved in a reasonably `smooth' barycentric decomposition of $\mu$.

\begin{lem}\label{lem:efficient-cov}
Suppose that $(X,\mu)$ and $(Z,\nu)$ are standard Borel probability spaces, and that $z\mapsto \mu_z$ is a measurable family of finite Radon measures on $X$, uniformly bounded, such that
\[\mu = \int_Z \mu_z\,\nu(\d z) \quad \hbox{and} \quad \mu_z \ll_{M,\eps} \mu \quad \forall z \in Z.\]
Suppose in addition that for each $z \in Z$, $Y_z$ is a Borel subset of $X$ for which $\mu_z(Y_z) = 1$. Then for every $\a < 1 - \eps$ there is a subset $S \subseteq Z$ with
\[|S| \leq \frac{M}{1 - \a - \eps} \quad \hbox{and} \quad \mu\Big(\bigcup_{z \in S}Y_z\Big) > \a.\]
\end{lem}

Note that the measures $\mu_z$ are not required to be probability measures; this flexibility will be helpful shortly.

\begin{proof}
The set $S$ is constructed by the following greedy recursion.

Suppose that $z_1,\ldots,z_m \in Z$ have already been picked, where this is vacuous if $m=0$.  If $\mu\big(\bigcup_{i=1}^m Y_{z_i}\big) > \a$, then Stop and set $S := \{z_1,\ldots,z_m\}$.  Otherwise, let $U := X \setminus \bigcup_{i=1}^m Y_{z_i}$, and observe that
\[1 - \a \leq \mu(U) = \int_Z\mu_z(U)\,\nu(\d z),\]
so there is some $z_{m+1} \in Z$ for which
\[\mu_{z_{m+1}}(U) = \mu_{z_{m+1}}(U\cap Y_{z_{m+1}}) \geq 1 - \a,\]
and hence
\[M\mu(U\cap Y_{z_{m+1}}) + \eps \geq 1 - \a \quad \Longrightarrow \quad \mu(U\cap Y_{z_{m+1}}) \geq \frac{1 - \a - \eps}{M}.\]
This gives the choice of the next point $z_{m+1}$.

Having obtained $z_1,\ldots,z_m$ by the above algorithm, we have
\[\mu\Big(\bigcup_{i=1}^m Y_{z_i}\Big) = \sum_{i=1}^m\mu\Big(Y_{z_i}\Big\backslash \bigcup_{j=1}^{i-1}Y_{z_j}\Big) \geq m\cdot \frac{1 - \a - \eps}{M}.\]
This requires that $m \leq M/(1 - \a - \eps)$, so the recursion must terminate in a set $S$ containing at most this many points.  The union of the corresponding supports must have measure greater than $\a$, since this was the condition for termination.
\end{proof}

Lemma~\ref{lem:efficient-cov} gives the covering conclusion that we will need later, but its assumption that $\mu_z \ll_{M,\eps} \mu$ uniformly in $z$ is stronger than we will meet directly.  We will next turn it into an estimate closer to our applications. This begins with a useful way of `trimming' a positive-measure subset $U$ of a probability space $(X,\S,\mu)$ relative to a finite measurable partition $\P$ of $X$.

\begin{dfn}
Let $\g \geq 0$, let $(X,\S,\mu)$ be a probability space, and let $\P \subseteq \S$ be a finite partition into positive-measure sets.  A subset $V \in \S$ is \textbf{locally $\g$-thick in $\P$} if for every $C\in \P$ one has
\[\hbox{either} \quad C \cap V = \emptyset \quad \hbox{or} \quad \mu(V\,|\,C) \geq \g.\]
\end{dfn}

For clarity, note that $\emptyset$ is locally $\g$-thick in every partition, for every $\g$.

The following lemma is an immediate consequence of Markov's Inequality, but it will be worth having it ready to hand.

\begin{lem}[Trimming a set to a partition]\label{lem:trim}
Let $(X,\S,\mu)$ be a probability space, let $\P \subseteq \S$ be a finite partition into positive-measure sets, let $U \in \S$ with $\mu(U) > 0$, and let $\a \in (1/2,1)$.  Then the subset
\[V := \bigcup_{\hbox{\scriptsize{$\begin{array}{c}C \in \P\\ \mu(U\,|\,C) \geq (1 - \a)\mu(U)\end{array}$}}} (U\cap C) \subseteq U\]
satisfies $\mu(V) \geq \a\mu(U)$ and is locally $((1 - \a)\mu(U))$-thick in $\P$. \qed
\end{lem}

The above definition and lemma have an obvious generalization to local thickness relative to a Borel map $\pi:X\to Y$ and a given disintegration of $\mu$ over $\pi$, but this will not be needed.

Now assume that $\calS$ and $\T$ are two fixed finite Borel partitions of $(X,\mu)$.

\begin{prop}\label{prop:efficient-cov}
Let $I := \rmI_\mu(\calS;\T)$ and suppose that $\a \in (0,1]$ and $\eta \in (0,\a)$.  Then for every Borel $U\subseteq X$ with $\mu(U) \geq \a$, there is a subset $S \subseteq U$ with
\[\log |S| \lesssim_{\a,\eta} I+1 \quad \hbox{and} \quad \mu\big(U\cap \T(U\cap \calS(S))\big) > \mu(U) - \eta\]
(where the notation in the first inequality indicates that the bound depends on $\a$ and $\eta$ but not otherwise on the choice of $U$).
\end{prop}

The connection between the mutual-information bound assumed here and the hypothesis of Lemma~\ref{lem:efficient-cov} will result from Lemma~\ref{lem:ent-to-unif-int}.

\begin{proof}
Let $\psi:X\to A$ be a map to a finite set that generates the partition $\T$. Then equation~(\ref{eq:rel-ent-of-bary}) gives
\[I = \int_X \rmD(\psi_\ast(\mu_{|\calS(x)})\,|\,\psi_\ast\mu)\,\mu(\d x).\]

\vspace{7pt}

\emph{Step 1.}\quad Choose $\zeta := \eta/3$, and observe that this implies
\[(1 - \zeta)^2(\mu(U) - \zeta) > (1 - 2\zeta)(\mu(U) - \zeta) > \mu(U) - 3\zeta = \mu(U) - \eta.\]

Now let $D := I/\zeta$.  Applying Markov's Inequality to the integral above gives that the set
\[X_1 := \big\{x\,\big|\ \rmD(\psi_\ast(\mu_{|\calS(x)})\,|\,\psi_\ast\mu) \leq D\big\}\]
has
\[\mu(X_1) \geq 1 - I/D = 1 - \zeta.\]
Letting $U_1 := U\cap X_1$, it follows that $\mu(U_1) \geq \mu(U) -\zeta$.

On the other hand, for any $C \in (0,\infty)$, Lemma~\ref{lem:ent-to-unif-int} gives
\begin{eqnarray}\label{eq:XD-elts-controlled}
X_1 \subseteq \big\{x\,\big|\ \psi_\ast(\mu_{|\calS(x)}) \ll_{\rme^C,(D + \rme^{-1})/C} \psi_\ast\mu\big\}.
\end{eqnarray}

\vspace{7pt}

\emph{Step 2.}\quad Now let $\g:= \zeta\mu(U_1) \geq \zeta(\mu(U)-\zeta)$, and apply Lemma~\ref{lem:trim} to find some $V \subseteq U_1$ with
\[\mu(V) \geq (1-\zeta)\mu(U_1) \geq (1-\zeta)(\mu(U) - \zeta) > (\mu(U) - \eta)/(1-\zeta)\]
and which is locally $\g$-thick in $\T$.

\vspace{7pt}

\emph{Step 3.}\quad The decomposition of $\mu$ into the measures $\mu_{|\calS(x)}$ may be conditioned on $V$ and pushed forward under $\psi$ to obtain
\begin{eqnarray}\label{eq:useful-disint}
\psi_\ast(\mu_{|V}) = \frac{1}{\mu(V)}\psi_\ast(1_V\cdot \mu) = \frac{1}{\mu(V)}\int_X \psi_\ast(1_V\cdot \mu_{|\calS(x)})\,\mu(\d x)
\end{eqnarray}
(being aware that the measures inside the right-hand integral may now not be probability measures).  Let $B := \psi(V) \subseteq A$, so the above pushforward measures are all supported on $B$.  Applying~(\ref{eq:XD-elts-controlled}), it follows that any $x \in V\subseteq X_1$ satisfies
\[\psi_\ast(1_V\cdot \mu_{|\calS(x)}) \leq 1_B\cdot \psi_\ast(\mu_{|\calS(x)}) \ll_{\rme^C,(D + \rme^{-1})/C} 1_B\cdot \psi_\ast\mu.\]
On the other hand, since $V$ is locally $\g$-thick in $\T$, for any $b \in B$ one has $\mu(V\cap \psi^{-1}\{b\}) \geq \g\mu(\psi^{-1}\{b\})$.  Therefore
\[1_B\cdot \psi_\ast\mu \leq \frac{1}{\g} \psi_\ast(1_V\cdot \mu)= \frac{\mu(V)}{\g}\psi_\ast(\mu_{|V}).\]

Combining this with the preceding inequalities, and recalling that $\mu(V) \geq \mu(U)-\eta \geq \a - \eta$, we obtain
\[\frac{1}{\mu(V)}\psi_\ast(1_V\cdot \mu_{|\calS(x)}) \ll_{\rme^C/\g,\ (D+\rme^{-1})/C(\a-\eta)} \psi_\ast(\mu_{|V}).\]

\vspace{7pt}

\emph{Step 4.}\quad This relates the integral and integrands in~(\ref{eq:useful-disint}), and so puts us in position to apply Lemma~\ref{lem:efficient-cov}.  The family of measures is $(1/\mu(V))\psi_\ast(1_V\cdot \mu_{|\calS(x)}) \in \Pr A$ for $x \in X$, and for each $x$ the relevant supporting subset is $\psi(V\cap \calS(x)) \subseteq A$.  To carry out this application, it remains to choose the constant $C$.  Let $\eps:= \zeta/2$, and now let
\[C := \frac{D + \rme^{-1}}{\eps (\a-\eta)} = \frac{2(I/\zeta + \rme^{-1})}{\zeta (\a - \eta)} \lesssim_{\a,\eta} I+1.\]
For the assumptions of Lemma~\ref{lem:efficient-cov} we must set $M := \rme^C$, and can then obtain some $S \subseteq V$ such that $|S| \leq M/(\zeta -\eps) = 2M/\zeta$ and
\begin{multline*}
\mu_{|V}(\T(V\cap \calS(S))) = \psi_\ast(\mu_{|V})\big(\psi(V\cap \calS(S))\big) > 1 - \zeta\\
\Longrightarrow \quad \mu\big(U\cap \T(U\cap \calS(S))\big) \geq \mu\big(V\cap \T(V\cap \calS(S))\big) > (1-\zeta)\mu(V) \geq \mu(U) - \eta.
\end{multline*}
This gives a subset $S$ contained in $X$, but not necessarily in $U$.  However, one can discard any $x \in S$ such that $U\cap \calS(x) = \emptyset$ without disrupting these estimates, and any remaining $x$ can be replaced by an element of $U\cap \calS(x)$ to give an element of $U$ with the same $\calS$-cell.  We may therefore take $S \subseteq U$, as required.
\end{proof}

One further generalization of the above result will be important later.  To formulate it, we now posit a third partition $\R$ such that $\R \preceq \calS\wedge \T$.

\begin{prop}\label{prop:rel-efficient-cov}
Let $I := \rmI_\mu(\calS;\T\,|\,\R)$ and suppose that $\a \in (0,1]$ and $\eta \in (0,\a)$.  Then for every Borel $U\subseteq X$ with $\mu(U) \geq \a$ there is a subset $S \subseteq U$ such that
\[\log \frac{|S|}{|\R|} \lesssim_{\a,\eta} I+1 \quad \hbox{and} \quad \mu\big(U\cap \T(U\cap \calS(S))\big) > \mu(U) - \eta.\]
\end{prop}

\begin{proof}
By discarding the union of all $\mu$-negligible cells of $\R$, we may assume that all cells have positive measure.

\quad

\emph{Step 1.}\quad Again let $\zeta := \eta/3$, and let $J := I/\zeta$.  Recall that
\[I = \int_X \rmI_{\mu_{|\R(x)}}(\calS;\T)\,\mu(\d x),\]
and let
\[\R_0 := \{C \in \R\,|\ \rmI_{\mu_{|C}}(\calS;\T) \leq J\},\]
so Markov's Inequality gives
\[\mu\big(\bigcup \R_0\big) \geq 1 - I/J = 1 - \zeta.\]
Letting $U_0 := U\cap \bigcup \R_0$, it follows that $\mu(U_0) \geq \mu(U) - \zeta$.

\vspace{7pt}

\emph{Step 2.}\quad Now let $\g := \zeta \mu(U_0)$, and apply Lemma~\ref{lem:trim} to obtain $V \subseteq U_0$ with $\mu(V) \geq (1-\zeta)\mu(U_0)$ and which is locally $\g$-thick in $\R$.  Since $V\subseteq U_0$, we know that $V\cap \R = V\cap \R_0$.  Let $\R_1 := \{C \in \R_0\,|\ C\cap V\neq \emptyset\}$.

\vspace{7pt}

\emph{Step 3.}\quad Finally, for each $C \in \R_1$, consider the probability space $(C,\mu_{|C})$ and the two partitions $\calS \cap C$ and $\T \cap C$.  Since $\R_1 \subseteq \R_0$, we know that
\[\rmI_{\mu_{|C}}(\calS \cap C; \T \cap C) \leq J,\]
while Step 2 guarantees that $\mu_{|C}(V) \geq \g \geq \zeta(\mu(U)-\zeta) \geq \zeta(\a - \zeta)$.  This last lower bound depends only on $\a$ and $\eta$, as does $\zeta$, so we may apply Proposition~\ref{prop:efficient-cov} within each of these conditioned probability spaces to obtain subsets $S_C \subseteq V\cap C$ such that
\[\log |S_C| \lesssim_{\a,\eta} J + 1 \lesssim_\eta I+1 \quad \forall C \in \R_1,\]
while
\[\mu_{|C}\big(V\cap C \cap \T\big(V\cap C\cap \calS(S_C)\big)\big) > \mu_{|C}(V) - \zeta.\]
Let $S := \bigcup_{C \in \cal{R}_1}S_C$.  Then
\[\log\frac{|S|}{|\R|} \leq \max_{C \in \cal{R}_1}\log |S_C| \lesssim_{\a,\eta} I + 1,\]
and
\begin{eqnarray*}
\mu\big(U\cap \T(U \cap \calS(S))\big) &\geq& \mu\big(V \cap \T(V \cap \calS(S))\big)\\
&=& \sum_{C \in \R_1}\mu(C)\mu_{|C}\big(V\cap \T(V \cap \calS(S))\big)\\
&\geq& \sum_{C \in \R_1}\mu(C)\mu_{|C}\big(V\cap C\cap \T\big(V \cap C\cap \calS(S_C)\big)\big)\\
&>& \sum_{C \in \R_1}\mu(C)(\mu_{|C}(V) - \zeta)\\
&\geq& \mu(V) - \zeta \geq (1-\zeta)(\mu(U) - \zeta) - \zeta \geq \mu(U) - \eta.
\end{eqnarray*}
\end{proof}

\subsection{An upper bound for general systems}\label{subs:gen-upper-bound}

This is another digressive subsection, but it also offers a warm-up to the upper bound in Theorem~\ref{thm:rate}.  It will use Proposition~\ref{prop:efficient-cov} to prove a bi-covering-rate upper bound for arbitrary p.-p. systems.  It proves our first concrete relation between bi-covering rates and mutual information.

\begin{prop}
Let $(X,d^X,\mu,T)$ be a compact model p.-p. system.  For any $\a \geq 1$, $\k > \k' > 0$, and $\delta > 0$, there is a finite Borel partition $\P$ of $X$ such that
\[\log\BICOV_{\a,\k,\k',\delta N}(X,d^\bfX_{[-N;0)},d^\bfX_{[0;N)},\mu) \lesssim_{\k,\k'} \rmI_\mu(\P^{[-N;0)},\P^{[0;N)}) + 1\]
as $N\to\infty$.
\end{prop}

\begin{proof}
It suffices to treat the case $\a = 1$, in which $\BICOV$ admits no choice of new measures on $X$.  For this case, let $\P$ be any finite Borel partition of $X$ into cells of diameter less than $\delta$.  Then
\[B^{d^\bfX_I}_{\delta|I|}(x) \supseteq \P^I(x) \quad \forall x\in X\ \hbox{and finite}\ I \subseteq \bbZ,\]
and so also
\[B^{d^\bfX_{[0;N)}}_{\delta N}\big(U\cap B^{d^\bfX_{[-N;0)}}_{\delta N}(S)\big) \supseteq \P^{[0;N)}(U\cap \P^{[-N;0)}(S))\quad \forall S \subseteq U.\]

Now Proposition~\ref{prop:efficient-cov} promises some $S \subseteq U$ such that
\[\mu\big(U \cap \P^{[0;N)}(U\cap \P^{[-N;0)}(S))\big) \geq \k'\]
and also
\[\log |S| \lesssim_{\k,\k'} \rmI_\mu(\P^{[-N;0)};\P^{[0;N)}) + 1.\]
\end{proof}

Combined with Lemma~\ref{lem:cond-mut-inf-sublin}, this immediately gives the following.

\begin{prop}\label{prop:always-sublin}
For any compact model p.-p. system $(X,d^X,\mu,T)$ one has
\[\log\BICOV_{\a,\k,\k',\delta N}(X,d^\bfX_{[-N;0)},d^\bfX_{[0;N)},\mu) = \rm{o}_{\k,\k',\delta}(N)\]
as $N\to\infty$. \qed
\end{prop}

Proposition~\ref{prop:always-sublin} begs the following question.

\begin{ques}
Can the upper bound in Proposition~\ref{prop:always-sublin} be improved to any fixed sub-linear function?
\end{ques}

This seems highly unlikely, but it could be interesting to see examples of $(X,d^X,\mu,T)$ whose log-bi-covering rates come arbitrarily close to linear.  Conjecture~\ref{conj:p-stable} will propose some systems that could achieve rate $N^{1-\eps}$ for any $\eps > 0$.

\section{The upper bound}\label{sec:upper-bd}

This section proves the upper bound in Theorem~\ref{thm:rate}.  The proof is based on the covering estimates of the previous section, similarly to the proof of Proposition~\ref{prop:always-sublin}.  The key is to replace certain balls for the metrics $d^{\bfY \ltimes_\s \bfX}_I$, $I \subseteq \bbZ$, with the cells of associated partitions, and then prove a mutual information bound for an application of Proposition~\ref{prop:rel-efficient-cov}.  Most of the delicacy here will be in choosing the partitions that approximate the metrics.

In principle, one feels that proofs of these results should be possible directly in terms of the metrics $d^{\bfY \ltimes_\s \bfX}_I$, without this switch to partitions.  However, I suspect that would require much thornier estimates in several places.

\subsection{Estimating balls in the skew-product metric}\label{subs:ball-est}

Let $\P$ be the time-zero partition of the SFT $Y$, as previously.  Also, let us normalize the metric on $X$ to assume that $\rm{diam}(X,d^X) \leq 1$.

\begin{lem}\label{lem:P-controls-sigma}
For every $\eps > 0$ there is a $p \in \bbN$ such that for all $N \in \bbN$ and $y,y' \in Y$ one has the following implication:
\[\P^{[-p,N+p)}(y) = \P^{[-p;N+p)}(y') \quad \Longrightarrow \quad \max_{n \in [0;N)}|\s^y_n - \s^{y'}_n| < \eps.\]
\end{lem}

\begin{proof}
Since $\s:Y \to \bbR$ is H\"older continuous, there are some $b < \infty$ and $\b \in (0,1)$ such that
\begin{multline*}
\P^{[-p,N+p)}(y) = \P^{[-p;N+p)}(y') \\ \Longrightarrow \quad |\s(S^iy) - \s(S^iy')| < b\b^{\min\{i+p,N+p-i\}} \quad \forall i \in [0;N).
\end{multline*}
Summing over $i$, this gives
\[|\s^y_n - \s^{y'}_n| \leq \sum_{i=0}^{n-1}|\s(S^iy) - \s(S^iy')| \leq 2\sum_{i\geq p}b\b^i \leq \frac{2b\b^p}{1 - \b} \quad \forall n \in [0;N),\]
which is less than $\eps$ provided $p$ is large enough.
\end{proof}

The following re-write of the above lemma will be useful in Subsection~\ref{subs:sim-to-sim}.

\begin{cor}\label{cor:P-controls-sigma}
For every $\eps > 0$ there is a $\delta > 0$ such that for all $N \in \bbN$ and $y,y' \in Y$ one has the following implication:
\[\max_{n \in [0;N)}d^Y(S^ny,S^ny') < \delta \quad \Longrightarrow \quad \max_{n \in [0;N)}|\s^y_n - \s^{y'}_n| < \eps.\]
\qed
\end{cor}

\begin{lem}\label{lem:ball-prods-in-balls}
For any $\delta > 0$ there are $p \in \bbN$ and a finite Borel partition $\Q$ of $X$ such that the following holds. If $I \in \rm{Int}(\bbR)$, if $y \in Y$ satisfies
\[\s^y_{[0;N)} + [-1,1]\subseteq I,\]
and if $x \in X$, then
\[B^{d^{\bfY \ltimes_\s \bfX}_{[0;N)}}_{\delta N}(y,x) \supseteq \P^{[-p;N+p)}(y) \times \Q^{I\cap \bbZ}(x),\]
and analogously with $[-N;0)$ in place of $[0;N)$ throughout.
\end{lem}

\begin{proof}
Since $T:\bbR\actson (X,d^X)$ is continuous and $\delta > 0$, any partition $\Q$ of $X$ into cells of sufficiently small diameter has the property that
\begin{eqnarray}\label{eq:Q-forces-small}
\forall x,x' \in X, \quad \Q(x) = \Q(x') \quad \Longrightarrow \quad d^X(T^tx,T^tx') < \delta/3 \quad \forall t \in [-1,1].
\end{eqnarray}
Fix a finite Borel partition $\Q$ with this property.

Next, using again the continuity of $T$, choose $\eps > 0$ so small that
\begin{eqnarray}\label{eq:eps-forces-small}
\sup_{|t| \leq \eps}\sup_{x \in X}d^X(x,T^t x) < \delta/3.
\end{eqnarray}

Lastly, choose $p$ as given by Lemma~\ref{lem:P-controls-sigma} for this value of $\eps$.  Increase $p$ further if necessary so that also $2^{-p+2} < \delta/3$.

Now let $y \in Y$, let $I \in \rm{Int}(\bbR)$ with $\s^y_{[0;N)} + [-1,1] \subseteq I$, and let $x \in X$.  Suppose that
\[(y',x') \in \P^{[-p;N+p)}(y)\times \Q^{I\cap \bbZ}(x).\]
The definition of $d^Y$ gives
\[d^Y(S^ny,S^ny') \leq 2^{-p+2} < \delta/3 \quad \forall n \in [0;N),\]
and the choice of $p$ gives
\[|\s^y_n - \s^{y'}_n| \leq \eps \quad \forall n \in [0;N).\]
Therefore, by~(\ref{eq:eps-forces-small}),
\begin{eqnarray*}
d^{\bfY \ltimes_\s \bfX}_{[0;N)}\big((y,x),(y',x')\big) &=& d^\bfY_{[0;N)}(y,y') + \sum_{n=0}^{N-1} d^X(T^{\s^y_n}x,T^{\s^{y'}_n}x')\\
&\leq& \delta N/3 + \sum_{n=0}^{N-1} \big(d^X(T^{\s^y_n}x,T^{\s^y_n}x') + d^X(T^{\s^y_n}x',T^{\s^{y'}_n}x')\big)\\
&\leq& \Big(2\delta/3 + \max_{n \in [0;N)}d^X(T^{\s^y_n}x,T^{\s^y_n}x')\Big)N.
\end{eqnarray*}

Finally, since $\Q^{I\cap \bbZ}(x) = \Q^{I\cap \bbZ}(x')$, the property~(\ref{eq:Q-forces-small}) gives that
\[d^X(T^tx,T^tx') < \delta/3 \quad \forall t \in (I\cap \bbZ) + [-1,1].\]
In particular, this holds for $t = \s^y_n$ for $n \in [0;N)$.
\end{proof}

\subsection{Completion of the upper bound}

Now fix arbitrary $\delta > 0$, $\a \in (1,\infty)$ and $\k > \k' > 0$.  Let $\psi := \psi_{\rm{BM}}(\a)$. For the upper bound in Theorem~\ref{thm:rate}, it will suffice to show that
\begin{multline*}
\BICOV_{\a,\k,\k',\delta N}(Y\times X,d^{\bfY \ltimes_\s \bfX}_{[-N;0)},d^{\bfY \ltimes_\s \bfX}_{[0;N)},\nu\otimes \mu)\\ \leq \exp\big((\rmh(\bfX) + \eps)(\psi + \eps)\sqrt{N} + \eps\sqrt{N}\big)
\end{multline*}
as $N \to \infty$, for every $\eps > 0$.

Having chosen $\delta$, let $p \in \bbN$ and let $\Q$ be a finite Borel partition of $X$ as given by Lemma~\ref{lem:ball-prods-in-balls}.  These will now also be fixed for the rest of the section.  Concerning the system $(X,\mu,T^1)$ and its partition $\Q$, if $\eps > 0$ and $I \in \rm{Int}(\bbR)$ is a bounded interval, then let $X^{\rm{SM}}_{I,\eps}$ denote the set of `typical' points for the partition $\Q^{I\cap \bbZ}$ according to the Shannon-McMillan Theorem~\ref{thm:SM}.

The next step will be to introduce certain auxiliary subsets and partitions of $X$ and $Y\times X$.

\begin{lem}\label{lem:useful-partn}
For each $\eps > 0$ there are finitely many pairs of intervals
\[(I_1,J_1),\ldots,(I_k,J_k) \in \rm{Int}(\bbR) \times \rm{Int}(\bbR)\]
such that
\[0 < \calL^1(I_i\cap J_i) < \psi + \eps \quad \forall i=1,2,\ldots,k\]
and for which the following holds.  For each $N \in \bbN$ there are pairwise-disjoint Borel subsets $W^1_N,\ldots,W_N^k \subseteq Y$ such that
\begin{itemize}
\item[i)] $\nu(W^1_N\cup \cdots \cup W^k_N) > 1/\a$ for all  sufficiently large $N$, and
\item[ii)] for every $y \in W^i_N$ one has
\[\s^y_{[-N;0)} + [-1,1] \subseteq \sqrt{N}I_i \quad \hbox{and} \quad \s^y_{[0;N)}  + [-1,1]\subseteq \sqrt{N}J_i.\]
\end{itemize}
\end{lem}

\begin{proof}
Consider the set
\[W^0 := \big\{(B,B') \in C_0(0,1]\times C_0(0,1]\,\big|\ 0 < \calL^1(B_{[0,1]}\cap B'_{[0,1]}) < \psi + \eps/3\big\}.\]
Definition~\ref{dfn:psi} and Lemma~\ref{lem:props-of-psiBM} give that
\[\sfW_{[0,1]}^{\otimes 2}(W_0) = \sfW_{[0,1]}^{\otimes 2}\big\{\calL^1(B_{[0,1]}\cap B'_{[0,1]}) < \psi + \eps/3\big\} > 1/\a.\]
Since $\sfW_{[0,1]}^{\otimes 2}$ is inner-regular with respect to compact sets, it follows that for some $k \in \bbN$ one can find
\begin{itemize}
\item pairs of intervals $(I_i,J_i) \in \rm{Int}(\bbR)^2$ for $i=1,2,\ldots,k$ such that
\[0 < \calL^1(I_i\cap J_i) < \psi + \eps\]
\item and pairwise-disjoint Borel subsets $W^1,\ldots,W^k \subseteq W^0$
\end{itemize}
such that
\begin{itemize}
\item[i)] $W^1 \cup \cdots \cup W^k$ is open,
\item[ii)] $\sfW_{[0,1]}^{\otimes 2}(W^1\cup \cdots \cup W^k) > 1/\a$, and
\item[iii)] for every $(B,B') \in W^i$ one has
\[B_{[0,1]} + [-\eps/3,\eps/3]\subseteq I_i \quad \hbox{and} \quad B'_{[0,1]} + [-\eps/3,\eps/3] \subseteq J_i.\]
\end{itemize}

Let
\[W^i_N := \{y \in Y\,|\ (\traj_{-N}(\s^y),\traj_N(\s^y)) \in W^i\} \quad \hbox{for each}\ i=1,\ldots,k.\]
Since $W^1\cup \cdots \cup W^k$ is open in $C_0(0,1]\times C_0(0,1]$, the Portmanteau Theorem and Theorem~\ref{thm:IP} imply that $\nu(W^1_N \cup \cdots \cup W^k_N) > 1/\a$ for all sufficiently large $N$.  The desired conclusion (ii) then holds provided also $1/\sqrt{N} < \eps/3$.
\end{proof}

Keeping the notation of the preceding lemma, now let $Z^i_N := W^i_N\times X^{\rm{SM}}_{\sqrt{N}(I_i \cap J_i),\eps}$ for each $i=1,2,\ldots,k$, and let
\[Z_N := Z^1_N \cup \cdots \cup Z^k_N, \quad Z_N^{\rm{c}} := (Y\times X)\setminus Z_N, \quad \hbox{and} \quad \Z_N:= \{Z^1_N,\ldots,Z^k_N,Z_N^{\rm{c}}\}.\]
For these sets $Z_N$ we have
\[(\nu\otimes \mu)(Z_N) \geq \nu(W_N) - \max_{i\leq k} \mu(X\setminus X^{\rm{SM}}_{\sqrt{N}(I_i\cap J_i)\cap \bbZ}),\]
so this is still greater than $1/\a$ for all sufficiently large $N$, by Theorem~\ref{thm:SM} and the fact that $|\sqrt{N}(I_i \cap J_i)\cap \bbZ|\to \infty$ for each $i$.

We next introduce the further partitions of $Y\times X$ that will enable an approximation of our bi-neighbourhoods.

First, for each $i = 1,2,\ldots,k$, define
\begin{eqnarray*}
\R^i_N &:=& \{\emptyset,Y\}\otimes \Q^{\sqrt{N}(I_i\cap J_i)\cap \bbZ},\\
\calS^i_N &:=& \P^{[-N-p;p)} \otimes \Q^{\sqrt{N}I_i\cap \bbZ}, \quad \hbox{and}\\
\T^i_N &=& \P^{[-p;N+p)} \otimes \Q^{\sqrt{N}J_i \cap \bbZ}.
\end{eqnarray*}

These different partitions can be adapted to the cells $Z^i_N$ as follows: let $\R_N$, $\calS_N$ and $\T_N$ be the partitions that refine $\cal{Z}_N$, all contain $Z^\rm{c}_N$ as a single cell, and satisfy
\[\R_N\cap Z^i_N = \cal{R}^i_N\cap Z^i_N,\ \calS_N \cap Z^i_N = \calS^i_N\cap Z^i_N \quad \hbox{and} \quad \T_N \cap Z^i_N = \T^i_N\cap Z^i_N\]
for all $i=1,2,\ldots, k$ and $N \in \bbN$.  Clearly
\[\Z_N \preceq \R_N \preceq \calS_N\wedge \T_N.\]

Now Lemma~\ref{lem:ball-prods-in-balls} and the properties of the sets $W^i_N$ given by Lemma~\ref{lem:useful-partn} imply that
\begin{eqnarray}\label{eq:UST}
\calS_N(y,x) \subseteq B^{d^{\bfY \ltimes_\s \bfX}_{[-N;0)}}_{\delta N}(y,x) \quad \hbox{and} \quad \T_N(y,x) \subseteq B^{d^{\bfY \ltimes_\s \bfX}_{[0;N)}}_{\delta N}(y,x)  \quad \forall (y,x) \in Z_N.
\end{eqnarray}

\begin{lem}\label{lem:base-ent-bound}
With $\R_N$ as above, one has
\[|\R_N| \leq \exp\big((\rmh(\bfX) + \eps)(\psi + \eps)\sqrt{N} + \rm{o}(\sqrt{N})\big) \quad \hbox{as}\ N\to\infty.\]
\end{lem}

\begin{proof}
The definition above gives
\[|\R_N| = 1 + \sum_{i=1}^k|\R_N \cap Z^i_N| = 1 + \sum_{i=1}^K|\Q^{\sqrt{N}(I_i\cap J_i)\cap \bbZ}\cap X^{\rm{SM}}_{\sqrt{N}(I_i\cap J_i),\eps}|.\]
As $N\to\infty$, the number of summands on the right-hand side here is fixed, and their cardinalities are bounded by
\[(\rmh(\bfX) + \eps)|\sqrt{N}(I_i\cap J_i)\cap \bbZ| \leq (\rmh(\bfX) + \eps)(\psi + \eps)\sqrt{N}\]
by the definition of $X^{\rm{SM}}_{\sqrt{N}(I_i\cap J_i),\eps}$.
\end{proof}

\begin{lem}\label{lem:rel-mut-inf-bound}
With $\R_N$, $\calS_N$ and $\T_N$ as above, one has
\[\rmI_{\nu\otimes \mu}(\calS_N;\T_N\,|\,\R_N) = \rm{o}(\sqrt{N}) \quad \hbox{as}\ N\to\infty.\]
\end{lem}

\begin{proof}
The definition of conditional mutual information gives
\begin{eqnarray*}
\rmI_{\nu\otimes \mu}(\calS_N;\T_N\,|\,\R_N) &=& (\nu\otimes \mu)(Z^\rm{c}_N)\rmI_{(\nu\otimes \mu)_{|Z^\rm{c}_N}}(\calS_N;\T_N\,|\,\R_N)\\ && + \sum_{i=1}^k (\nu\otimes \mu)(Z^i_N)\rmI_{(\nu\otimes \mu)_{|Z^i_N}}(\calS_N;\T_N\,|\,\R_N).
\end{eqnarray*}
Now the definition of $\R_N$, $\calS_N$ and $\T_N$ gives that the first term here is zero, and the remaining sum is equal to
\[\sum_{i=1}^k (\nu\otimes \mu)(Z^i_N)\rmI_{(\nu\otimes \mu)_{|Z^i_N}}(\calS^i_N;\T^i_N\,|\,\R^i_N).\]

Applying Lemma~\ref{lem:cond-cond-mut-inf}, this is bounded by
\[\sum_{i=1}^k\big(\log 2 + \rmI_{\nu\otimes \mu}(\calS^i_N;\T^i_N\,|\,\R^i_N)\big),\]
and within each of these summands, an application of Lemma~\ref{lem:mut-inf-add}, Corollary~\ref{cor:Gibbs-mut-inf} and Lemma~\ref{lem:cond-mut-inf-sublin} gives
\begin{eqnarray*}
&&\rmI_{\nu\otimes \mu}(\calS^i_N;\T^i_N\,|\,\R^i_N)\\
&&= \rmI_\nu(\P^{[-N-p;p)};\P^{[-p;N+p)}) + \rmI_\mu(\Q^{\sqrt{N}I_i\cap \bbZ};\Q^{\sqrt{N}J_i\cap \bbZ}\,|\,\Q^{\sqrt{N}(I_i\cap J_i)\cap \bbZ})\\
&&= \rm{O}(1) + \rm{o}(|(\sqrt{N}J_i\cap \bbZ)\setminus (\sqrt{N}I_i\cap \bbZ)|) = \rm{o}(\sqrt{N})
\end{eqnarray*}
as $N\to\infty$.
\end{proof}

\begin{proof}[Proof of upper bound in Theorem~\ref{thm:rate}]
Letting $\l_N:= (\nu\otimes \mu)_{|Z_N}$, the lower bound on $(\nu\otimes \mu)(Z_N)$ implies that $\|\d\l_N/\d(\mu\otimes \nu)\|_\infty < \a$ for all sufficiently large $N$.  This will be Min-er's choice of new measures on $Y\times X$: that is, we will prove that
\begin{multline}\label{eq:bpackupbd}
\max_{\hbox{\scriptsize{$\begin{array}{c} U\subseteq Y\times X\\ \l_N(U) \geq \k\end{array}$}}}\bicov_{\k'}\big((U,d^{\bfY \ltimes_\s \bfX}_{[-N;0)},d^{\bfY \ltimes_\s \bfX}_{[0;N)},\l_N),\delta N\big)\\ \leq \exp((\psi + \eps)(\rmh(\bfX) + \eps)\sqrt{N} + \rm{o}(\sqrt{N}))
\end{multline}
as $N\to\infty$.  In estimating this maximum we may assume that $U \subseteq Z_N$, for any remainder $U\setminus Z_N$ carries none of the measure $\l_N$ and so does not need to be covered.  After assuming that $U\subseteq Z_N$, we have the containments in~(\ref{eq:UST}) for all $(y,x) \in U$.

Now Proposition~\ref{prop:rel-efficient-cov} gives a subset $S \subseteq U$ such that
\begin{eqnarray}\label{eq:log-frac-bound}
\log\frac{|S|}{|\R_N|} \lesssim_{\a,\k,\k''} \rmI_{\nu\otimes \mu}(\calS_N;\T_N\,|\,\R_N) + 1
\end{eqnarray}
and
\begin{multline*}
(\nu\otimes \mu)\big(U\cap \T_N(U\cap \calS_N(S))\big) > (\nu\otimes \mu)(U) - \k''/\a\\
\Longrightarrow \quad \l_N\big(U\setminus \T_N(U\cap \calS_N(S))\big) < \k''\\
\Longrightarrow \quad \l_N\big(U\cap \T_N(U\cap \calS_N(S))\big) > \k - \k'' = \k'.
\end{multline*}
Applying Lemmas~\ref{lem:base-ent-bound} and~\ref{lem:rel-mut-inf-bound} to the estimate~(\ref{eq:log-frac-bound}), it follows that
\[|S| \leq \exp\big((\rmh(\bfX) + \eps)(\psi + \eps)\sqrt{N} + \rm{o}(\sqrt{N})\big)\]
as $N\to\infty$.  This completes the proof.
\end{proof}

The proof of the matching lower bound in Theorem~\ref{thm:rate} will require more delicate analysis than the above, and will occupy the remainder of this paper.

\begin{rmk}
The previous lemma is the point at which we make crucial use of the very fast mixing of $\P$ via Corollary~\ref{cor:Gibbs-mut-inf}.  In fact, that conclusion is slightly stronger than we need: it would suffice in the above argument to know that
\[\rmI_\nu(\P^{[-N;0)};\P^{[0;N)}) = \rm{o}(\sqrt{N}).\]
However, if one had instead
\[\rmI_\nu(\P^{[-N;0)};\P^{[0;N)}) = \Omega(\sqrt{N}),\]
then this would disrupt all of the subsequent estimates, since it would turn out that the mutual information of $\P^{[-N;0)}$ and $\P^{[0;N)}$ is of the same order as the entropy contained in that part of the scenery that has been visited over both $[-N;0)$ and $[0;N)$.

It would be interesting to know whether the upper bound proved above holds if one knows only that $\bfY$ is a Bernoulli system, that $\P$ is a generating partition, and that $\s$ satisfies the Invariance Principle.  Of course, in this case one may choose an independent generating partition $\P'$ of $\bfY$, but now $\P'$ may not give precise enough control over the cocycle $\s$, in the sense of Lemma~\ref{lem:P-controls-sigma}, to give a usable analog of Lemma~\ref{lem:ball-prods-in-balls}. This is why the argument above needs the precise relation between $\P$ and $\s$ that holds in a well-distributed pair. \fin
\end{rmk}

\section{Meandering of cocycles and discrete Cantor sets}\label{sec:dCs}

For the lower bound in Theorem~\ref{thm:rate}, we will imagine playing as Max-er in the competition described in Subsection~\ref{subs:competition}.  They key to Max-er's strategy will be an `inverse theorem', asserting that for most pairs of sceneries and walk trajectories, if the pair of resulting strings produced by the RWRS process are close, then it must be because of some `structural' similarly involving the sceneries alone.

This section introduces the structures that appear in this notion of similarity.  The first of these is a class of special subsets of certain intervals $[0;N)$ on which our cocycle-trajectories are often (approximately) injective.  These will be introduced after a discussion of another pre-requisite property of cocycle-trajectories.

Throughout this section $(\bfY,\s)$ will be a well-distributed pair. Recall that this means $Y \subseteq A^\bbZ$ is a mixing SFT with a H\"older-potential Gibbs measure $\nu$, and $\s:Y\to \bbR$ is a one-sided H\"older non-coboundary with $\int \s\,\d\nu = 0$ and effective variance $1$.  We assume also that $(\bfY,\s)$ satisfies the Enhanced Invariance Principle (Definition~\ref{dfn:EIP}) with some mollifier $\phi$.  Finally, fix $\ell \geq \|\s\|_\infty$ large enough that $\rm{spt}\,\phi \subseteq [-\ell,\ell]$.

Many of the arguments of this section are adapted from similar steps in~\cite{Kal82}, or their re-telling in~\cite{denHSte97}.

\subsection{Two useful estimates}

The starting point for this section is a pair of basic estimates on the distribution of our cocycle-trajectories.

\begin{lem}\label{lem:one-time-not-sm}
In the above setting, one has
\[\nu\{\s^y_N \in I\} \lesssim_{\bfY,\s} \frac{\max\{\calL^1(I),1\}}{\sqrt{N}} \quad \forall I \in \rm{Int}(\bbR),\ N\in\bbN.\]
In particular,
\[\nu\big\{|\s^y_N| \leq a\sqrt{N}\big\} \lesssim_{\bfY,\s} \max\Big\{a,\frac{1}{\sqrt{N}}\Big\} \quad \forall a\in (0,\infty),N\in\bbN.\]
\end{lem}

\begin{proof}
Let $I = [a,b]$. Theorem~\ref{thm:BE} gives
\begin{eqnarray*}
\nu\{a \leq \s^y_N \leq b\} &=& \nu\{\s^y_N \leq b\} - \nu\{\s^y_N \leq a\}\\
&\lesssim_{\bfY,\s}& \rm{N}(-\infty,b/\sqrt{N}) - \rm{N}(-\infty,a/\sqrt{N}) + \frac{1}{\sqrt{N}},
\end{eqnarray*}
so the result now follows from the smoothness of the Gaussian density.  The second conclusions follows by taking $I = [-a,a]$.
\end{proof}

\begin{lem}\label{lem:max-not-lg}
In the above setting,
\[\nu\big\{\max_{n \in [0;N)}|\s^y_n| \geq b\sqrt{N}\big\} \lesssim_{\bfY,\s} \frac{1}{b^2} \quad \forall b \in (0,\infty),N\in\bbN.\]
\end{lem}

The proof of Lemma~\ref{lem:max-not-lg} is a little more involved.  It begins with the following property of Gibbs measures.

\begin{lem}[Four-fold exponential mixing]\label{lem:four-fol}
With $(\bfY,\s)$ are above, there are some $c < \infty$ and $\g \in (0,1)$ such that
\[\Big|\int \s\cdot (\s \circ S^p)\cdot (\s \circ S^{p+q})\cdot (\s \circ S^{p+q+r})\,\d\nu\Big| \leq c \g^{\max\{p,r\}} \quad \forall p,q,r \in \bbN\cup \{0\}.\]
\end{lem}

\begin{proof}
The definition of a Gibbs measure via equation~(\ref{eq:dfn-Gibbs}) is invariant under time-reversal, as is the conclusion of the present lemma.  It therefore suffices to find $c$ and $\g$ such that
\[\Big|\int \s \cdot (\s \circ S^p)\cdot (\s \circ S^{p+q})\cdot (\s \circ S^{p+q+r})\,\d\nu\Big| \leq c \g^r \quad \forall p,q,r \in \bbN.\]

In this proof, let us re-scale $\s$ so that $\|\s\|_\infty \leq 1$, and let $L_\psi$ be a normalized Ruelle-Perron-Frobenius operator whose adjoint has invariant measure $\nu^-$, as in Subsection~\ref{subs:Gibbs}.

Now let $k := \lfloor r/3\rfloor$.  Since $\s$ is H\"older, there $b_1 < \infty$ and $\b_1 \in (0,1)$ such that
\[|\s(y) - \s(y')| \leq b_1\b_1^{d(y,y')},\]
and so
\[\|\s - \sfE_\nu(\s\,|\,\P^{[-k;k]})\|_\infty \leq b_1\b_1^k \leq (b_1/\b_1)(\sqrt[3]{\b_1})^r.\]
Letting $\s' := \sfE_\nu(\s\,|\,\P^{[-k;k]})$, it follows that
\begin{multline*}
\Big|\int \s\cdot (\s \circ S^p)\cdot (\s \circ S^{p+q})\cdot (\s \circ S^{p+q+r})\,\d\nu\Big|\\ \leq 
\Big|\int \s'\cdot (\s'\circ S^p)\cdot (\s'\circ S^{p+q})\cdot (\s' \circ S^{p+q+r})\,\d\nu\Big| + (b_1/\b_1)(\sqrt[3]{\b_1})^r.
\end{multline*}

It therefore suffices to give an exponentially-decaying bound on the first term here.  Since the whole integral is $S$-invariant, we may reduce instead to the case of $\s'$ being $\P^{[-2k,0]}$-measurable (hence, in particular, one-sided).  Recalling that
\[\sfE_\nu(\s' \circ S^r\,|\,\P^{(-\infty;0]}) = L^r_\psi \s'\]
for one-sided functions $\s'$, this leads to
\[\int \s' \cdot (\s' \circ S^p)\cdot (\s' \circ S^{p+q})\cdot (\s' \circ S^{p+q+r})\,\d\nu = \int \s' \cdot (\s' \circ S^p)\cdot (\s' \circ S^{p+q})\cdot (L^r_\psi \s' \circ S^{p+q})\,\d\nu.\]
Finally, because $L_\psi$ has spectral radius less than $1$ on any space of mean-zero, one-sided H\"older functions (see again~\cite[Theorem 2.2]{ParPol90}), there are $b_2 < \infty$ and $\b_2 \in (0,1)$ such that
\[\Big\|L_\psi^r\s' - \int \s'\,\d\nu\Big\|_\infty \leq b_2 \b_2^k,\]
so this completes the proof.
\end{proof}

\begin{cor}[Fourth-moment bound]\label{cor:fourth-moment}
With $(\bfY,\s)$ as above, one has
\[\|\s^y_N\|_{L^4(\nu)} \lesssim_{\bfY,\s} \sqrt{N} \quad \forall N \in \bbN.\]
\end{cor}

\begin{proof}
Let $c$ and $\g$ be as given by the preceding lemma. Expanding the power inside the $L^4$-norm gives
\begin{eqnarray*}
\|\s^y_N\|^4_{L^4(\nu)} &=& \sum_{n_1,n_2,n_3,n_4 \in [0;N)}\int \s(S^{n_1}y)\s(S^{n_2}y)\s(S^{n_3}y)\s(S^{n_4}y)\,\nu(\d y)\\
&\lesssim& \sum_{p,q,r \geq 0,\ p+q+r < N}\int \s(y)\s(S^py)\s(S^{p+q}y)\s(S^{p+q+r}y)\,\nu(\d y)\\
&\leq& c\sum_{p,q,r \geq 0,\ p+q+r < N}\g^{\max\{p,r\}}\\
&\lesssim_\g& cN^2.
\end{eqnarray*}
\end{proof}

\begin{proof}[Proof of Lemma~\ref{lem:max-not-lg}]
Whenever $m,n \in [0;N)$ with $m \leq n$, the previous corollary gives
\[\int|\s^y_n - \s^y_m|^4\,\nu(\d y) = \int |\s^{S^my}_{n-m}|^4 \,\nu(\d y) \lesssim_{\bfY,\s} |n-m|^2.\]

This moment bound is strong enough to enable a standard chaining argument for controlling $\max_{n \in [0;N)}|\s^y_n|$.  A suitable quantitative version is given by Billingsley as~\cite[Theorem 12.2]{Bil68}.  The bound above is the hypothesis of that theorem with (in his notation) parameters $\g = 4$, $\a = 2$ and $u_i = 1$ for all $i$, and its conclusion becomes
\[\nu\big\{\max_{n \in [0;N)}|\s^y_n| \geq b\sqrt{N}\big\} \lesssim_{\bfY,\s} \frac{N^2}{(b\sqrt{N})^4} = \frac{1}{b^4},\]
which is actually stronger than we require.
\end{proof}

\begin{rmk}
The proof of~\cite[Theorem 12.2]{Bil68} is really a quantitative implementation of Kolmogorov's classical proof that Brownian motion has a continuous version (see, for instance,~\cite[Theorem 3.23]{Kal02}).  The full sequence of arguments above --- from a mixing result (Lemma~\ref{lem:four-fol}), to a fourth-moment bound (Corollary~\ref{cor:fourth-moment}), to an application of Kolmogorov's method --- are essentially the steps taken by Bunimovich and Sinai in~\cite[Section 4]{BunSin81} for their proof that the laws of the random variables $\traj_N(\s^y)$ on $(Y,\nu)$ form a tight sequence in $\Pr C[0,1]$.

An easy extension of Lemma~\ref{lem:four-fol} and Corollary~\ref{cor:fourth-moment} gives
\[\|\s^y_N\|_{L^{2p}(\nu)} \lesssim_{\bfY,\s,p} \sqrt{N} \quad \forall p \in \bbN.\]
The fourth moment is the simply smallest with which the method of Kolmogorov can be applied. \fin
%
\end{rmk}

\subsection{Meandering of cocycle-trajectories}\label{subs:intro-meander}

\begin{dfn}\label{dfn:meander}
Let $\ell,\a > 0$, and let $M,L \in \bbN$.  Let $a \in \bbZ$, let $I := [a;a+LM)$, and let
\[\cal{C} = \big\{[a + iM;a + (i+1)M)\,\big|\ i \in \{0,\ldots,L-1\}\big\}\]
be the partition of this interval into length-$M$ subintervals.

For $y \in Y$, the cocycle-trajectory $\s^y$ is \textbf{$(\a,\ell)$-meandering over $(I,\cal{C})$} if the following holds:
\[\forall \cal{J} \subseteq \cal{C}\ \hbox{with}\ |\cal{J}| \geq \a L \quad \exists J,J' \in \cal{J}\ \hbox{such that}\ \rm{dist}(\s^y_J,\s^y_{J'}) > 2\ell.\]
\end{dfn}

This property gives a sense in which the cocycle-trajectory $\s^y$ has many well-separated images of intervals from $\cal{C}$.  Importantly, once $M$ is large, well-distributed pairs have a strong lower bound on the probability of this occurring.

\begin{prop}\label{prop:spade}
If $(\bfY,\s)$ is well-distributed, then there is a $C < \infty$ such that for all $\ell \in (0,\infty)$ and all sufficiently large $M \in \bbN$, the following holds for all $L \in \bbN$ and all $\a > 0$:
\[\nu\big\{\s^y\ \hbox{is}\ (\a,\ell)\hbox{-meandering over}\ ([0;LM),\cal{C})\big\} \geq 1 - \frac{C}{L^{1/3} \a^2},\]
where $\cal{C}$ is the partition of $[0;LM)$ into subintervals of length $M$.
\end{prop}

\begin{rmks}
\emph{(1.)}\quad This is essentially the lower bound denoted by $\spadesuit$ in~\cite{denHSte97}, except that they require only $\s^y_J \cap \s^y_{J'} = \emptyset$, rather than separation by $2\ell$.  Both the proof and the later applications of this proposition roughly follow their paper, except that we work throughout with well-distributed pairs.

In fact,~\cite{denHSte97} allows a more general lower bound of the form $1 - \frac{C}{L^\g \a^2}$ for some $\g > 0$.  The specific value $\g = 1/3$ arises for random walks in the domain of attraction of a Gaussian distribution, but a smaller value may be needed for a heavier-tailed walk, such as one in the domain of attraction of a $p$-stable law for some $p \in [1,2)$.  Their main results also apply to such random walks, hence their need for this generality, but our focus on well-distributed pairs precludes it.  We will return to this point in Subsection~\ref{subs:other-walks}.

\vspace{7pt}

\emph{(2.)}\quad It will be very important that the value of $M$ at which this inequality starts to hold does not depend on $L$ or $\a$.  \fin
\end{rmks}

\begin{proof}[Proof of Proposition~\ref{prop:spade}]
This proof is essentially as in paragraphs (2.5) and (2.6) of~\cite{Kal82}.

First, $L^{1/3}\a^2 \leq (\a L)^2$ for all $\a > 0$ and $L \in \bbN$.  Therefore, provided $C$ is sufficiently large, the desired bound is vacuous for small values of $\a L$, and we may henceforth assume $\a L \geq 2$, and hence $\binom{\lceil \a L\rceil}{2} \geq (\a L)^2/4$.

Fix $i,j \in [L]$ with $j > i$, and suppose $\b \in (0,1/2)$ (it will be optimized later).  Then Lemmas~\ref{lem:one-time-not-sm} and~\ref{lem:max-not-lg} give
\begin{eqnarray*}
&&\nu\big\{\rm{dist}(\s^y_{[iM;(i+1)M)},\s^y_{[jM;(j+1)M)}) \leq (j-i)^\b\sqrt{M}\big\}\\
&&\leq \nu\big\{|\s^y_{(j-i)M}| \leq 2(j-i)^\b\sqrt{M}\big\} + 2\nu\big\{\max_{n \in [0;M)}|\s^y_n| \geq (j-i)^\b\sqrt{M}\big\}\\
&&\lesssim_{\bfY,\s} \max\Big\{\frac{(j-i)^\b}{\sqrt{(j-i)}},\frac{1}{\sqrt{(j-i)M}}\Big\} + \frac{1}{(j-i)^{2\b}} = \frac{(j-i)^\b}{\sqrt{(j-i)}} + \frac{1}{(j-i)^{2\b}}.
\end{eqnarray*}
Choosing $\b := 1/6$, this last bound becomes $1/(j-i)^{1/3}$.  Therefore, provided $\sqrt{M} > 2\ell$,
\begin{eqnarray*}
&&\nu\big\{\s^y\ \hbox{is not}\ (\a,\ell)\hbox{-meandering over}\ ([0;LM),\cal{C})\big\}\\
&&\leq \nu\big\{\exists J \subseteq [L]\ \hbox{with}\ |J| \geq \a L\ \hbox{such that}\ \forall i,j \in J\ \hbox{distinct one has}\\
&& \quad \quad \quad \quad \quad \quad \quad \quad \rm{dist}(\s^y_{[iM,(i+1)M)},\s^y_{[jM,(j+1)M)}) \leq |j-i|^{1/6}\sqrt{M} \big\}\\
&&\leq \nu\Big\{\exists\ \hbox{at least}\ \binom{\lceil\a L\rceil}{2}\ \hbox{pairs}\ i,j \in [L]\ \hbox{such that}\ j > i\ \hbox{and}\\
&& \quad \quad \quad \quad \quad \quad \quad \quad \rm{dist}(\s^y_{[iM,(i+1)M)},\s^y_{[jM,(j+1)M)}) \leq |j-i|^{1/6}\sqrt{M} \Big\}\\
&& \leq \frac{1}{\binom{\lceil \a L\rceil}{2}}\sum_{1 \leq i < j \leq L}\nu\big\{\rm{dist}(\s^y_{[iM,(i+1)M)},\s^y_{[jM,(j+1)M)}) \leq (j-i)^{1/6}\sqrt{M}\big\}\\
&& \lesssim_{\bfY,\s} \frac{1}{\a^2L^2}\sum_{1 \leq i < j \leq L}\frac{1}{(j-i)^{1/3}} \lesssim \frac{L\cdot L^{2/3}}{\a^2L^2} = \frac{1}{\a^2L^{1/3}},
\end{eqnarray*}
as required.
\end{proof}

\subsection{Regularity for occupation measures}


We next introduce the consequence that we will need of the Enhanced Invariance Principle (Definition~\ref{dfn:EIP}).  Our application of this principle will be essentially the same as made by Aaronson in~\cite{Aar12}, for the special case of random walks: see the `Local Time Lemma' and Lemma 4 of that paper.

\begin{lem}[Smoothness of typical occupation measures]\label{lem:occ-smooth}
Suppose that $(\bfY,\s)$ satisfies the Enhanced Invariance Principle.  Let $\phi$ and $\ell$ be as at the beginning of this section.  For every $\eps > 0$ there is an $M < \infty$ such that
\[\nu\big\{\phi\star \g^y_{[0;N)} \sim_{M,\eps} \rm{U}_{B_\ell(\s^y_{[0;N)})}\big\} \geq 1 - \eps\]
for all sufficiently large $N$.
\end{lem}

\begin{proof}
Because $L^B$ is a.s. a continuous function with support \emph{equal} to the positive-length compact interval $B_{[0,1]}$~(\cite[Corollary 22.18]{Kal02}), for any $\eps > 0$ there are $M < \infty$ and finitely many $K_1,\ldots, K_r \in \rm{Int}(\bbR)$ such that $\calL^1(K_s) > 1/(M\eps)$ for each $s$, and such that the set
\begin{eqnarray*}
U &:=& \bigcup_{s=1}^r \Big\{(K,f) \in \rm{Int}(\bbR)\times C_{\rm{c}}(\bbR)\,\Big|\\
&&\quad \quad \quad \quad \quad \quad K_s \subseteq \rm{interior}(K) \subseteq K \subseteq K_s + (-1/2M,1/2M),\\
&&\quad \quad \quad \quad \quad \quad \quad \quad \quad \quad \quad \quad 2/M < f|_{K_s} < M/2,\ \hbox{and}\ \int_{K_s} f > 1 - \eps\Big\}
\end{eqnarray*}
satisfies
\[\sfW_{[0,1]}\{B\,|\ (B_{[0,1]},L^B) \in U\} > 1-\eps.\]

This last subset of $C_0(0,1]$ is easily checked to be open.  On the other hand, the Enhanced Invariance Principle implies that
\[\big(B_\ell(\s^y_{[0;N)})/\sqrt{N},\ \phi\star \g^y_{[0;N)}(\sqrt{N}(\cdot))\big) \stackrel{\rm{law}}{\to} (B_{[0,1]},L^B)\]
as $(\rm{Int}(\bbR)\times C_{\rm{c}}(\bbR))$-valued random variables.  Therefore the Portmanteau Theorem implies that
\[\nu\big\{\big(B_\ell(\s^y_{[0;N)})/\sqrt{N},\ \phi\star \g^y_{[0;N)}(\sqrt{N}(\cdot))\big) \in U\big\}> 1 - \eps\]
for all sufficiently large $N$.  Since $\phi\star \g^y_{[0;N)}(\sqrt{N}(\cdot))$ is always non-negative and has integral equal to $1$, this membership of $U$ implies that
\[\phi \star \g^y_{[0;N)} \sim_{M,\eps} \rm{U}_{B_\ell(\s^y_{[0;N)})},\]
as required.
\end{proof}

Based on the preceding results, we can introduce certain sets that occur with high probability as the relevant time scale tends to $\infty$.  With $\phi$ and $\ell$ as before, and for $[a;b) \subseteq \bbZ$, set
\[Y^{\rm{spread}}_{[a;b)} := \big\{y\,\big|\ [\s^y_a -(b-a)^{1/3}, \s^y_a + (b-a)^{1/3}] \subseteq B_\ell(\s^y_{[a;b)})\big\}.\]
Given also $\eps > 0$ and $M < \infty$, set
\[Y^{\rm{smooth}}_{[a;b),M,\eps} := \big\{y\,\big|\ \phi \star \g^y_{[a;b)} \sim_{M,\eps} \rm{U}_{B_\ell(\s^y_{[a;b)})}\big\}.\]

The results above show that
\begin{eqnarray}\label{eq:spread-whp}
\nu(Y^{\rm{spread}}_{[a;b)}) \to 1 \quad \hbox{as}\ b-a \to \infty
\end{eqnarray}
(indeed, this would follow if the exponent $1/3$ were replaced with any value less than $1/2$), and that
\begin{eqnarray}\label{eq:smooth-whp}
\nu(Y^{\rm{smooth}}_{[a;b),M,\eps}) \to 1
\end{eqnarray}
as $b-a\to \infty$ and then $M\to\infty$.

\subsection{A discrete filtration}\label{subs:disc-filt}

We will next define recursively some sequences of auxiliary parameters and sets.

Firstly, let
\[L_d := \lceil (d+1)^{18}\rceil \quad \forall d\geq 1\]
(the reason for the exponent $18$ will emerge shortly). Let $N_0 := 1$, and then let
\[N_{d+1} := L_{d+1}N_d \quad \forall d\geq 0.\]
A crude application of Stirling's Inequality gives
\begin{eqnarray}\label{eq:N_s-est}
\log N_d \sim (d+1)\log (d+1) \quad \forall d \geq 1.
\end{eqnarray}

In addition, let
\[\a_d := \frac{1}{(d+1)^2}\]
and
\[\k_{r,d} := \prod_{i=r+1}^d(1 - \a_i) \quad \hbox{whenever}\ 1 \leq r \leq d.\]
The series $\sum_i \a_i$ converges, and therefore $\k_{r,d}$ tends to a limit $\k_{r,\infty} \in [0,1)$ as $d \to\infty$, and then $\k_{r,\infty} \uparrow 1$ as $r\to\infty$.

Also, for each $d\geq 0$, let $\cal{D}_d$ be the partition of $\bbZ$ into the discrete intervals
\[Q^d_i := [iN_d;(i+1)N_d) \quad \hbox{for}\ i \in \bbZ.\]
Given $t \in \bbZ$, let $Q^d(t)$ denote the cell of $\cal{D}_d$ that contains $t$, so $Q^d(t) = Q^d_i$ if and only if $iN_d \leq t < (i+1)N_d$.

Letting $\ell$ be as before, for each $d \geq 1$ and $t \in \bbZ$ let
\[Y_{d,t}^{\rm{mndr}}:= \big\{y \in Y\,\big|\ \s^y\ \hbox{is}\ (\a_d,\ell)\hbox{-meandering over}\ (Q^d(t),\cal{D}_{d-1}\cap Q^d(t)) \big\}\]
(this is essentially a repeat of the definition of the set `$\theta_d'$' in~\cite[Section 6.1]{denHSte97}; interestingly, our argument does not seem to need any analog of the sets `$\theta'''_d$' from that paper).  Observe that $Y^{\rm{mndr}}_{d,t}$ actually depends on $t$ only through $Q^d(t)$, hence only on $i := \lfloor t/N_d\rfloor$.  For this $i$, one has
\[Y^{\rm{mndr}}_{d,t} = S^{-iN_d}(Y^{\rm{mndr}}_{d,0}),\]
and so Proposition~\ref{prop:spade} gives
\begin{multline}\label{eq:X-good-big}
\nu(Y^{\rm{mndr}}_{d,t}) = \nu(Y^{\rm{mndr}}_{d,0}) > 1 - \frac{C}{L_d^{1/3}\a_d^2}\\ \geq 1 - \frac{C}{(d+1)^{18/3}\frac{1}{(d+1)^4}} = 1 - \frac{C}{(d+1)^2},
\end{multline}
with $C$ as in that proposition, for all $d$ and $t$.  (This explains the choice of the exponent $18$: we will soon need the error term at the end here to be summable in $d$.)

Given $1 \leq s \leq d$ and $Q \in \cal{D}_d$, let
\[H^{\rm{mndr}}_{s,Q}(y) := \{t \in Q\,|\ y \in Y^{\rm{mndr}}_{s,t}\}.\]
For a given $y$, this is the set of times $t \in Q$ such that $\s^y$ `behaves well', in the sense of Proposition~\ref{prop:spade}, over the interval $Q^s(t)$. Since $Y^{\rm{mndr}}_{s,t}$ depends on $t$ only through $Q^s(t)$, the set $H^{\rm{mndr}}_{s,Q}(y)$ is a union of cells from $\cal{D}_s\cap Q$: it is equal to
\[\bigcup\big\{[iN_s;(i+1)N_s)\,\big|\ i \in \bbZ,\ [iN_s;(i+1)N_s) \subseteq Q,\ \hbox{and}\ S^{iN_s}y \in Y^{\rm{mndr}}_{s,0}\big\}.\]

The analysis below will need cocycle-trajectories that are simultaneously `well-behaved' over most intervals $Q^s(t) \subseteq Q \in \cal{D}_d$ on all sufficiently large scales up to some $d$.  A high probability of such cocycle-trajectories is given by the following lemma.

\begin{lem}\label{lem:cal-T-big}
If $1 \leq r < d$ and $Q \in \cal{D}_d$, then
\[\int_Y \Big|\bigcap_{s=r+1}^d H^{\rm{mndr}}_{s,Q}(y)\Big|\ \nu(\d y) \geq \Big(1 - C\sum_{s=r+1}^d\frac{1}{(s+1)^2}\Big)N_d\]
with $C$ as in Proposition~\ref{prop:spade}.
\end{lem}

\begin{proof}
For each fixed $s \in [1;d]$, the bound~(\ref{eq:X-good-big}) gives
\begin{multline*}
\int_Y |Q\setminus H^{\rm{mndr}}_{s,Q}(y)|\,\nu(\d y) = \sum_{t \in Q}\nu\{y\,|\ t \not\in H^{\rm{mndr}}_{s,Q}(y)\}\\ = N_d\nu(Y \setminus Y^{\rm{mndr}}_{s,0}) < \frac{CN_d}{(s+1)^2},
\end{multline*}
and therefore
\begin{multline*}
\int_Y \Big|\bigcap_{s=r+1}^d H^{\rm{mndr}}_{s,Q}(y)\Big|\ \nu(\d y) \geq N_d - \sum_{s=r+1}^d\int_Y |Q\setminus H^{\rm{mndr}}_{s,Q}(y)|\,\nu(\d y)\\ > \Big(1 - C\sum_{s=r+1}^d\frac{1}{(s+1)^2}\Big)N_d.
\end{multline*}
\end{proof}

We can introduce other sets of times in a fixed large interval $Q$ at which a cocycle-trajectory is behaving well in one sense or another.  Fix again $1 \leq r < d$ and $Q \in \cal{D}_d$.  Then we set
\[H^{\rm{spread}}_{r,Q}(y) := \{t \in Q\,|\ y \in Y^{\rm{spread}}_{Q^r(t)}\}\]
and
\[H^{\rm{smooth}}_{r,Q,M,\eps}(y) := \{t \in Q\,|\ y \in Y^{\rm{smooth}}_{Q^r(t),M,\eps}\}.\]
Now we may combine the estimates~(\ref{eq:spread-whp}) and~(\ref{eq:smooth-whp}) with Lemma~\ref{lem:cal-T-big} and Markov's Inequality to conclude the following.

\begin{cor}\label{cor:cal-T-big}
For every $\eta > 0$ and $\eps > 0$ there are $r,M \in \bbN$ such that, for every $d > r$ and every $Q \in \cal{D}_d$, one has
\[\nu\Big\{y \in Y\,\Big|\ \Big|H^{\rm{spread}}_{r,Q}(y)\cap H^{\rm{smooth}}_{r,Q,M,\eps}(y)\cap \bigcap_{s=r+1}^d H^{\rm{mndr}}_{s,Q}(y)\Big| \geq (1 - \eta)N_d\Big\} > 1 - \eta.\]
\qed
\end{cor}

The largeness of the intersection of times appearing here implies that the cocycle-trajectory $\s^y$ enjoys several different useful properties simultaneously on most of $Q$.  The conjunction of all of these properties will play a r\^ole in our later analysis of generalized RWRS systems.

\subsection{Discrete Cantor sets and approximate injectivity}\label{subs:intro-dCs}

This subsection will focus on the intersection $\bigcap_{s=r+1}^d H^{\rm{good}}_{s,Q}(y)$, as appears in Corollary~\ref{cor:cal-T-big}.  Provided it is large enough, one can find special, highly-structured subsets inside it on which $\s^y$ is approximately injective.

\begin{dfn}[Discrete Cantor sets]\label{dfn:dCs}
Let $d \in \bbN$.  A \textbf{discrete Cantor set} of \textbf{depth} $d$ is an indexed family $(t_\omega)_{\omega \in \{0,1\}^d}$ of points in $\bbR$, and it is \textbf{proper} if they are all distinct.

In addition, given $D_1 \geq D_2 \geq \ldots \geq D_d > 0$, these are \textbf{gap upper bounds} for the discrete Cantor set if
\[|t_\omega - t_{\omega'}| \leq D_{i+1} \quad \hbox{whenever}\ i < d\ \hbox{and}\ \omega_j = \omega'_j\ \forall j \leq i.\]

Given $K \in \rm{Int}(\bbR)$, $d \in \bbN$ and $D = (D_1\geq \ldots \geq D_d)$ as above, we will let $\DCS_{d,D}(K) \subseteq K^{\{0,1\}^d}$ denote the collection of all discrete Cantor sets of depth $d$, contained in $K$, and having gap upper bounds given by $D$.
\end{dfn}

Our first result about discrete Cantor sets is an estimate on their `number'; or, more correctly, their covering number for some natural metric.  We will endow $\DCS_{d,D}(K)$ with the metric $d_\DCS$ obtained from the norm $\|\cdot\|_\infty$ on $\bbR^{\{0,1\}^d}$: that is,
\[d_\DCS((x_\omega)_\omega,(y_\omega)_\omega) = \max_{\omega \in \{0,1\}^d}|x_\omega - y_\omega|.\]

\begin{lem}[Bounding the number of discrete Cantor sets]\label{lem:bound-dCs}
Let $K \in \rm{Int}(\bbR)$ with $L:= \calL^1(K)$, and fix gap upper bounds $D = (D_1\geq D_2 \geq \ldots \geq D_d)$.  Suppose $\delta \leq D_d/10,L/10$.  Then
\[\cov\big((\DCS_{d,D}(K),d_\DCS),\delta\big) \leq \frac{2L}{\delta}\Big(\frac{2D_1}{\delta}\Big)\Big(\frac{2D_2}{\delta}\Big)^2\cdots \Big(\frac{2D_d}{\delta}\Big)^{2^{d-1}}.\]
\end{lem}

\begin{proof}
The desired inequality is invariant under re-scaling $\bbR$, so we may simply assume that $\delta = 2$ and that $D_d,L \geq 20$.

Let $\Phi:\bbR\to \bbZ$ be the discretization map
\[\Phi(x) := \lfloor x\rfloor.\]
Clearly if $(x_\omega)_\omega,(y_\omega)_\omega \in \DCS_{d,D}(K)$ and $(\Phi(x_\omega))_\omega = (\Phi(y_\omega))_\omega$ then
\[d_\DCS((x_\omega)_\omega,(y_\omega)_\omega) < 2.\]
It therefore suffices bound the cardinality of the set $\Phi^{\times \{0,1\}^d}(\DCS_{d,D}(K))$.

Let $\ol{K}$ be a closed interval with end-points in $\bbZ$ that contains $K$ and has length at most $4L/3$.  Using that $D_d,L\geq 20$, one sees that
\[\Phi^{\times \{0,1\}^d}(\DCS_{d,D}(K)) \subseteq \DCS_{d,4D/3}(\ol{K}) \cap \bbZ^{\{0,1\}^d},\]
so it suffices to bound the cardinality of this right-hand set by the desired product.

In the base case, $d=1$, the set $\DCS_{d,4D/3}(\ol{K}) \cap \bbZ^{\{0,1\}^d}$ just consists of pairs $(x_0,x_1)$ in $\ol{K}\cap \bbZ$ separated by distance at most $4D_1/3$, and there are at most $(2L)(2D_1)$ of these.

Now, for the recursion clause, suppose the result is known for all depths less than some $d \geq 2$, and consider a discrete Cantor set $(x_\omega)_\omega \in \DCS_{d,4D/3}(\ol{K})\cap \bbZ^{\{0,1\}^d}$.  It may be identified with the pair of depth-$(d-1)$ discrete Cantor sets
\[(x_{0\omega})_{\omega \in \{0,1\}^{d-1}} \quad \hbox{and} \quad (x_{1\omega})_{\omega \in \{0,1\}^{d-1}}.\]
Let
\[K_i := \min_{\omega \in \{0,1\}^{d-1}}x_{i\omega} + [0,4D_2/3] \quad \hbox{for}\ i=0,1,\]
and let $D' := (D_2,\ldots,D_d)$. In view of the gap upper bounds, the above two depth-$(d-1)$ discrete Cantor sets are members of
\[\DCS_{d-1,4D'/3}(K_0)\cap \bbZ \quad \hbox{and} \quad \DCS_{d-1,4D'/3}(K_1)\cap \bbZ,\]
respectively.

Therefore the cardinality of $\DCS_{d,4D/3}(\ol{K})\cap \bbZ$ is bounded by the number of possible choices of $K_1$ and $K_2$, multiplied by the square of $|\DCS_{d-1,4D'/3}([0,4D_2/3])|$.  By the base-case argument and the inductive hypothesis, this is bounded by
\begin{eqnarray*}
(2L)(2D_1)|\DCS_{d-1,2D'}([0,4D_2/3])|^2 &\leq& (2L)(2D_1)\big((2D_2)(2D_3)^2\cdots (2D_d)^{2^{d-2}}\big)^2\\
&\leq& (2L)(2D_1)(2D_2)^2(2D_3)^4\cdots (2D_d)^{2^{d-1}},
\end{eqnarray*}
as required.
\end{proof}

\begin{dfn}[Discrete Cantor families]
A \textbf{discrete Cantor family} of \textbf{depth} $d$ is an indexed family $(K_\w)_{\w \in \{0,1\}^d}$ of pairwise-disjoint members of $\rm{Int}(\bbR)$.  It has \textbf{gap upper bounds} $D_1 \geq \ldots \geq D_d$ if
\[\rm{diam}(K_\w\cup K_{\w'}) \leq D_{i+1} \quad \hbox{whenever}\ i < d\ \hbox{and}\ \w_j = \w'_j\ \forall j \leq i.\]
If $\mathfrak{K} = (K_\w)_{\w\in \{0,1\}^d}$ is such a discrete Cantor family, then its \textbf{domain} is
\[\dom(\mathfrak{K}) := \bigcup_{\w \in \{0,1\}^d} K_\w\]
(the reason for this terminology will become clear later).

We will let $\DCF_{d,D}(K) \subseteq \rm{Int}(K)^{\{0,1\}^d}$ denote the collection of all discrete Cantor families of depth $d$, contained in $K$, and having gap upper bounds given by $D$.
\end{dfn}

It is clear that if $(K_\omega)_\omega \in \DCF_{d,D}(K)$, then
\[(\min K_\omega)_\omega \quad \hbox{and} \quad (\max K_\omega)_\omega \in \DCS_{d,D}(K).\]
Similarly to $d_\DCS$, we will endow $\DCF_{d,D}(K)$ with the metric
\begin{multline*}
d_{\DCF}\big((K_\omega)_\omega,(K'_\omega)_\omega\big) := \max_{\omega \in \{0,1\}^d}d_{\rm{Hdf}}(K_\w,K'_\w)\\
= \max\big\{d_\DCS\big((\min K_\omega)_\omega,(\min K'_\omega)_\omega\big), d_\DCS\big((\max K_\omega)_\omega,(\max K'_\omega)_\omega\big)\big\},
\end{multline*}
where $d_{\rm{Hdf}}$ is the classical Hausdorff metric on the space of nonempty compact subsets. Lemma~\ref{lem:bound-dCs} immediately gives the following.

\begin{cor}[Bounding the number of discrete Cantor families]\label{cor:bound-dCf}
Let $K \in \rm{Int}(\bbR)$ with $L:= \calL^1(K)$, and fix gap upper bounds $D = (D_1\geq D_2 \geq \ldots \geq D_d)$.  Suppose $\delta \leq D_d/10,L/10$.  Then
\[\cov\big((\DCF_{d,D}(K),d_\DCF),\delta\big) \leq \Big(\frac{2L}{\delta}\Big(\frac{2D_1}{\delta}\Big)\Big(\frac{2D_2}{\delta}\Big)^2\cdots \Big(\frac{2D_d}{\delta}\Big)^{2^{d-1}}\Big)^2.\]
\qed
\end{cor}

We will also need to know how certain discrete Cantor families relate to the filtration $\cal{D}_\bullet$ introduced previously.  If $1 \leq r \leq d$ and $(K_\omega)_{\omega \in \{0,1\}^{d-r}}$ is a discrete Cantor family, then it is \textbf{adapted} to $(\cal{D}_d,\cal{D}_{d-1},\ldots,\cal{D}_r)$ if
\begin{itemize}
\item each $K_\omega \in \cal{D}_r$,
\item if $s \in [0;d-r]$ and $\omega,\omega' \in \{0,1\}^{d-r}$ satisfy $\omega_i = \omega'_i$ for all $i \in [1;s]$, then
\[\cal{D}_{d-s}(K_\omega) = \cal{D}_{d-s}(K_{\omega'})\]
(including the case $s = 0$, when the assumption is vacuous),
\item but if $\omega_s \neq \omega_{s'}$ for some $s \in [1;d-r]$, then
\[\cal{D}_{d-s}(K_\omega) \neq \cal{D}_{d-s}(K_{\omega'}).\]
\end{itemize}

We now turn to the main result of this subsection, which provides discrete Cantor families on which a given cocycle-trajectory is (approximately) injective.

\begin{prop}[Finding a good discrete Cantor set for a good trajectory]\label{prop:gooddCs}
Let $1 \leq r < d$, let $Q \in \cal{D}_d$, let $y \in Y$, and suppose that $\cal{J} \subseteq \cal{D}_r\cap Q$ is a family of intervals such that
\[|\cal{J}| \geq (1 - \kappa_{r,d})N_d/N_r\]
and
\[\mcup \cal{J} \subseteq \bigcap_{s=r+1}^d H^{\rm{mndr}}_{s,Q}(y).\]
Then there is a discrete Cantor family $(Q_\omega)_{\omega \in \{0,1\}^{d-r}}$ contained in $\cal{J}$ and adapted to $(\cal{D}_d,\ldots,\cal{D}_r)$ such that the images $B_\ell(\s^y_{Q_\omega})$, $\omega \in \{0,1\}^{d-r}$, are also pairwise-disjoint.
\end{prop}

\begin{proof}
This is proved by induction on $d$.

\vspace{7pt}

\emph{Base clause.}\quad When $d=r+1$, our assumptions are
\[|\cal{J}| \geq (1 - \a_{r+1})L_{r+1}\]
and
\[\mcup\cal{J} \subseteq H^{\rm{mndr}}_{r+1,Q}(y).\]
The first of these implies that $\cal{J} \neq \emptyset$, and hence the second implies that also $H^{\rm{mndr}}_{r+1,Q}(y) \neq \emptyset$.  However, this is possible only if $H^{\rm{mndr}}_{r+1,Q}(y) = Q$, and hence $S^{iN_{r+1}}y \in Y^{\rm{mndr}}_{r+1,0}$ where $Q = [iN_{r+1};(i+1)N_{r+1})$.  Since $1 - \a_{r+1} \geq \a_{r+1}$, the definition of $Y^{\rm{mndr}}_{r+1,0}$ gives two (necessarily disjoint) intervals $Q_0,Q_1 \in \cal{J}$ for which $B_\ell(\s^y_{Q_0})$ and $B_\ell(\s^y_{Q_1})$ are also disjoint.

\vspace{7pt}

\emph{Recursion clause.}\quad Now suppose $d > r+1$, and that the result is already known at scales up to $d-1$.  Let $\cal{J}$ satisfy the two assumptions, and let
\[\cal{J}_R := \{K \in \cal{J}\,|\ K\subseteq R\}\]
for each $R \in \cal{D}_{d-1}\cap Q$.

Our first assumption about $\cal{J}$ gives
\[|(\cal{D}_r\cap Q)\setminus \cal{J}| = \sum_{R \in \cal{D}_{d-1}\cap Q}|(\cal{D}_r\cap R)\setminus \cal{J}_R| \leq \k_{r,d}N_d/N_r,\]
and so Markov's Inequality implies that the set
\[\cal{J}' := \big\{R \in \cal{D}_{d-1}\cap Q\,\big|\ |\cal{J}_R| \geq (1 - \k_{r,d-1})N_{d-1}/N_r\big\}\]
has cardinality at least $\a_d L_d$.

Our second assumption about $\cal{J}$ requires that $H^{\rm{mndr}}_{d,Q}(y) \neq \emptyset$, and hence instead $H^{\rm{mndr}}_{d,Q}(y) = Q$.   We may therefore apply the base-clause argument to the family $\cal{J}' \subseteq \cal{D}_{d-1}\cap Q$ to obtain a pair of intervals $Q_0,Q_1 \in \cal{J}'$ such that
\[B_\ell(\s^y_{Q_0}) \cap B_\ell(\s^y_{Q_1}) = \emptyset.\]

On the other hand, by the definition of $\cal{J}'$, for each $i\in \{1,2\}$ we have
\[|\cal{J}_{Q_i}|\geq (1 - \k_{r,d-1})N_{d-1}/N_r,\]
and also
\begin{multline*}
\mcup \cal{J}_{Q_i} = Q_i\cap \big(\mcup \cal{J}\big) \subseteq Q_i\cap \Big(\bigcap_{s=r+1}^d H^{\rm{mndr}}_{s,Q}(y)\Big)\\ \subseteq \bigcap_{s=r+1}^{d-1}(Q_i\cap H^{\rm{mndr}}_{s,Q}(y)) = \bigcap_{s=r+1}^{d-1}H^{\rm{mndr}}_{s,Q_i}(y).
\end{multline*}

We may therefore apply the inductive hypothesis to each of $\cal{J}_{Q_0}$ and $\cal{J}_{Q_1}$ to obtain Cantor families $(Q_{0\omega})_{\omega \in \{0,1\}^{d-r-1}}$ and $(Q_{1\omega})_{\omega \in \{0,1\}^{d-r-1}}$ inside them with the asserted properties.  Assembling these into a single family shows that the induction continues, and hence completes the proof.
\end{proof}


\begin{rmk}
Proposition~\ref{prop:gooddCs} is a finitary cousin of the classical result that if $C\subseteq [0,1]$ has Hausdorff dimension at most $\frac{1}{4}$, then a.e. Brownian sample path is injective on $C$: see Section 16.6 in Kahane~\cite{Kah85}, up to Theorem 6 of that section.  An unusual feature of the present setting is that our partitions of $\bbZ$ are increasingly coarse, and each cell of $\cal{D}_d$ contains roughly $d^2$ cells of $\cal{D}_{d-1}$, so the `index' of $\cal{D}_{d-1}$ in $\cal{D}_d$ tends to $\infty$ with $d$.  This is different from the classical analysis of self-crossing for Brownian motion on $[0,1]$, which is easiest using simply dyadic partitions. \fin
\end{rmk}

\subsection{Approximate covering with discrete Cantor families}

Let $(\bfY,\s)$, $\phi$ and $\ell$ be as before.

\begin{prop}\label{prop:pre-key}
For every $\beta,\eta > 0$ there are $M < \infty$, $r_0 \in \bbN$ and a family of subsets $Y^{\rm{good}}_{r,d} \subseteq Y$, $d > r \geq r_0$, satisfying $\nu(Y^{\rm{good}}_{r,d}) > 1 - \beta$ and such that the following holds. For every $r \geq r_0$ there is a $\delta > 0$ such that if $d > r$, $y \in Y^{\rm{good}}_{r,d}$ and $P \subseteq [0;N_d)$ with $|P| > (1 - \delta)N_d$ then there is a collection $\cal{G}$ of discrete Cantor families of depth $d-r$, all contained in $P$ and subordinate to $(\cal{D}_d,\cal{D}_{d-1},\ldots,\cal{D}_{r+1})$, for which the following hold:
\begin{enumerate}
\item[1)] (images are not too short) for every $(Q_\w)_\w \in \cal{G}$ and all $\w \in \{0,1\}^{d-r}$, one has
\[[\s^y_{\min Q_\w} - N_r^{1/3},\s^y_{\min Q_\w} + N_r^{1/3}] \subseteq B_\ell(\s^y_{Q_\w});\]
\item[2)] (well-separated images) for every $(Q_\w)_\w \in \cal{G}$, the images $B_\ell(\s^y_{Q_\w})$ for $\w \in \{0,1\}^{d-r}$ are pairwise disjoint (beware that this is not asserting any disjointness among images from distinct members of $\cal{G}$);
\item[3)] (cocycle-range is mostly covered) one has
\[\cal{L}^1\Big( B_\ell(\s^y_{[0;N_d)}) \Big\backslash\bigcup_{(Q_\w)_\w \in \cal{G}}\bigcup_{\w \in \{0,1\}^{d-r}}B_\ell(\s^y_{Q_\w})\Big) \leq \eta \calL^1\big(B_\ell(\s^y_{[0;N_d)})\big);\]
\item[4)] (cocycle-range is covered fairly efficiently) one has
\[\sum_{(Q_\w)_\w \in \cal{G}}\sum_{\w \in \{0,1\}^{d-r}}\calL^1(B_\ell(\s^y_{Q_\w})) \leq M\cal{L}^1\big(B_\ell(\s^y_{[0;N_d)})\big).\]
\end{enumerate}
\end{prop}

It seems worth emphasizing here that the final choice of $\delta$ must be allowed to depend on $r \geq r_0$.

\begin{proof}
\quad\emph{Step 1: Choice of parameters and sets.}\quad We start with the selection of $M$ and $r_0$, along with some auxiliary parameters that will be used during the proof.  We will then construct the sets $Y^{\rm{good}}_{r,d}$.

The parameters are given by the following choices (which will be motivated during the course of the proof):
\begin{itemize}
\item[C1)] Let $\eps_1 := \eta/2$, and choose $r_{0,1} \in \bbN$ and $M_1 < \infty$ according to the convergence~(\ref{eq:smooth-whp}) so that $\nu(Y^{\rm{smooth}}_{[0;N_d),M_1,\eps_1}) > 1 - \beta/2$ for all $d > r_{0,1}$.
\item[C2)] Let $\delta_1 := \eta/4M_1$.
\item[C3)] Choose $r_{0,2}$ so large that
\[\k_{r_{0,2},\infty} > \max\Big\{1 - \frac{\eta(1 - \delta_1)}{8M_1},\delta_1\Big\}.\]
This implies that also
\[1 - \k_{r,d} < \frac{\eta(1 - \delta_1)}{8M_1} \quad \hbox{and} \quad \delta_1 < \k_{r,d} \quad \hbox{whenever}\ d >r \geq r_{0,2}.\]
\item[C4)] Let $\eps_2 := \eta/16M_1$, and let $r_{0,3} \in \bbN$ and $M_2 < \infty$ be given by Corollary~\ref{cor:cal-T-big} so that for every $d >r \geq r_0$ the set
\begin{multline*}
Y^{\rm{good},1}_{r,d}\\
:= \Big\{y\,\Big|\ \Big|H^{\rm{spread}}_{r,[0;N_d)}(y)\cap H^{\rm{smooth}}_{r,[0;N_d),M_2,\eps_2}(y)\cap \bigcap_{s=r+1}^d H^{\rm{mndr}}_{s,[0;N_d)}(y)\Big| \geq (1 - \delta_1/2)N_d\Big\}
\end{multline*}
has
\[\nu(Y^{\rm{good},1}_{r,d}) > 1 - \beta/2.\]
\item[C5)] Finally, let $M := 16M_1M_2/\eta$ and $r_0 := \max\{r_{0,1},r_{0,2},r_{0,3}\}$.
\end{itemize}

Now, for each $d > r$, let
\[Y^{\rm{good}}_{r,d} := Y^{\rm{smooth}}_{[0;N_d),M_1,\eps_1}\cap Y^{\rm{good},1}_{r,d},\]
where the second of these right-hand sets was introduced in choice (C4) above.  Choices (C1) and (C4) together imply that
\[\nu(Y^{\rm{good}}_{r,d}) > 1 - \beta.\]

It remains to prove that these choices give the desired consequences.

\vspace{7pt}

\emph{Step 2: Finding a large family of good intervals.}\quad Now fix $d > r \geq r_0$, and let
\[\delta := \frac{\delta_1}{2 N_r}.\]
We will complete the proof with this choice of $\delta$.  Thus, suppose that $y \in Y^{\rm{good}}_{r,d}$ and that $P \subseteq [0;N_d)$ with $|P| > (1 - \delta)N_d$.

Consider the family of intervals
\[\cal{J}_0 := \{Q \in \cal{D}_r\,|\ Q\subseteq P\}.\]
For this family, one has
\begin{multline*}
|(\cal{D}_r\cap [0;N_d))\setminus \cal{J}_0| \leq \sum_{Q \in (\cal{D}_r\cap [0;N_d))\setminus \cal{J}_0}|Q\setminus P|\\ \leq \sum_{Q \in \cal{D}_r\cap [0;N_d)}|Q\setminus P| = |[0;N_d)\setminus P| < \delta N_d,
\end{multline*}
and hence
\[|\cal{J}_0| > (1 - \delta N_r)N_d/N_r = (1 - \delta_1/2)N_d/N_r.\]

Combining this bound with the fact that $y \in Y^{\rm{good}}_{r,d} \subseteq Y^{\rm{good},1}_{r,d}$, it follows that the family
\[\cal{J} := \Big\{Q \in \cal{D}_r\,\Big|\ Q \subseteq P\cap H^{\rm{spread}}_{r,[0;N_d)}(y)\cap H^{\rm{smooth}}_{r,[0;N_d),M_2,\eps_2}(y)\cap \bigcap_{s=r+1}^d H^{\rm{mndr}}_{s,[0;N_d)}(y)\Big\}\]
has
\[|\cal{J}| \geq (1 - \delta_1)N_d/N_r,\]
and this is greater than $(1 - \k_{r,d})N_d/N_r$ by choice (C3).

\vspace{7pt}

\emph{Step 3: Constructing the Cantor families}\quad Having introduced $\cal{J}$, Proposition~\ref{prop:gooddCs} gives the ability to find a discrete Cantor family inside $\cal{J}$ on which the cocycle-trajectories $\s^y$ is approximately injective.  It only remains to show how a careful repeated appeal to Proposition~\ref{prop:gooddCs} can produce a whole collection of discrete Cantor families with the desired properties.  This is achieved by the following recursion.

\vspace{7pt}

\quad\emph{Base step.}\quad First, with $\cal{J}$ as above, and in view of choice (C3), Proposition~\ref{prop:gooddCs} gives a discrete Cantor family $\mathfrak{Q}_1 := (Q_\w)_{\w \in \{0,1\}^{d-r}}$ contained in $\cal{J}$ such that
\begin{itemize}
\item[i)] the intervals $B_\ell(\s^y_{Q_\w})$ for $\w \in \{0,1\}^{d-r}$ are pairwise disjoint.
\end{itemize}
Moreover, since $\bigcup\cal{J}\subseteq H^{\rm{spread}}_{r,[0;N_d)}(y)$, we also have
\begin{itemize}
\item[ii)] $[\s^y_{\min Q_\w} - N_r^{1/3},\s^y_{\min Q_\w} + N_r^{1/3}] \subseteq B_\ell(\s^y_{Q_\w})$ for every $\w$.
\end{itemize}

\vspace{7pt}

\quad\emph{Recursion step.}\quad Now suppose that discrete Cantor families $\mathfrak{Q}_1$, \ldots, $\mathfrak{Q}_m$ in $\cal{J}$ have already been constructed for some $m \geq 1$.  If
\begin{eqnarray}\label{eq:domsbigyet?}
\cal{L}^1\Big(B_\ell(\s^y_{[0;N_d)}) \Big\backslash \bigcup_{s\leq m}B_\ell(\s^y_{\dom(\mathfrak{Q}_s)})\Big) \leq \eta\cal{L}^1\big(B_\ell(\s^y_{[0;N_d)})\big),
\end{eqnarray}
then stop the recursion, and let $\cal{G} := \{\mathfrak{Q}_1,\ldots,\mathfrak{Q}_m\}$.  In that case the construction is finished.  Otherwise, suppose the opposite inequality of~(\ref{eq:domsbigyet?}).  Since $y \in Y^{\rm{good}}_{r,d} \subseteq Y^{\rm{smooth}}_{[0;N_d),M_1,\eps_1}$, the opposite of~(\ref{eq:domsbigyet?}) implies that
\begin{eqnarray}\label{eq:meas-comp-big}
(\phi \star \g^y_{[0;N_d)})\Big(B_\ell(\s^y_{[0;N_d)}) \Big\backslash \bigcup_{s\leq m}B_\ell(\s^y_{\dom(\mathfrak{Q}_s)})\Big) > \frac{\eta - \eps_1}{M_1} \geq \frac{\eta}{2M_1}.
\end{eqnarray}

Since
\[\phi \star \g^y_{[0;N_d)} = \frac{N_r}{N_d}\sum_{Q \in \cal{D}_r\cap [0;N_d)} \phi \star \g^y_Q\\ \ll_{1,\delta_1} \frac{1}{|\cal{J}|}\sum_{Q \in \cal{J}} \phi \star \g^y_Q,\]
inequality~(\ref{eq:meas-comp-big}) gives that
\[\frac{1}{|\cal{J}|}\sum_{Q \in \cal{J}}( \phi \star \g^y_Q)\Big( B_\ell(\s^y_{[0;N_d)}) \Big\backslash \bigcup_{s\leq m}B_\ell(\s^y_{\dom(\mathfrak{Q}_s)})\Big) > \frac{\eta}{2M_1} - \delta_1 > \frac{\eta}{4M_1},\]
by the choice of $\delta_1$ in (C2).  Therefore, letting
\[\cal{J}' := \Big\{Q \in \cal{J}\,\Big|\ (\phi \star \g^y_Q) \Big( B_\ell(\s^y_{[0;N_d)}) \Big\backslash\bigcup_{s\leq m}B_\ell(\s^y_{\dom(\mathfrak{Q}_s)})\Big) > \frac{\eta}{8M_1}\Big\},\]
another appeal to Markov's Inequality implies that
\[|\cal{J}'| \geq \frac{\eta}{8M_1}|\cal{J}| \geq \frac{\eta(1 - \delta_1)}{8M_1}N_d/N_r,\]
and this is greater than $(1-\k_{r,d})N_d/N_r$ by our choice in (C3), because $r \geq r_{0,2}$. We may therefore apply Proposition~\ref{prop:gooddCs} again to obtain a discrete Cantor family $\mathfrak{Q}_{m+1} := (Q'_\w)_{\w \in \{0,1\}^{d-r}}$ in $\J'$ satisfying the same properties (i) and (ii) as in the base step.

Observe that
\[(\phi \star \g^y_{Q'_\w})(B_\ell(\s^y_{Q'_\w})) = 1 \quad \forall \w \in \{0,1\}^{d-r}.\]
Combined with the defining property of $\J'$, this gives
\[(\phi \star \g^y_{Q'_\w})\Big(B_\ell(\s^y_{Q'_\w})\Big\backslash \bigcup_{s \leq m}B_\ell(\s^y_{\dom(\mathfrak{Q}_s)})\Big) > \frac{\eta}{8M_1}.\]
Since
\[\mcup\J' \subseteq \mcup \J \subseteq H^{\rm{smooth}}_{r,[0;N_d),M_2,\eps_2}(y),\]
this now implies that
\begin{multline*}
\rm{U}_{B_\ell(\s^y_{Q'_\w})}\Big(B_\ell(\s^y_{Q'_\w}) \Big\backslash \bigcup_{s \leq m}B_\ell(\s^y_{\dom(\mathfrak{Q}_s)})\Big) > \frac{\frac{\eta}{8M_1} - \eps_2}{M_2} > \frac{\eta}{16M_1M_2} = \frac{1}{M}.\\
\Longrightarrow \quad \calL^1\Big(B_\ell(\s^y_{Q'_\w}) \Big\backslash \bigcup_{s \leq m}B_\ell(\s^y_{\dom(\mathfrak{Q}_s)})\Big) \geq \frac{1}{M}\calL^1(B_\ell(\s^y_{Q'_\w})).
\end{multline*}

Therefore
\begin{eqnarray}\label{eq:growth-in-meas}
&&\calL^1\Big(\bigcup_{s \leq m+1}B_\ell(\s^y_{\dom(\mathfrak{Q}_s)})\Big\backslash \bigcup_{s \leq m}B_\ell(\s^y_{\dom(\mathfrak{Q}_s)})\Big) \nonumber\\
&&= \sum_{\w \in \{0,1\}^{d-r}}\calL^1\Big(B_\ell(\s^y_{Q'_\w}) \Big\backslash \bigcup_{s \leq m}B_\ell(\s^y_{\dom(\mathfrak{Q}_s)})\Big)\nonumber\\
&&\geq \frac{1}{M}\sum_{\w \in \{0,1\}^{d-r}}\calL^1(B_\ell(\s^y_{Q'_\w})) 
= \frac{1}{M}\calL^1(B_\ell(\s^y_{\dom(\mathfrak{Q}_{m+1})})).
\end{eqnarray}

Now continue the recursion.

\vspace{7pt}

\emph{Step 4: Completion of the proof.}\quad Since all quantities are finite, the algorithm described above must end after finitely many steps in some choice of $\cal{G} = \{\mathfrak{Q}_1,\ldots,\mathfrak{Q}_m\}$.  When it does so:
\begin{itemize}
\item conclusion (1) follows because all the chosen intervals were contained in $H^{\rm{spread}}_{r,[0;N_d)}(y)$;
\item conclusion (2) follows from the use of Proposition~\ref{prop:gooddCs} at each step of the construction;
\item conclusion (3) holds because it was the condition for the algorithm to terminate;
\item conclusion (4) holds by an iterated appeal to inequality~(\ref{eq:growth-in-meas}):
\begin{eqnarray*}
&&\sum_{s=1}^m\calL^1(B_\ell(\s^y_{\dom(\mathfrak{Q}_s)}))\\
&&\leq M\sum_{s=1}^m\calL^1\Big(B_\ell(\s^y_{\dom(\mathfrak{Q}_s)})\Big\backslash \bigcup_{s' =1}^{s-1}B_\ell(\s^y_{\dom(\mathfrak{Q}_{s'})})\Big)\\
&&= M\calL^1\Big(\bigcup_{s \leq m}B_\ell(\s^y_{\dom(\mathfrak{Q}_s)})\Big) \leq M\calL^1\big(B_\ell(\s^y_{[0;N_d)})\big).
\end{eqnarray*}
\end{itemize}
$\phantom{i}$\end{proof}



\section{The lower bound, completion of the proof}\label{sec:lower-bd}

This section proves the lower bound in Theorem~\ref{thm:rate}, and hence completes the proof of that theorem.  It rests on a result showing that if $y$ and $y'$ satisfy some constraints in terms of $\s^y_{[0;N)}$ and $\s^{y'}_{[0;N)}$, and if $(y,x) \approx (y',x')$ according to $d^{\bfY \ltimes_\s \bfX}_{[0;N)}$, then $x$ and $x'$ must themselves be `similar', in a sense involving discrete Cantor sets.

This notion of similarity is quite cumbersome.  I suspect this is a defect of the proof method, in that a much tighter (and more easily-stated) notion of similarity for sceneries also holds in our setting, as discussed in Subsection~\ref{subs:approx-recov}, but I have not been able to prove this.

It should be stressed that the structural relation between sceneries that we deduce is already implicitly at the heart of Kalikow's argument in~\cite{Kal82}, as well as its various sequels~\cite{Rud88,denHSte97}.  However, our formulation is superficially quite different from Kalikow's.  He introduces certain hierarchically-defined events in the space of walk-scenery pairs, involving longer and longer time-scales, and then uses a recursion to show that the probabilities of all these events remain close to $1$.  We re-interpret each of these as an event involving a `tree-like' embedding of a discrete Cantor set into the domain of the scenery: this use of discrete Cantor sets converts Kalikow's hierarchy of properties into a single geometric structure.  In addition to making the relevant properties easier to visualize, this affords new ways of using Kalikow's estimates: we will ultimately need to work with whole families of these discrete Cantor sets, whereas Kalikow needs to apply his hierarchical probability-estimate only once.

In the following, $\bfY$, $\bfX$, $\s$, $\phi$ and $\ell$ will continue to be as in the previous section.  We also continue to assume that $\rm{diam}(X,d^X) \leq 1$.

\subsection{Reduction to estimates for single bi-neighbourhoods}

The key to the lower bound in Theorem~\ref{thm:rate} will be that, once $d$ is large enough and for a suitable choice of the subset $U \subseteq Y\times X$ appearing in Definition~\ref{dfn:big-bicov}, all bi-neighbourhoods in $(U,d^{\bfY \ltimes_\s \bfX}_{[-N_d,0)},d^{\bfY\ltimes_\s \bfX}_{[0;N_d)},\nu\otimes \mu)$ are `small' in a certain sense. This will then force one to use many of them to cover a positive proportion of $U$.  This `smallness' of the bi-neighbourhoods is in terms of the metric on the sceneries alone.  Let $\pi_X:Y\times X \to X$ denote the coordinate projection, and let $\pi_X^\ast d^X$ denote the pullback of $d^X$ to a pseudometric on $Y\times X$, and similarly for other metrics on $X$.

The following piece of notation will now be quite useful: given $(B,B') \in C_0(0,1]\times C_0(0,1]$, let
\[\rm{aspect}(B,B') := \min\Big\{\frac{\calL^1(B_{[0,1]}\cap B'_{[0,1]})}{\calL^1(B_{[0,1]})},\frac{\calL^1(B_{[0,1]}\cap B'_{[0,1]})}{\calL^1(B'_{[0,1]})}\Big\},\]
interpreting this as $0$ if either $B$ or $B'$ is constant.  For $(B,B') \sim \sfW_{[0,1]}^{\otimes 2}$, the random variable $\rm{aspect}(B,B')$ is a.s. positive with a continuous distribution on $(0,1]$.

\begin{prop}[Bounding the covering number of a bi-neighbourhood]\label{prop:cov-of-bi-neigh}
Let $\bfX$ and $\bfY$ be as before. For any $\g,\psi,\eps,\beta,\delta' > 0$ there exist $\eta,\delta > 0$ and a sequence of subsets $Y_d \subseteq Y$, $d \geq 1$, such that $\nu(Y_d) > 1 - \beta$ for all sufficiently large $d$, and such that the following holds.  If
\begin{quote}
$B,B' \in C_0(0,1]$ with $\rm{aspect}(B,B') > \g$ and $\calL^1(B_{[0,1]}\cap B'_{[0,1]}) \geq \psi/2$,
\end{quote}
and if we set $I := \sqrt{N_d}B_{[0,1]}$, $J := \sqrt{N_d}B'_{[0,1]}$, and
\[U := \big\{y\in Y_d\,\big|\ \|\traj_{-N_d}(\s^y) - B\|_\infty < \eta\ \hbox{and}\ \|\traj_{N_d}(\s^y) - B'\|_\infty < \eta\big\} \times X,\] 
then
\[\cov\Big(\Big(U\cap B^{d^{\bfY \ltimes_\s \bfX}_{[0;N_d)}}_{\delta N_d}\big(U\cap B^{d^{\bfY \ltimes_\s \bfX}_{[-N_d;0)}}_{\delta N_d}(y,x)\big),\pi_X^\ast d^{\bfX,\infty}_{I\cap J}\Big),\delta'\Big) \leq \exp(\eps \calL^1(I\cap J))\]
for any $(y,x) \in U$.
\end{prop}

Heuristically, this proposition asserts that, provided one looks only within the set $U$ defined by suitable approximate trajectories, both $d^{\bfY\ltimes_\s \bfX}_{[-N_d;0)}$ and $d^{\bfY \ltimes_\s \bfX}_{[0;N_d)}$ remember something about the scenery $x$ according to the metric $d^{\bfX,\infty}_{I \cap J}$.  On the other hand, we will be able to arrange that $U$ has measure bounded below by some $\k > 0$, and so it will follow that one needs roughly $\exp(\rmh(\bfX)|I\cap J \cap \bbZ|)$ of these bi-neighbourhoods to cover a sizable portion of $U$.

Proposition~\ref{prop:cov-of-bi-neigh} could be formulated using the Hamming-like metric $d^\bfX_{I\cap J}$ in place of the Bowen-Dinaburg metric $d^{\bfX,\infty}_{I\cap J}$, but the latter choice turns out to give a slightly shorter proof.

Most of this section will be given to the proof of Proposition~\ref{prop:cov-of-bi-neigh}.  However, let us first show how it implies Theorem~\ref{thm:rate}.  This will also use the following auxiliary lemma.

\begin{lem}\label{lem:transfer-psi}
For every $\a \in (1,\infty)$ and $\eps > 0$ there exists $\g > 0$ for which the following holds.  For any $\eta > 0$ there is a $\k > 0$ such that, for all sufficiently large $N \in \bbN$,
\begin{eqnarray*}
&&\forall \nu' \in \Pr Y\ \hbox{such that}\ \|\d\nu'/\d\nu\|_\infty \leq \a,\ \exists B,B' \in C_0(0,1]\ \hbox{such that}\\
&& \quad \quad \cal{L}^1(B_{[0,1]}\cap B_{[0,1]}') > \psi_{\rm{BM}}(\a) - \eps,\\
&& \quad \quad \rm{aspect}(B,B') > \g,\\
&& \hbox{and}\\
&& \quad \quad \nu'\big\{y\,\big|\ \|\traj_{-N}(\s^y) - B\|_\infty < \eta\ \hbox{and}\ \|\traj_N(\s^y) - B'\|_\infty < \eta\big\} \geq \k.
\end{eqnarray*}
\end{lem}

\begin{proof}
Let $\psi:= \psi_{\rm{BM}}(\a)$.  Definition~\ref{dfn:psi} gives
\[\sfW_{[0,1]}^{\otimes 2}\big\{(B,B')\,\big|\ \calL^1(B_{[0,1]}\cap B'_{[0,1]})  \leq \psi_{\rm{BM}}(\a) - \eps \big\} < 1/\a.\]
Since $\rm{aspect}(B,B') > 0$ for $\sfW^{\otimes 2}_{[0,1]}$-a.e. $(B,B')$, we may now choose $\g > 0$ such that the closed set
\[K := \big\{(B,B')\,\big|\ \calL^1(B_{[0,1]}\cap B'_{[0,1]}) \leq \psi_{\rm{BM}}(\a) - \eps\ \ \hbox{or}\ \ \rm{aspect}(B,B') \leq \g\big\}\]
still has $\sfW^{\otimes 2}_{[0,1]}(K) < 1/\a$.

Now suppose that $\eta > 0$. By the inner-regularity of $\sfW_{[0,1]}^{\otimes 2}$ with respect to compact sets, there are finitely many open subsets $W_1,\ldots,W_m \subseteq C_0(0,1]^2\setminus K$ such that
\begin{itemize}
\item $\sfW_{[0,1]}^{\otimes 2}(W_1\cup \cdots \cup W_m) > 1 - 1/\a + \k_0$ for some $\k_0 > 0$, and
\item each $W_i$ has diameter less than $\eta$ for the maximum of the metrics $\|\cdot\|_\infty$ on each coordinate in $C_0(0,1]^2$.
\end{itemize}
Set $\k := \a\k_0/m$, and choose representative pairs $(B_i,B_i') \in W_i$ for each $i\leq m$.

The Invariance Principle (Theorem~\ref{thm:IP}) implies that
\[(\traj_{-N}(\s^y),\traj_N(\s^y)) \stackrel{\rm{law}}{\to} (B,B') \sim \sfW_{[0,1]}^{\otimes 2} \quad \hbox{as}\ N \to \infty.\]
Therefore, since each $W_i$ is open, the Portmanteau Theorem~(\cite[Theorem 4.25]{Kal02}) gives
\[\nu\big\{(\traj_{-N}(\s^y),\traj_N(\s^y)) \in W_1 \cup \cdots \cup W_m\big\} > 1 - 1/\a + \k_0\]
for all sufficiently large $N$.

Suppose $N$ is large enough that this last inequality holds, and now consider some $\nu' \in \Pr Y$ with $\|\d\nu'/\d\nu\|_\infty \leq \a$.  Then that last inequality gives
\begin{multline*}
\nu'\big\{(\traj_{-N}(\s^y),\traj_N(\s^y)) \not\in W_1 \cup \cdots \cup W_m\big\}\\ \leq \a\nu\big\{(\traj_{-N}(\s^y),\traj_N(\s^y)) \not\in W_1 \cup \cdots \cup W_m\big\} < \a(1/\a - \k_0) = 1 - \a\k_0,
\end{multline*}
so there is some $i \leq m$ for which
\[\nu'\{(\traj_{-N}(\s^y),\traj_N(\s^y)) \in W_i\} > \k.\]
Letting $(B,B') := (B_i,B'_i)$ for this choice of $i$ completes the proof.
\end{proof}

\begin{proof}[Proof of Theorem~\ref{thm:rate} from Proposition~\ref{prop:cov-of-bi-neigh}]
The upper bound has already been obtained, so it remains to prove the lower bound.  This is vacuous if $\rmh(\bfY) = 0$, so assume $0 < \rmh(\bfY) < \infty$.  We now imagine playing as Max-er in the competition of Subsection~\ref{subs:competition}.

\vspace{7pt}

\emph{Step 1: The choice of measure and subset.}\quad First, Min-er chooses some $\l' \in \Pr(Y \times X)$ such that $\|\d\l'/\d(\nu\otimes \mu)\|_\infty \leq \a$.  Let $\nu' \in \Pr Y$ and $\mu' \in \Pr X$ be its marginals, so we know that also $\|\d\nu'/\d\nu\|_\infty \leq \a$.

Let $\psi := \psi_{\rm{BM}}(\a)$.  Fix $\eps > 0$, and assume without loss of generality that $\eps < \psi/4$.  For this $\eps$, we will show that there are choices of $\delta > 0$ and $\k > \k' > 0$, depending on $\eps$ but not on the particular measure $\l'$, such that for each sufficiently large $d$ there is some $U \subseteq Y\times X$ with $\l'(U) \geq \k$ and
\[\bicov_{\k'}\big((U,d^{\bfY \ltimes_\s \bfX}_{[-N_d;0)},d^{\bfY \ltimes_\s \bfX}_{[0;N_d)},\l'),\delta N_d\big) \geq \exp\big((\rmh(\bfX) - 3\eps)(\psi - 2\eps)\sqrt{N_d}\big).\]
Since $\eps$ is arbitrary this will complete the proof.

The parameters and subset are chosen in the following steps:
\begin{itemize}
\item Let $\delta' > 0$ be so small that $\rmh(\mu,T,d^X,\delta') > \rmh(\bfX) - \eps$, as is possible by Proposition~\ref{prop:spatial-ent-and-KS-ent}.

\item By the continuity of $\psi_{\rm{BM}}$ (Lemma~\ref{lem:props-of-psiBM}), choose $\a' > \a$ such that $\psi_{\rm{BM}}(\a') > \psi_{\rm{BM}}(\a) - \eps$.

\item Let $\g > 0$ be given by Lemma~\ref{lem:transfer-psi} for $\a'$ and $\eps$.

\item Choose $\b$ so small that $\frac{\a}{1 - \b\a} < \a'$.  Given $\g$, $\psi$, $\eps$, $\delta'$ and this $\b$, now apply Proposition~\ref{prop:cov-of-bi-neigh} to obtain $\eta,\delta > 0$ and the subsets $Y_d \subseteq Y$ having the properties listed there.  Observe that
\[\nu'(Y\setminus Y_d) \leq \a\nu(Y\setminus Y_d) < \a\b \quad \Longrightarrow \quad \l'(Y_d\times X) = \nu'(Y_d) \geq 1 - \a\b,\]
and hence
\[\Big\|\frac{\d\l'_{|Y_d\times X}}{\d(\nu\otimes \mu)}\Big\|_\infty \leq \frac{\a}{1 - \a\b} < \a'\]
for all sufficiently large $d$.

\item For our given $\eps$ and for the values of $\g$ and $\eta$ chosen above, and for $d$ sufficiently large, now return to Lemma~\ref{lem:transfer-psi}, applied to the measure $\nu'_{|Y_d}$, to obtain some $\k > 0$ such that, for all sufficiently large $d$, there are $B,B' \in C_0(0,1]$ satisfying
\[\calL^1(B_{[0,1]}\cap B'_{[0,1]}) > \psi_{\rm{BM}}(\a') - \eps > \psi - 2\eps > \psi/2,\]
\[\rm{aspect}(B,B') > \g,\]
and such that the set
\[U_1 := \big\{y\in Y_d\,\big|\ \|\traj_{-N_d}(\s^y) - B\|_\infty < \eta\ \hbox{and}\ \|\traj_{N_d}(\s^y) - B'\|_\infty < \eta\big\}\]
has
\[\nu'(U_1) = \nu'(Y_d)\nu'_{|Y_d}(U_1) \geq (1 - \a\b)\nu'_{|Y_d}(U_1) \geq \k.\]
Let $I:= \sqrt{N_d}B_{[0,1]}$ and $J := \sqrt{N_d}B'_{[0,1]}$.

\item Finally, let $\k' := \k/2$, and let $U := U_1 \times X$, so $\l'(U) = \nu'(U_1) \geq \k$ for all sufficiently large $d$.
\end{itemize}

\vspace{7pt}

\emph{Step 2: Bounding the bi-covering number.}\quad 
The conclusion of Proposition~\ref{prop:cov-of-bi-neigh} now gives that
\[\cov\Big(\Big(U\cap B^{d^{\bfY \ltimes_\s \bfX}_{[0;N_d)}}_{\delta N_d}\big(U\cap B^{d^{\bfY \ltimes_\s \bfX}_{[-N_d;0)}}_{\delta N_d}(y,x)\big),\ \pi_X^\ast d^{\bfX,\infty}_{I\cap J}\Big),\delta'\Big) \leq \exp(\eps \calL^1(I\cap J))\]
for all $(y,x) \in U$, for all sufficiently large $d$.  On the other hand, if $V \subseteq U$ with $\l'(V) \geq \k'$, then also $(\nu \otimes \mu)(V) \geq \k'/\a$ and hence $\mu(\pi_X(V)) \geq \k'/\a$, a fixed positive constant. Therefore, provided $d$ and hence $\calL^1(I\cap J) \geq (\psi/2)\sqrt{N_d}$ are sufficiently large, Proposition~\ref{prop:from-OW} and the choice of $\delta'$ give
\begin{eqnarray*}
\cov((V,\pi_X^\ast d^{\bfX,\infty}_{I\cap J}),\delta') &=& \cov((\pi_X(V),d^{\bfX,\infty}_{I\cap J}),\delta')\\
&\geq& \cov\big((\pi_X(V),d^\bfX_{I\cap J}),\delta'\calL^1(I\cap J)\big)\\
&>& \exp\big((\rmh(\bfX) - 2\eps)\calL^1(I\cap J)\big).
\end{eqnarray*}

These two bounds together imply that, if $d$ is sufficiently large, then the number of $(\delta N_d)$-bi-neighbourhoods needed to cover such a subset $V \subseteq U$ is at least
\[\frac{\exp\big((\rmh(\bfX) - 2\eps)\calL^1(I\cap J)\big)}{\exp(\eps \calL^1(I\cap J))} \geq \exp\big((\rmh(\bfX) - 3\eps)(\psi - 2\eps)\sqrt{N_d}\big),\]
as required.
\end{proof}

%
%

The rest of this section is occupied by the proof of Proposition~\ref{prop:cov-of-bi-neigh}.

\subsection{Discrete Cantor matchings}

First we need the following relative of Definition~\ref{dfn:dCs}.  Its importance will appear in the formulation of Proposition~\ref{prop:key}.

\begin{dfn}[Discrete Cantor matchings]\label{dfn:dCm}
Let $d \in \bbN$.  A \textbf{discrete Cantor matching} of \textbf{depth} $d$ is a pair $(K_\w,u_\w)_{\w \in \{0,1\}^d}$ in which $(K_\w)_\w$ is a discrete Cantor family and $(u_\w)_\w$ is a discrete Cantor set, both of depth $d$.  It has \textbf{gap upper bounds} $D_1 \geq \ldots \geq D_d$ if these are gap upper bounds for both this discrete Cantor family and this discrete Cantor set.

If $\frM = (K_\w,u_\w)_\w$ is a discrete Cantor matching, then its \textbf{domain} is
\[\dom(\frM) := \bigcup_\w K_\w.\]

For fixed $d$, $D = (D_1 \geq \cdots \geq D_d)$ and $J \in \rm{Int}(\bbR)$, the collection of depth-$d$ discrete Cantor matchings contained in $J$ and with gap upper bounds $D$ will be denoted by $\DCM_{d,D}(J)$.
\end{dfn}

Clearly $\DCM_{d,D}(J)$ may be identified with $\DCF_{d,D}(J)\times \DCS_{d,D}(J)$.  We endow $\DCM_{d,D}(J)$ with the metric $d_\DCM$ given by the maximum of the metrics $d_\DCF$ and $d_\DCS$ on these coordinate factors, and so Lemma~\ref{lem:bound-dCs} and Corollary~\ref{cor:bound-dCf} immediately imply
\begin{multline}\label{eq:DCM-cov}
\cov\big((\DCM_{d,D}(J),d_\DCM),\delta\big)\\ \leq \Big(\frac{2\calL^1(J)}{\delta}\Big(\frac{2D_1}{\delta}\Big)\Big(\frac{2D_2}{\delta}\Big)^2\cdots \Big(\frac{2D_d}{\delta}\Big)^{2^{d-1}}\Big)^3
\end{multline}
provided $\delta < \calL^1(J)/10,D_d/10$.

\subsection{Similarity of sceneries from similarity of pairs}\label{subs:sim-to-sim}

Let the parameters $L_d$, $N_d$, $\a_d$ and $\k_{r,d}$ be as introduced in Subsection~\ref{subs:disc-filt}.

\begin{prop}\label{prop:key}
Suppose that $\beta,\eta > 0$, and let $M < \infty$, $r_0 \in \bbN$, and the sets $Y^{\rm{good}}_{r,d}$ for $d > r \geq r_0$ be as provided by Proposition~\ref{prop:pre-key} for this $\beta$ and $\eta$.

Then for any $r \geq r_0$ there is a $\delta > 0$ such that if
\begin{itemize}
\item $B \in C_0(0,1]$ and $J := \sqrt{N_d} B_{[0,1]} + [-\eta\sqrt{N_d},\eta\sqrt{N_d}]$,
\item $(y,x),(y',x')\in Y^{\rm{good}}_{r,d}\times X$ with
\[\|\traj_{N_d}(\s^y) - B\|_\infty,\ \|\traj_{N_d}(\s^{y'}) - B\|_\infty < \eta,\]
\item and
\begin{eqnarray}\label{eq:d-big-small}
d_{[0;N_d)}^{\bfY \ltimes_\s \bfX}\big((y,x),(y',x')\big) \leq \delta N_d,
\end{eqnarray}
\end{itemize}
then there is a tuple
\[\F \in \DCM_{d-r,4\ell(N_d \geq \ldots \geq N_{r+1})}(J)^m \quad \hbox{for some} \quad m \leq \frac{M\calL^1(J)}{2^{d-r}N_r^{1/3}}\]
such that the following hold:
\begin{enumerate}
\item[P1)] (discrete Cantor sets are small) for every $(K_\omega,u_\w)_\w \in \F$ and every $\w \in \{0,1\}^{d-r}$, one has $|u_\omega| < 2\eta\sqrt{N_d}$;
\item[P2)] (range interval is mostly covered) one has
\[\cal{L}^1\Big(J \Big\backslash\bigcup_{\frM \in \F}\dom(\frM)\Big) < \eta \calL^1(J) + 4\eta\sqrt{N_d};\]
\item[P3)] (sceneries approximately agree across matchings) for every $(K_\w,u_\w)_\w \in \F$ and $\w \in \{0,1\}^{d-r}$, one has
\[d^X(T^zx,T^{z+u_\w}x') < \eta \quad \forall z \in K_\w.\]
\end{enumerate}
\end{prop}

\begin{rmk}
Beware that the assumptions of this proposition are symmetric between $(y,x)$ and $(y',x')$, but the conclusions are not. \fin
\end{rmk}

\begin{proof}\quad \emph{Step 1: Choice of parameters.}\quad  Suppose that $r \geq r_0$, and let $\delta_1 > 0$ be the error tolerance given by Proposition~\ref{prop:pre-key} for this $r$.

Since the action $T:\bbR\actson X$ is jointly continuous, given our $\eta > 0$ and $\ell$, there is some $\t{\eta} > 0$ such that for any $x,x' \in X$ one has
\[d^X(x,x') < 2\t{\eta} \quad \Longrightarrow \quad d^X(T^zx,T^zx') < \eta \quad \forall z \in [-\ell,\ell].\]
Again by the continuity of this action, we may now choose $\eps > 0$ such that
\[|z| \leq \eps \quad \Longrightarrow \quad d^X(x,T^zx) < \t{\eta} \quad \forall x\in X.\]

Next, given this $\eps$, Corollary~\ref{cor:P-controls-sigma} gives some $\t{\delta} > 0$ such that for any discrete interval $[0;L) \subseteq \bbZ$ and any $x,x' \in Y$ one has
\[\max_{n \in [0;L)}d^Y(S^ny,S^ny') < \t{\delta} \quad \Longrightarrow \quad \max_{n \in [0;L)}|\s^y_n - \s^{y'}_n| < \eps.\]

Finally, let $\delta := \delta_1\cdot\min\{\t{\delta},\t{\eta}\}$, and assume~(\ref{eq:d-big-small}) with this value of $\delta$.

\vspace{7pt}

\emph{Step 2: Using Proposition~\ref{prop:pre-key}.}\quad Consider the set
\[P := \big\{n \in [0;N_d)\,\big|\ d^Y(S^ny,S^ny') \leq \t{\delta}\ 
\hbox{and}\ d^X(T^{\s^y_n}x,T^{\s^{y'}_n}x')\leq \t{\eta}\big\}.\]
By~(\ref{eq:d-big-small}), our choice of $\delta$ and Markov's Inequality, one has
\[\delta N_d \geq |[0;N_d)\setminus P|\cdot \min\{\t{\delta},\t{\eta}\} \quad \Longrightarrow \quad |P| \geq (1 - \delta_1)N_d.\]
We may therefore subject $P$ to an application of Proposition~\ref{prop:pre-key}.  Let $\cal{G}$ be the collection of discrete Cantor families produced by that proposition, ordered so that $\cal{G} = ((Q_{s,\w})_{\w})_{s=1}^m$ for some $m$, and let
\[\cal{F} := \big((B_\ell(\s^y_{Q_{s,\w}}),\s^{y'}_{\min Q_{s,\w}} - \s^y_{\min Q_{s,\w}})_{\w \in \{0,1\}^{d-r}}\big)_{s=1}^m.\]
To see that each entry of $\F$ is a member of $\DCM_{d-r,4\ell(N_d\geq \ldots \geq N_{r+1})}(J)$, observe the following:
\begin{itemize}
\item property (2) of Proposition~\ref{prop:pre-key} gave that the image-neighbourhoods $B_\ell(\s^y_{Q_{s,\w}})$ are pairwise disjoint for distinct $\w \in \{0,1\}^{d-r}$, for each fixed $s$;
\item these neighbourhoods are all contained in $B_\ell(\s^y_{[0;N_d)}) \subseteq J$;
\item and, since $\ell \geq \|\s\|_\infty$, for any bounded discrete interval $R \subseteq \bbZ$ and any $y \in Y$, the image $\s^y_R$ must have diameter at most $\ell \cal{L}^1(R)$. It follows that $B_\ell(\s^y_R)$ is an interval of length at most $\ell\calL^1(R) + 2\ell$.  Therefore, for each $s\leq m$, the gap upper bounds $\ell (N_d + 2)\geq \ldots \geq \ell (N_{r+1} + 2)$ hold for the discrete Cantor family $(B_\ell(\s^y_{Q_{s,\w}}))_\w$ because $(Q_\w)_\w$ was adapted to $(\cal{D}_d,\ldots,\cal{D}_{r+1})$, and the gap upper bounds $2\ell (N_d + 2)\geq \ldots \geq 2\ell (N_{r+1} + 2)$ hold for the discrete Cantor set $(\s^{y'}_{\min Q_{s,\w}} - \s^y_{\min Q_{s,\w}})_\w$ as it is a set of differences of elements of such discrete Cantor families. This conclusion is stronger than the gap upper bounds $4\ell N_d \geq \ldots \geq 4\ell N_{r+1}$ of property (P3), because $N_{r+1} \geq N_{r_0 + 1} \geq 2$.
\end{itemize}

The desired upper bound on $m$ holds because property (1) of Proposition~\ref{prop:pre-key} gives
\begin{multline*}
\calL^1(K_\w) \geq N_r^{1/3} \quad \forall (K_\w,u_\w)_\w \in \F\ \hbox{and}\ \w \in \{0,1\}^{d-r}\\
\Longrightarrow \quad \calL^1(\dom(\frM)) \geq 2^{d-r}N_r^{1/3} \quad \forall \frM \in \F,
\end{multline*}
while property (4) of that proposition gives
\[\sum_{\dom(\frM) \in \F}\calL^1(\dom(\frM)) \leq M\calL^1\big(B_\ell(\s^y_{[0;N_d)})\big) \leq M\calL^1(J).\]

\vspace{7pt}

\emph{Step 3: Verifying the remaining properties.}\quad It remains to prove (P1)--(P3).

Property (P1) holds because $\traj_N(\s^y)$ and $\traj_N(\s^{y'})$ are both close to $B$: for $t := \min Q_\omega/N_d$, those approximations give
\[|\s^y_{\min Q_\omega} - \s^{y'}_{\min Q_\omega}| = \sqrt{N_d}|\traj_{N_d}(\s^y)(t) - \traj_{N_d}(\s^{y'})(t)| < 2\eta\sqrt{N_d}.\]

Property (P2) results from property (3) of Proposition~\ref{prop:pre-key}, combined with the facts that
\[B_\ell(\s^y_{[0;N_d)}) \subseteq J \quad \hbox{and} \quad \calL^1\big(J\setminus B_\ell(\s^y_{[0;N_d)})\big) < 4\eta \sqrt{N_d}.\]

Finally, property (P3) holds because for each $s\leq m$ and each $\w \in \{0,1\}^{d-r}$, we have that $Q_{s,\w} \subseteq P$ by construction, and so the definition of $P$ gives
\begin{eqnarray}\label{eq:two-ineqs}
d^Y(S^ny,S^ny') < \t{\delta} \quad \hbox{and} \quad d^X(T^{\s^y_n}x,T^{\s^{y'}_n}x') < \t{\eta} \quad \forall n \in Q_{s,\w}.
\end{eqnarray}
For this $s$ and each $\w$, now abbreviate $n_\w := \min Q_{s,\w}$ and $u_\w := \s^{y'}_{n_\w} - \s^y_{n_\w}$. By the choice of $\t{\delta}$, the first inequality in~(\ref{eq:two-ineqs}) implies that
\[|\s_n^{y'} - \s_n^y - u_\w| = |(\s_n^{y'} - \s^{y'}_{n_\w}) - (\s_n^y - \s^y_{n_\w})| = |\s^{S^{n_\w}y'}_{n - n_\w} - \s^{S^{n_\w}y}_{n - n_\w}|< \eps \quad \forall n \in Q_{s,\w}.\]
Given this, and letting $y_1 := S^{n_\w}y$, $y_1' := S^{n_\w}y'$, $x_1 := T^{\s^y_{n_\w}}x$ and $x_1' := T^{\s^{y'}_{n_\w}}x'$, the second inequality in~(\ref{eq:two-ineqs}) may be re-written as
\[d^X\big(T^{\s^{y_1}_n}x_1,\,T^{\s^{y_1'}_n}x_1'\big) < \t{\eta} \quad \forall n \in [0;N_r).\]
Combining the above inequalities, and recalling the choice of $\eps$, we now obtain
\[d^X\big(T^{\s^{y_1}_n}x_1,\,T^{\s^{y_1}_n}x_1'\big) < 2\t{\eta} \quad \forall n \in [0;N_r),\]
and now by the choice of $\t{\eta}$ this implies that
\[d^X(T^zx,T^{z+ u_\w}x') < \eta \quad \forall z \in B_\ell(\s^{y_1}_{[0;N_r)}) = B_\ell(\s^y_{Q_{s,\w}}).\]
\end{proof}

We will retain the names (P1)--(P3) for the above properties throughout the rest of the paper.  Note that, by duplicating some members of the resulting family $\F$, we may always assume that $m = \lfloor M\calL^1(J)/2^{d-r}N_r^{1/3}\rfloor$ without disrupting these other properties.


\subsection{Bounding the covering number of discrete Cantor matchings}

The next lemma is an elementary estimate which will lie at the heart of the competition between two different sources of entropy in the sequel.

\begin{lem}\label{lem:baby}
There is some absolute constant $C_0 < \infty$ such that
\[\sum_{s=0}^{d-r} 2^s (d-s)\log (d-s) \leq C_0 2^d \quad \hbox{whenever}\ d > r\geq 1.\]
\end{lem}

\begin{proof}
Dividing the left-hand side by $2^d$ produces the sum
\[\sum_{s=0}^{d-r} 2^{-(d-s)} (d-s)\log (d-s) = \sum_{\ell=r}^d 2^{-\ell}\ell\log \ell \lesssim \sum_{\ell=r}^\infty 2^{-\ell}\cdot 2^{\ell/2} < \infty.\]
\end{proof}

For any fixed $d$, $D = (D_1 \geq \ldots \geq D_d)$ and $J \in \rm{Int}(\bbR)$, and for each $m \in \bbN$, let $d_{\DCM,m}$ be the metric on $(\DCM_{d,D}(J))^m$ given as the maximum of the metric $d_\DCM$ on each of the $m$ coordinates.

\begin{lem}[Bounding the number of discrete-Cantor-matching tuples]\label{lem:no-of-matchings}
For every $\delta,\eps > 0$ there exists $r_1 \in \bbN$ such that if $d > r \geq r_1$, if $J \in \rm{Int}(\bbR)$ has length at most $4\ell N_d$, and if
\[m := \Big\lfloor \frac{M\calL^1(J)}{2^{d-r}N_r^{1/3}}\Big\rfloor,\]
then
\[\cov\big(\big((\DCM_{d-r,4\ell(N_d\geq \ldots \geq N_{r+1})}(J))^m,d_{\DCM,m}\big),\delta \big) \leq \exp (\eps \calL^1(J)) \quad \forall m \in \bbN.\]
\end{lem}

\begin{proof}
By~(\ref{eq:DCM-cov}) and~(\ref{eq:N_s-est}), there is an absolute constant $C < \infty$ such that this covering number is bounded by
\begin{eqnarray*}
&&\Big(\frac{8\ell N_d}{\delta}\Big(\frac{8\ell N_d}{\delta}\Big)\Big(\frac{8\ell N_{d-1}}{\delta}\Big)^2\cdots \Big(\frac{8\ell N_{r+1}}{\delta}\Big)^{2^{d-r-1}}\Big)^{3m}\\
&&= \Big(\frac{8\ell}{\delta}\Big)^{3m2^{d-r}} \big(N_d\cdot N_d \cdot N_{d-1}^2 \cdots N_{r+1}^{2^{d-r-1}}\big)^{3m}\\
&&\leq \exp\Big(Cm\big((d+1)\log (d+1) + 2d\log d + \cdots + 2^{d-r-1}(r+2)\log (r+2)\big)\\
&& \quad \quad \quad \quad \quad \quad \quad \quad \quad \quad \quad \quad \quad \quad \quad \quad \quad \quad \quad \quad \quad \quad \quad \quad + 3m2^{d-r}\log \frac{8\ell}{\delta}\Big)\\
&&= \exp\Big(Cm\sum_{s=0}^{d-r-1} 2^s(d+1-s)\log (d+1-s) + 3m2^{d-r}\log \frac{8\ell}{\delta}\Big).
\end{eqnarray*}

Substituting for $m$, this is bounded by
\[\exp\Big(\frac{C'\calL^1(J)}{2^{d-r}N_r^{1/3}}\sum_{s=0}^{d-r-1} 2^s(d+1-s)\log (d+1-s) + C''\frac{\calL^1(J)}{N_r^{1/3}}\Big),\]
where
\[C' := C M \quad \hbox{and} \quad C'' := 3 M\log\frac{8\ell}{\delta},\]
neither of which depends on $d$ or $r$.  Letting $C_0$ be the constant from Lemma~\ref{lem:baby}, the above expression is in turn bounded by
\[\exp\Big(\frac{2C'C_0\calL^1(J)}{N_r^{1/3}2^{-r}} + C''\frac{\calL^1(J)}{N_r^{1/3}}\Big).\]
Another appeal to~(\ref{eq:N_s-est}) implies that $N_r^{1/3}2^{-r} \to \infty$ as $r \to\infty$, so the above is bounded by $\exp(\eps \calL^1(J))$ provided $r$ was large enough, irrespective of the value of $d$.
\end{proof}

\subsection{Bounding covering numbers of bi-neighbourhoods}

\begin{lem}\label{lem:close-match-close-scenery}
For any $\g',\zeta > 0$ there are $\eta,\delta''' \in (0,1)$ for which the following holds.  Suppose that $K \subseteq J \subseteq \bbR$ are compact intervals, both containing $0$, such that $\calL^1(K) \geq \g'\calL^1(J)$.  Suppose also that $d > r \geq 1$, that $x_1,x_1',x_2,x_2' \in X$, and that
\[m \in \bbN \quad \hbox{and} \quad \F,\cal{G} \in \big(\DCM_{d-r,4\ell(N_d\geq \ldots \geq N_{r+1})}(J)\big)^m\]
are such that
\[J,x_1,x_1'\ \hbox{and}\ \F\ \hbox{satisfy (P1)--(P3) for these values of}\ r,d,\eta,\]
\[J,x_2,x_2'\ \hbox{and}\ \cal{G}\ \hbox{satisfy (P1)--(P3) for these values of}\ r,d,\eta,\]
\begin{eqnarray}\label{eq:M-N-close}
d_{\DCM,m}(\F,\cal{G}) < \delta''',
\end{eqnarray}
and
\begin{eqnarray}\label{eq:K'inf-bound}
d^{\bfX,\infty}_K(x'_1,x'_2) < \delta'''.
\end{eqnarray}
Then
\[d^\bfX_K(x_1,x_2) < \zeta(\calL^1(K) + \sqrt{N_d}).\]
(The choice of the notation `$\g'$' and `$\delta'''$' is for ease of reference later.)
\end{lem}

\begin{rmks}\emph{(1.)}\quad 
It is very important here that we assume only proximity of $x_1$ and $x_2$ in $d^{\bfX,\infty}_K$, rather than $d^\bfX_J$, and then (of course) also conclude only that kind of proximity.  On the other hand, it is also important that the input is an inequality for $d^{\bfX,\infty}_K$, whereas the output is only for $d^\bfX_K$; this difference will be taken into account later by an appeal to Lemma~\ref{lem:d1-cov-and-dinf-cov}.

\vspace{7pt}

\emph{(2.)}\quad Note also that this lemma does not use the bound on the length $m$ of the tuples $\F$ and $\cal{G}$. \fin
\end{rmks}

\begin{proof}
First, using the joint continuity of $T$, choose $\delta_0 > 0$ so small that
\[|z| \leq 2\delta_0 \quad \Longrightarrow \quad \max_{x \in X}d^X(x,T^zx) < \zeta/4.\]
Next, choose $\eta > 0$ and $\delta''' \in (0,\delta_0]$ both so small that
\[\eta/\g' + 2\eta + \delta''' < \zeta/2 \quad \hbox{and} \quad 8\eta < \zeta.\]

Now assume also that $d > r \geq 1$, and that $x_1$, $x_1'$, $x_2$, $x_2'$, $\F$ and $\cal{G}$ satisfy the stated assumptions.  For each $s \leq m$, let
\[\F = \big((K_{1,s,\omega},u_{1,s,\w})_{\w \in \{0,1\}^{d-r}}\big)_{s=1}^m \quad \hbox{and} \quad \cal{G} = \big((K_{2,s,\omega},u_{2,s,\w})_{\w \in \{0,1\}^{d-r}}\big)_{s=1}^m.\]
Let
\[D := \bigcup_{s=1}^m\bigcup_{\w \in \{0,1\}^{d-r}} K_{1,s,\w},\]
and let $K' \subseteq K$ be the closed subinterval with the same centre and length $\calL^1(K) - 4\eta\sqrt{N_d}$ (understood as $\emptyset$ if this value is negative).  By property (P2),
\begin{multline*}
\calL^1(K\setminus (K'\cap D)) \leq \calL^1(K\setminus K') + \calL^1(J\setminus D)\\ \leq 4\eta\sqrt{N_d} + (\eta\calL^1(J) + 4\eta\sqrt{N_d}) \leq 8\eta\sqrt{N_d} + (\eta/\g')\calL^1(K).
\end{multline*}

Now suppose that $t \in K'\cap D$.  There are $s \leq m$ and $\omega \in \{0,1\}^{d-r}$ such that $t \in K_{1,s,\omega}$, and now the approximation~(\ref{eq:M-N-close}) gives some $w \in (-\delta''',\delta''')$ such that $t + w \in K_{2,s,\omega}$.  Let $u_i := u_{i,s,\w}$ for $i=1,2$.

By the triangle inequality,
\begin{eqnarray*}
d^X(T^tx_1,T^tx_2) &\leq& d^X(T^tx_1,T^{t+u_1}x'_1) + d^X(T^{t+u_1}x'_1,T^{t+u_1}x'_2)\\
&& + d^X(T^{t+u_1}x'_2,T^{t+w+u_2}x'_2) + d^X(T^{t+w+u_2}x'_2,T^{t+w}x_2)\\
&& + d^X(T^{t+w}x_2,T^tx_2).
\end{eqnarray*}
These five right-hand terms may now be bounded separately:
\begin{itemize}
\item property (P3) gives
\[d^X(T^tx_1,T^{t+u_1}x_1') < \eta \quad \hbox{and} \quad  d^X(T^{t+w+u_2}x'_2,T^{t+w}x_2) < \eta;\]
\item since $t \in K''$ and property (P1) gives $|u_1| < 2\eta\sqrt{N_d}$, we still have $t + u_1 \in K'$, and so the approximation~(\ref{eq:K'inf-bound}) gives
\[d^X(T^{t+u_1}x'_1,T^{t+u_1}x'_2) < \delta''';\]
\item finally, our assumptions gave $|w| \leq \delta'''$ and the approximation~(\ref{eq:M-N-close}) gives $|u_1 - u_2| < \delta$, so the choice of $\delta''' \leq \delta_0$ implies that
\[d^X(T^{t+u_1}x'_2,T^{t+w+u_2}x'_2) < \zeta/4 \quad \hbox{and} \quad d^X(T^{t+w}x_2,T^tx_2) < \zeta/4.\]
\end{itemize}

Putting these estimates together gives
\[d^X(T^tx_1,T^tx_2) < 2\eta + \delta + \zeta/2 \quad \forall t \in K'\cap D.\]
Integrating over $t \in K'$, this becomes
\begin{eqnarray*}
d^\bfX_{K'}(x_1,x_2) &\leq& \int_{K\setminus (K'\cap D)}d^X(T^tx_1,T^tx_2)\,\d t + \int_{K'\cap D}d^X(T^tx_1,T^tx_2)\,\d t\\
&\leq& \calL^1(K\setminus (K'\cap D)) + (2\eta + \delta''' + \zeta/2)\calL^1(K'\cap D)\\
&\leq& (\eta/\g')\calL^1(K) + 8\eta\sqrt{N_d} + (2\eta + \delta''' + \zeta/2)\calL^1(K)\\
&=& (\eta/\g' + 2\eta + \delta''' + \zeta/2)\calL^1(K) + 8\eta\sqrt{N_d},
\end{eqnarray*}
and this is less than $\zeta\calL^1(K) + \zeta\sqrt{N_d}$ by the choice of $\eta$ and $\delta'''$.
\end{proof}

\begin{lem}\label{lem:pre-ofpropkey}
For any $\g,\psi,\eps,\beta,\delta' > 0$ there are $\eta,\delta > 0$ and a sequence of subsets $Y_d \subseteq Y$, $d \geq 1$, such that $\nu(Y_d) > 1 - \beta$ for all sufficiently large $d$, and such that the following holds.  If
\begin{quote}
$B \in C_0(0,1]$, $J := \sqrt{N_d}B_{[0,1]}$ and $K \in \rm{Int}(J)$ satisfy both $\calL^1(K) \geq (\psi/2)\sqrt{N_d}$ and $\calL^1(K) \geq \g\calL^1(J)$,
\end{quote}
and if
\[U \subseteq \big\{y \in Y_d\,\big|\ \|\traj_{N_d}(\s^y) - B\|_\infty < \eta\big\} \times X,\]
then
\[\rm{cov}\big(\big(U\cap B^{d^{\bfY \ltimes_\s \bfX}_{[0;N_d)}}_{\delta N_d}(C),\,\pi_X^\ast d^{\bfX,\infty}_K\big),\delta'\big) \leq \exp(\eps \calL^1(K))\]
whenever $C\subseteq U$ has diameter at most $\delta$ according to the pseudometric $\pi_X^\ast d^{\bfX,\infty}_K$.

The analogous result holds when $\traj_{N_d}$ is replaced by $\traj_{-N_d}$ and $d^{\bfY \ltimes_\s \bfX}_{[0;N_d)}$ is replaced by $d^{\bfY \ltimes_\s \bfX}_{[-N_d;0)}$.
\end{lem}

\begin{proof}
It will be clear that the second assertion follows in the same way as the first, so we concentrate on that.

\vspace{7pt}

\emph{Step 1: Choosing the parameters.}\quad
\begin{itemize}
\item First recall from Lemma~\ref{lem:d1-cov-and-dinf-cov} that for our given $\delta' > 0$, there is some $\delta'' > 0$ such that for every $x \in X$ and $K \in \rm{Int}(\bbR)$ with $\calL^1(K) \geq 1$ one has
\[\cov\big((B^{d^\bfX_K}_{\delta''\calL^1(K)}(x),d^{\bfX,\infty}_K),\delta'\big) < \exp(\eps \calL^1(K)/2).\]

\item Next, choose $\g' := \g/(1 + 4/\psi)$, choose $\zeta$ so small that $\zeta(1 + 2/\psi) <  \delta''$, and now implement Lemma~\ref{lem:close-match-close-scenery} with this $\g'$ and $\zeta$ to obtain some $\eta,\delta''' \in (0,1)$ with the property described there.

\item For the given value of $\beta$ and for the $\eta$ chosen above, now let $M < \infty$, $r_0 \in \bbN$ and the subsets $Y^{\rm{good}}_{r,d} \subseteq Y$ for $d > r \geq r_0$ be as provided by Proposition~\ref{prop:pre-key}.

\item Now let
\[\eps' := \eps/(2/\g + 8\eta/\psi),\]
and apply Lemma~\ref{lem:no-of-matchings} to obtain some $r \geq r_0$ such that for any $d > r$, if
\[I \in \rm{Int}(\bbR), \quad \calL^1(I) \leq 4\ell N_d \quad \hbox{and} \quad m := \Big\lfloor \frac{M\calL^1(I)}{2^{d-r}N_r^{1/3}} \Big\rfloor,\]
then
\[\cov\big(\big((\DCM_{d-r,4\ell(N_d\geq \ldots \geq N_{r+1})}(I))^m,d_{\DCM,m}\big),\delta'''/2\big)\big) \leq \exp(\eps'\calL^1(I)).\]
Fix this $r$, let $Y_d := Y^{\rm{good}}_{r,d}$ for all $d > r$, and, for completeness, let $Y_d := \emptyset $ for $d \leq r$.  The conclusion of Proposition~\ref{prop:pre-key} gives $\nu(Y_d) > 1 - \beta$ for all sufficiently large $d$.

\item Finally, having found this $r$, in addition to the other parameters chosen above, let $\delta > 0$ be given by Proposition~\ref{prop:key}.
\end{itemize}

Now assume that $d > r$ is sufficiently large and that $B$, $K$ and $J$ are as in the statement of the lemma.  Observe that
\[J' := J + [-\eta\sqrt{N_d},\eta\sqrt{N_d}] \subseteq B_\ell(\s^y_{[0;N_d)}) + [-2\eta\sqrt{N_d},2\eta\sqrt{N_d}],\]
so one also has $\calL^1(J') \leq 4\ell N_d$ once $d$ is sufficiently large. Let
\[m := \Big\lfloor \frac{M\calL^1(J')}{2^{d-r}N_r^{1/3}} \Big\rfloor.\]

\vspace{7pt}

\emph{Step 2: The Hamming-like metric.}\quad The next step is to prove an analog of the desired bound with $\pi_X^\ast d^\bfX_K$ in place of $\pi_X^\ast d^{\bfX,\infty}_K$ and with diameter $\delta''$ in place of radius $\delta'$.

Given our assumption on $U$ and choice of $\delta$, Proposition~\ref{prop:key} asserts that
\begin{multline*}
U\cap B^{d^{\bfY \ltimes_\s \bfX}_{[0;N_d)}}_{\delta N_d}(C)  \subseteq \big\{(y,x)\,\big|\ \exists (y',x') \in C\ \hbox{and}\ \F \ \hbox{such that}\ |\F| = m\\ \hbox{and}\ J',x,x',\F\ \hbox{satisfy (P1)--(P3)}\big\}.
\end{multline*}
Next, for any $d > r$, Lemma~\ref{lem:no-of-matchings} gives a Borel partition $\Q$ of $\DCM_{d-r,4\ell(N_d\geq \ldots \geq N_{r+1})}(J')^m$ into cells of diameter at most $\delta'''$ according to $d_{\DCM,m}$, and with
\begin{multline*}
|\Q| \leq \exp(\eps' \calL^1(J')) \leq \exp(\eps'(\calL^1(J) + 4\eta\sqrt{N_d}))\\ \leq \exp
(\eps'(1/\g + 4\eta/\psi)\calL^1(K))\leq \exp(\eps\calL^1(K)/2).
\end{multline*}
The above containment may now be written
\[U\cap B^{d^{\bfX \ltimes_\s \bfX}_{[0;N_d)}}_{\delta N_d}(C)\subseteq \bigcup_{Q \in \Q}R_Q\]
with
\[R_Q := \big\{(y,x)\,\big|\ \exists (y',x') \in C\ \hbox{and}\ \F \in Q \ \hbox{such that}\ J',x,x',\F\ \hbox{satisfy (P1)--(P3)}\big\}.\]
Now observe also that since $\calL^1(J) \geq \calL^1(K) \geq (2/\psi)\sqrt{N_d}$, we have
\[\calL^1(K) \geq \g\calL^1(J) \geq \g\frac{\calL^1(J')}{1 + 4\eta/\psi} \geq \g\frac{\calL^1(J')}{1 + 4/\psi} \geq \g'\calL^1(J').\]
Therefore, the choice of $\eta$ and $\delta'''$ using Lemma~\ref{lem:close-match-close-scenery} implies that
\[\rm{diam}(R_Q,\pi_X^\ast d^\bfX_K) < \zeta \calL^1(K) + \zeta\sqrt{N_d} \leq \zeta(1 + 2/\psi)\calL^1(K) < \delta''\calL^1(K) \quad \forall Q \in \Q.\]

\vspace{7pt}

\emph{Step 3: The Bowen-Dinaburg metric}\quad It remains to improve our conclusion from $\pi_X^\ast d^\bfX_K$ to $\pi_X^\ast d^{\bfX,\infty}_K$.  This follows because, by Lemma~\ref{lem:d1-cov-and-dinf-cov} and our choice of $\delta''$, each of the sets $R_Q$ obtained above may in turn be covered by at most $\exp(\eps \calL^1(K)/2)$ balls of radius $\delta'$ for the pseudometric $\pi_X^\ast d^{\bfX,\infty}_K$.
\end{proof}

\begin{proof}[Proof of Proposition~\ref{prop:cov-of-bi-neigh}]
This follows from two back-to-back appeals to Lemma~\ref{lem:pre-ofpropkey}, with some care over the values of all the error tolerances.

\vspace{7pt}

\emph{Step 1: Choosing the parameters.}\quad We are given $\g,\psi,\eps,\b,\delta' > 0$.

By the first part of Lemma~\ref{lem:pre-ofpropkey}, we may choose some $\eta_1,\delta_1 > 0$ and subsets $Y_{d,1} \subseteq Y$ such that $\nu(Y_{d,1}) > 1 - \beta/2$ for all sufficiently large $d$, and such that the following holds.  If $B' \in C_0(0,1]$, $J := \sqrt{N_d}B'_{[0,1]}$, and $K \in \rm{Int}(J)$ with both
\[\calL^1(K) \geq (\psi/2)\sqrt{N_d} \quad \hbox{and} \quad \calL^1(K) \geq \g\calL^1(J),\]
and if
\[U \subseteq \big\{y \in Y_{d,1}\,\big|\ \|\traj_{N_d}(\s^y) - B'\|_\infty < \eta_1\big\} \times X,\]
then
\[\rm{cov}\big(\big(U\cap B^{d^{\bfY \ltimes_\s \bfX}_{[0;N_d)}}_{\delta_1 N_d}(C),\,\pi_X^\ast d^{\bfX,\infty}_K\big),\delta'\big) \leq \exp(\eps \calL^1(K)/2)\]
whenever $C\subseteq U$ has diameter at most $\delta_1$ according to the pseudometric $\pi_X^\ast d^{\bfX,\infty}_K$.

Having done so, now the second part of Lemma~\ref{lem:pre-ofpropkey} gives some $\eta \in (0,\eta_1],\delta \in (0,\delta_1]$ and subsets $Y_d \subseteq Y_{d,1}$ such that $\nu(Y_d) > 1 - \beta$ for all sufficiently large $d$, and such that the following holds. If $B \in C_0(0,1]$, $I := \sqrt{N_d}B_{[0,1]}$, and $K \in \rm{Int}(I)$ with both
\[\calL^1(K) \geq (\psi/2)\sqrt{N_d} \quad \hbox{and} \quad \calL^1(K) \geq \g\calL^1(I),\]
and if
\[U' \subseteq \big\{y \in Y_d\,\big|\ \|\traj_{-N_d}(\s^y) - B\|_\infty < \eta\big\} \times X,\]
then
\[\rm{cov}\big(\big(U'\cap B^{d^{\bfY \ltimes_\s \bfX}_{[-N_d;0)}}_{\delta N_d}(C),\,\pi_X^\ast d^{\bfX,\infty}_K\big),\delta_1/2\big) \leq \exp(\eps \calL^1(K)/2)\]
whenever $C\subseteq U$ has diameter at most $\delta$ according to the pseudometric $\pi_X^\ast d^{\bfX,\infty}_K$.

This gives our choice of $\eta$, $\delta$ and $Y_d$.

\vspace{7pt}

\emph{Step 2: Completion of the proof.}\quad Now suppose that $B,B' \in C_0(0,1]$ have $\rm{aspect}(B,B') > \g$ and $\calL^1(B_{[0,1]}\cap B'_{[0,1]}) \geq \psi/2$, and let
\[I:= \sqrt{N_d}B_{[0,1]}, \quad J := \sqrt{N_d}B'_{[0,1]}, \quad \hbox{and} \quad K := I\cap J.\]
Then $\calL^1(K) \geq (\psi/2)\sqrt{N_d}$, and the lower bound on $\rm{aspect}(B,B')$ implies that $\calL^1(I),\calL^1(J) \geq \g\calL^1(K)$.  Also set
\[U:= \big\{y \in Y_d\,\big|\ \|\traj_{-N_d}(\s^y) - B\|_\infty < \eta\ \hbox{and}\ \|\traj_{N_d}(\s^y) - B'\|_\infty < \eta\big\} \times X,\]
and suppose that $(y,x) \in U$.

Since
\[U \subseteq \big\{y \in Y_d\,\big|\ \|\traj_{-N_d}(\s^y) - B'\|_\infty < \eta\big\} \times X\]
the choice of $\delta$ (applied with $C := \{(y,x)\}$) implies that the set $U\cap B^{d^{\bfY \ltimes_\s \bfX}_{[-N_d;0)}}_{\delta N_d}(y,x)$ has a Borel partition $\R$ into cells of diameter at most $\delta_1$ according to the pseudometric $\pi_X^\ast d^{\bfX,\infty}_K$ and with
\[|\R| \leq \exp(\eps \calL^1(I\cap J)/2).\]

Next, for each $C \in \R$, since
\[C \subseteq U \subseteq \big\{y \in Y_d\,\big|\ \|\traj_{N_d}(\s^y) - B\|_\infty < \eta\big\} \times X,\]
and since $\delta \leq \delta_1$ and $\eta \leq \eta_1$, the choice of $\delta_1$ and $\eta $ gives that
\[\cov\big(\big(U\cap B^{d^{\bfY \ltimes_\s \bfX}_{[0;N_d)}}_{\delta N_d}(C),\pi_X^\ast d^{\bfX,\infty}_K\big),\delta'\big) < \exp(\eps \calL^1(I\cap J)/2).\]
Since
\[U\cap B^{d^{\bfY \ltimes_\s \bfX}_{[0;N_d)}}_{\delta N_d}\big(U\cap B^{d^{\bfY \ltimes_\s \bfX}_{[-N_d;0)}}_{\delta N_d}(y,x)\big) = \bigcup_{C \in \R}\big(U\cap B^{d^{\bfY \ltimes_\s \bfX}_{[0;N_d)}}_{\delta N_d}(C)\big),\]
these bounds combine to give an overall $\delta'$-covering number of the whole bi-neighbourhood, according to $\pi_X^\ast d^{\bfX,\infty}_K$, of at most $\exp(\eps \calL^1(I\cap J))$.
\end{proof}

\section{Further questions and directions}\label{sec:further-ques}

\subsection{Further understanding of the marginal m.p. spaces}\label{subs:lim-geom}

Another natural approach to Theorem A would seek an enhancement of Kalikow's proof of Theorem~\ref{thm:Kal2} which somehow quantifies the failure of either the Very Weak Bernoulli condition or extremality.

An interesting proposal towards this end has been widely discussed by Thouvenot, often in connection with his Weak Pinsker Conjecture.  For a general shift-invariant process $(A^\bbZ,\mu,S)$ with marginal m.p. spaces $(A^N,d_{\rm{Ham}},\mu_N)$, he suggests considering the smallest number of pairwise-disjoint subsets of $A^N$ that one needs in order that their union carry most of $\mu_N$, and so that the conditional measure of $\mu_N$ on each of them exhibits exponential concentration.  We will not define this more carefully here, but refer to it as the `concentrating-decomposition rate'.

This is an attractive idea in the context of $\rm{RWRS}_\mu$, because the decomposition~(\ref{eq:cond-meass2}) can be associated with the family of graphs
\[\{(y,F_c(y)\,|\ y \in Y_N\} \quad \hbox{for}\ c \in X_N,\]
where the notation is as in Section~\ref{sec:prelim-discuss}. With a little trimming, this decomposition can be turned into a pairwise-disjoint family of subsets of $\{\pm 1\}^N\times C^N$ that carry most of $\rho_N$, and number roughly $\exp (2R\rmh(\mu,S)\sqrt{N})$.  If one could show that this decomposition is, up to order $\exp(\rm{o}(\sqrt{N}))$, among the most efficient ways to break $\rho_N$ into exponentially-concentrated components, then it seems that the scenery entropy $\rmh(\mu,S)$ naturally appears inside this intrinsic geometric invariant of the spaces $(\{\pm 1\}^N\times C^N,d_{\rm{Ham}},\rho_N)$.

Unfortunately, it is not clear that the conditional measures $\rho_{N,c}$ \emph{are} exponentially concentrated.  By definition, we had
\begin{eqnarray}\label{eq:pushfwd}
\rho_{N,c} = (\rm{id},F_c)_\ast \nu_{1/2}^{\otimes N},
\end{eqnarray}
but we have no guarantee that the functions $F_c$ enjoy any `approximate continuity': indeed, it is easy to see that they do not, by slightly modifying Example~\ref{ex:hard-to-recover}.

Thus, there is no reason why the pushforward in~(\ref{eq:pushfwd}) should preserve the exponential concentration property of $\nu_{1/2}^{\otimes N}$, and I do not see any other reason why that property should hold for $\rho_{N,c}$.  It could be that, in order to decompose $\rho_N$ into exponentially-concentrated measures, one needs to decompose each $\rho_{N,c}$ further by conditioning on some additional properties of a random walk path $y$, and I know of no very good estimate on the number of further cells that one would need.  For the above idea, it would be essential that this further partition for each $\rho_{N,c}$ use at most $\exp(\rm{o}(\sqrt{N}))$ cells, so that it does not change the leading-order estimate given by the decomposition according to the graphs of $F_c$.

%

\subsection{Other random walks}\label{subs:other-walks}

Several variants of the RWRS processes do not fall into the class considered by Theorem A.

Perhaps the nearest relatives are those in which the underlying random walk is $p$-stable for some $p \in (1,2)$, so that one has an invariance principle for convergence to a $p$-stable L\'evy process.  In this case, I suspect that the proofs above can be easily adapted to give the following.

\begin{conj}\label{conj:p-stable}
If $(\bfY,\s)$ are the system and cocycle corresponding to a $p$-stable random walk on $\bbZ$ for some $p \in (1,2)$, and if $\bfX$ is a Bernoulli flow, then
\begin{multline*}
\sup_{\k> \k' > 0}\sup_{\delta > 0}\limsup_{N\to\infty}\frac{\log \BIPACK_{\a,\k,\k',\delta}(Y\times X,d_{[-N;0)}^{\bfY \ltimes_\s \bfY},d_{[0;N)}^{\bfY \ltimes_\s \bfX},\nu\otimes \mu)}{N^{1/p}}\\ = \psi_{p\hbox{\scriptsize{-}}\rm{stab}}(\a)\rmh(\bfX),
\end{multline*}
where $\psi_{p\hbox{\scriptsize{-}}\rm{stab}}$ is the obvious analog of $\psi_{\rm{BM}}$ for the $p$-stable L\'evy process.
\end{conj}

On the other hand, the generalization to random walks in $\bbZ^2$, as in~\cite{denHSte97}, is quite different.  The problem there is that a typical pair of trajectories $\s^y_{[0;N)},\s^{y'}_{[0;N)}$ spend only $\rm{o}(N)$ amount of time at locations which are visited by both of them.  I suspect this implies that \emph{no} information is robustly remembered by both the $N$-step past and the $N$-step future, in the sense of the following.

\begin{conj}
Let $\bf{e}_1,\bf{e}_2$ be the usual basis of $\bbZ^2$.  Let $\bfY = (\{\pm \bf{e}_1,\pm \bf{e}_2\}^\bbZ,\nu,S)$ be a Bernoulli shift with $\nu = \nu_{(1/4,1/4,1/4,1/4)}^{\otimes \bbZ}$, let $\s:Y\to \{\pm \bf{e}_1,\pm \bf{e}_2\}$ be the time-zero coordinate, and let $\bfX$ be a finite-entropy Bernoulli $\bbZ^2$-system.  Then
\[\rm{BIPACK}_{\a,\k,\k',\delta}(Y\times X,d^{\bfY \ltimes_\s \bfX}_{[-N;0)},d^{\bfY \ltimes_\s \bfX}_{[0;N)},\nu \otimes \mu) = 1\]
for all sufficiently large $N$, for all $\a > 1$, $\k > \k' > 0$ and $\delta > 0$.
\end{conj}

Similar remarks might apply to a $p$-stable walk if $p < 1$, in which case the occupation measures are no longer absolutely continuous.  This conjecture promises the same behaviour for these systems as for Bernoulli systems (Proposition~\ref{prop:Bern-bicov-trivial}), even though they are among those shown to be non-Bernoulli by den Hollander and Steif in~\cite{denHSte97}, using an adaptation of Kalikow's argument.  It seems that a different invariant is needed to distinguish these examples one from another.

\bibliographystyle{alpha}
\bibliography{bibfile}

\end{document}